\documentclass[english,11pt]{article}

\pdfoutput=1 
\usepackage[utf8]{inputenc}
\usepackage[T1]{fontenc}
\usepackage[a4paper,left=2.5cm,right=2cm,top=2cm,bottom=2.5cm]{geometry}
\usepackage{amsmath, amsthm, amssymb, units, dsfont}

\usepackage[dvipsnames]{xcolor}
\newcommand\myshade{85}
\colorlet{mylinkcolor}{violet}
\colorlet{mycitecolor}{YellowOrange}
\colorlet{myurlcolor}{Aquamarine}

\usepackage[unicode=true,%
            pdftitle={Levy I},%
            pdfauthor={Oliver, Hanno, Thomas, Toni, Marco},%
            colorlinks=true,%
            citecolor=red!\myshade!black,%
            linkcolor=red!\myshade!black,%
            urlcolor=blue!\myshade!black,%
            breaklinks=false]{hyperref}
\usepackage{enumerate}

\usepackage[numbers]{natbib}	
\usepackage{doi}
\usepackage{url}
\usepackage{graphicx}
\usepackage{comment}
\ifpdf
  \DeclareGraphicsExtensions{.eps,.pdf,.png,.jpg}
\else
  \DeclareGraphicsExtensions{.eps}
\fi

\theoremstyle{plain}
	\newtheorem{theorem}{Theorem}[section]
	\newtheorem{corollary}[theorem]{Corollary}
	\newtheorem{lemma}[theorem]{Lemma}
	\newtheorem{proposition}[theorem]{Proposition}
\theoremstyle{remark}
	\newtheorem{remark}[theorem]{Remark}
\theoremstyle{definition}
	\newtheorem{definition}[theorem]{Definition}
	
\newtheorem{assumption}[theorem]{Assumption}

\usepackage{dsfont}  	
\usepackage{newtxtext,newtxmath}
\usepackage{microtype}          
\usepackage{mathrsfs}
\usepackage{amsopn}
\usepackage{cleveref}

\usepackage{mathtools}
\DeclarePairedDelimiter\ceil{\lceil}{\rceil} 
\DeclarePairedDelimiter\floor{\lfloor}{\rfloor} 

\DeclareMathAlphabet{\mathbf}{OT1}{cmr}{bx}{it}

\DeclareMathOperator{\Var}{\boldsymbol{\mathsf {Var}}}
\DeclareMathOperator{\Cov}{\boldsymbol{\mathsf {Cov}}}

\DeclareMathOperator*{\essinf}{ess\,inf}
\DeclareMathOperator*{\esssup}{ess\,sup}
\DeclareMathOperator{\trace}{tr}
\DeclareMathOperator{\supp}{supp}
\DeclareMathOperator{\interior}{int}

\renewcommand{\d}{\mathrm d}
\newcommand{\e}{\mathrm e}

\newcommand{\spF}{\mathscr F}
\newcommand{\spS}{\mathscr S}
\newcommand{\pspace}{(\Omega,\mathfrak{A},\prob)}
\newcommand{\prob}{{\boldsymbol {\mathsf P}}}
\newcommand{\ev}[1]{ {\boldsymbol{\mathsf  E}} \left[  #1 \right]}
\newcommand{\tripleBar}[1]{{\left\vert\kern-0.25ex\left\vert\kern-0.25ex\left\vert #1 
    \right\vert\kern-0.25ex\right\vert\kern-0.25ex\right\vert}}

\newcommand{\ve}{\mathbf e}
\newcommand{\vn}{\mathbf n}

\newcommand{\N}{\mathbb{N}}
\newcommand{\Z}{\mathbb{Z}}
\newcommand{\Q}{\mathbb{Q}}
\newcommand{\R}{\mathbb{R}}
\renewcommand{\C}{\mathbb{C}}

\newcommand{\valpha}{{\boldsymbol{\alpha}}}
\newcommand{\vbeta}{{\boldsymbol{\beta}}}

\newcommand{\oge}[1]{\textcolor{magenta}{#1}}

\begin{document}
	\title{Integrability and Approximability of Solutions to the Stationary Diffusion Equation with Lévy Coefficients}
	\author{Oliver G. Ernst\thanks{Department of Mathematics, TU Chemnitz, Germany}
	\and 
	Hanno Gottschalk\thanks{School of Mathematics and Natural Sciences, University of Wuppertal, Germany}
	\and
	Thomas Kalmes\footnotemark[1]
	\and
	Toni Kowalewitz\footnotemark[1]
	\and
	Marco Reese\footnotemark[2]\textsuperscript{~}
	}
\maketitle

\begin{abstract}
We investigate the stationary diffusion equation with a coefficient given by a (transformed) Lévy random field. 
Lévy random fields are constructed by smoothing Lévy noise fields with kernels from the Matérn class. 
We show that Lévy noise naturally extends Gaussian white noise within Minlos' theory of generalized random fields. 
Results on the distributional path spaces of Lévy noise are derived as well as the amount of smoothing to ensure such distributions become continuous paths. 
Given this, we derive results on the pathwise existence and measurability of solutions to the random boundary value problem (BVP). 
For the solutions of the BVP we prove existence of moments (in the $H^1$-norm) under adequate growth conditions on the Lévy measure of the noise field. 
Finally, a kernel expansion of the smoothed Lévy noise fields is introduced and convergence in $L^n$ ($n\geq 1$) of the solutions associated with the approximate random coefficients is proven with an explicit rate.
\end{abstract}



\section{Introduction}

Random differential equations, i.e., initial and boundary value problems with uncertain data modeled as random variables or stochastic processes, have become an active area of research as a mathematical framework for uncertainty quantification.
The stationary linear diffusion problem 
\begin{equation} \label{eq:random-diff}
    -\nabla \cdot (a \nabla u) = f, \quad u|_{\partial D} = 0, \qquad D \subset \mathbb R^d \text{ bounded,}  
\end{equation}
with diffusion coefficient $a$ given as a random field has been extensively investigated as a model problem for a variety of numerical approximation methods.
Numerous physical phenomena can be modeled by the random diffusion equation, among these Darcy flow in a porous medium with an uncertain spatial variation in hydraulic conductivity.
In this setting the stochastic conductivity model $a = a(x,\omega)$ is typically assumed to follow a lognormal distribution, with the mean and covariance structure of the underlying Gaussian random field $\log a$ estimated using a variety of geostatistical methods.
In each case, the random variation of the input data serves as a mathematical model for the uncertainty associated with these quantities, and the objective is typically to derive the statistical properties of functionals of the solution often refered to as \emph{quantities of interest}.

Early existence and uniqueness results for the random linear elliptic diffusion problem \eqref{eq:random-diff} in an uncertainty quantification (UQ) setting for variational formulations in Bochner spaces modelling the combined deterministic and stochastic variation were presented in 
\cite{
DebEtAl2001,
BabuskaEtAl2003,
BabuskaEtAl2004,
BabuskaEtAl2005,
MatthiesKeese2005,
FrauenfelderEtAl2005%
}
with a focus on numerical methods for their approximate solution.
Subsequent work extended these initial results, formulated for random fields characterized by finite-dimensional parameters, to the infinite-dimensional setting
\cite{
CohenEtAl2010,
CohenEtAl2011,
HoangSchwab2014,
ChkifaEtAl2014,
ChkifaEtAl2015,
NobileEtAl2016,
BachmayrEtAl2017a,
BachmayrEtAl2017b,
ErnstEtAl2018%
}.
In the absence of uniform ellipticity, as in the case of a lognormal diffusion field $a$, variational formulations in the Bochner space setting require more intricate well-posedness analysis
\cite{
GalvisSarkis2009,
Gittelson2010,
MuglerStarkloff2013%
}.

These analysis and approximation methods are based on the assumption that realizations of the diffusion coefficient lie in $L^\infty(D)$ or a subspace of smoother functions.
Moreover, it is assumed that the diffusion field displays positive covariance between distinct locations, as this is typically the case for many modeled phenomena originating in the physical and engineering sciences; such differential equation models featuring correlated data are sometimes distinguished by the term \emph{random} differential equations from the more general designation of \emph{stochastic} differential equations, which may contain rougher stochastic processes.
Indeed, early work on stochastic partial differential equations such as \cite{Walsh1986}
was aimed at generalizing stochastic ODEs driven by rougher processes in the It\^{o}, Skorohod or Stratonovich sense (see \cite{PrevotRockner2007} for a more recent account).
For stationary PDEs, Holden et al.\ \cite{HoldenEtAl1996} considered random diffusion coefficients with values in spaces of distributions, i.e., distribution-valued random variables, interpreting the product $a \nabla u$ as a \emph{Wick product}.
Regularity results on the stationary diffusion equation in Wick sense are given in \cite{BenthTheting2002}.
See  \cite{MatthiesBucher1999,Theting2000} for numerical methods for Wick-random PDEs.

While much of the cited work employs random models based on transformed Gaussian random fields,
there are effects which a Gaussian model cannot capture, particularly discontinuities and heavy-tail behavior, which nonetheless occur in applications such as flow in fractured media, anomalous diffusion and the modeling of heterogeneous materials \cite{torquato2013random,chiu2013stochastic}.
It is thus of interest to consider more general stochastic models for the diffusion coefficient, and in this work we extend the Gaussian model to random fields which follow a Lévy distribution \cite{levy1954theorie,applebaum2009levy,gottschalk2007determine}.

L\'evy random fields have been studied in a number of contexts, including among others stochastic analysis \cite{albeverio1998parabolic}, physics \cite{albeverio96}, statistics \cite{wolpert1998poisson} and simulation \cite{wolpert1998poisson}. For extensions that include interaction between the discrete, discontinuous particle sources of L\'evy fields, see \cite{albeverio05,gottschalk2007determine}.  

In geostatistical applications, Gaussian random fields with Mat\'ern covariance function have been obtained by a stochastic pseudodifferential equation driven by Gaussian noise \cite{Lindgren2011}. 
In this work, we generalise this approach by carefully analyzing the notion of noise as generalized random fields in the sense of Minlos \cite{minlos59,gelfand64}. 
In fact, the resulting class of L\'evy fields coincides with those studied in \cite{albeverio96} in the context of (Euclidean) quantum field theory.  
We also mention recent work on numerical methods for Lévy diffusion fields in \cite{BarthStein2018} .

Our work is most closely related to that of Sarkis et al. \cite{RomanSarkis2006,GalvisSarkis2009,GalvisSarkis2012}, who consider \eqref{eq:random-diff} in the absence of uniform ellipticity and boundedness, allowing for a diffusion coefficient which is a smooth transformation of Gaussian white noise.

The contribution of this work consists of three parts: First, we show that L\'evy random fields naturally extend \cite{Lindgren2011} and \cite{RomanSarkis2006,GalvisSarkis2009,GalvisSarkis2012} by passing from Gaussian white noise to L\'evy noise, see e.g. \cite{gelfand64,albeverio1989representation}. 
To this end, we characterize noise fields as generalized (distribution-valued) random fields ``$Z(x)$'' in the sense of Minlos \cite{minlos59,gelfand64,ito84}, where $Z(f)=$``$\int Z(x) f(x) \d x$'' is only defined as a random variable after ``integrating'' the distribution-valued random variable $Z$ against a test function $f$. 
Here we restrict to fields that can also be ``integrated'' over bounded regions $\Lambda$, i.e., for which $Z(\mathds{1}_\Lambda)=$ ``$\int_\Lambda Z(x)\, \d x$'' can be suitably defined, where $\mathds{1}_\Lambda$ denotes the indicator function of the set $\Lambda$. 
Noise fields can be characterized by the property of independent increments, i.e., $Z(\mathds{1}_{\Lambda_j})$  are independent random variables for mutually disjoint sets $\Lambda_j$. 
Furthermore, we are interested in \emph{stationary} noise fields $Z$ for which, roughly speaking, translations  ``$Z(x-c)$'' follow the same statistical distribution for any shift $c$. 
Under conditions made precise below, we show that all such noise fields are actually L\'evy noise fields. 
In this sense, passing from Gaussian white noise to L\'evy noise naturally extends the approach in \cite{Lindgren2011}. 
We also give sufficient conditions on Mat\'ern smoothing kernels $k(x)$ under which the smoothed noise field has continuous paths. 
This analysis relies on a detailed analysis of Hilbert-Schmidt embeddings of test function spaces to prove that the paths of L\'evy noise fields lie within certain dual spaces. 
Technically, this is based on the harmonic analysis of the Hamiltonian operator of the harmonic oscillator $|x|^2-\Delta$, following It\^o \cite{ito84}. We determine conditions under which such dual spaces are mapped to continuous functions by Mat\'ern smoothing kernels and thus conclude that $Z_k(x)$=``$\int k(x-y) Z(y) \d y$'' has continuous paths. 
By composition with a continuous positive function $T$ of real arguments, we obtain L\'evy models $a(x)=T(Z_k(x))$ of strictly positive random coefficients.
The associated random boundary value problem \eqref{eq:random-diff} can then be solved in a strong sense, i.e., path-wise for almost all paths $a(x)$.  
$T(z)=\exp(z)$ should be seen as the standard choice for the transformation $T(z)$ leading to the well-established log-normal random coefficients, if the noise field is Gaussian white noise. 
Other choices of $T(z)$ may also be of interest, e.g. a (smoothed) step function which, in combination with purely Poisson noise, results in a smoothed version of the penetrating spheres model for random two-phase composite materials, see e.g. \cite{torquato2013random}.     

Our second contribution is a proof of integrability of the random solutions of \eqref{eq:random-diff}. 
In place of the variational approach in both deterministic and random variables, we base our investigations on \emph{a priori estimates} of elliptic partial differential equations.
These crucially depend on the minimal value of the coefficient $a(x)$ on the domain $D$. 
Given that $T(z)$ assumes finite minimal values on bounded intervals, the minimal value problem for $T(Z_k(x))$ turns into an extremal value problem for $|Z_k(x)|$ on the domain $D$. 
We thus have to control the tails of the distribution of $\sup_{x\in D}|Z_k(x)|$. 
This set of problems has been intensively studied in the context of \emph{empirical processes} by exploiting metric entropy estimates and concentration phenomena \cite{gine2016mathematical,talagrand2014upper,vanhandel2014probability}. Here we follow this approach for the Gaussian part of the L\'evy field,  relying on a metric entropy estimate by Talagrand \cite{Tal94}. 
For its Poisson part, standard metric entropy estimates are not available (see however \cite[Chapter 11]{talagrand2014upper} for some results for a different class of L\'evy processes), as these are based on the theory of sub-Gaussian processes, and the Poisson part of a L\'evy field is non sub-Gaussian. 
We instead develop Chernov-like bounds for the non-Gaussian part under the assumption that the L\'evy measure defining the Poisson contributions has a Laplace transform. 
Instead of developing a general theory based on chaining-like arguments, we exploit the explicit representation of the Poisson part as an infinite sum of smoothed point processes. 
Combining both estimates for the suprema of Gaussian and Poisson parts with the a priori estimate, we prove that $u\in L^n(\pspace;H^1(D))$, where $n\geq1$, $\pspace$ is the underlying probability space and $H^1(D)$ is the Sobolev space of weakly differentiable functions.     

Our paper's third contribution consists of a suitable adaptation of the Karhuhnen-Lo\`eve (KL) expansion for smoothed L\'evy noise fields. As L\'evy fields, unlike the Gaussian case, are not determined by their covariance function, we  expand the smoothing kernel function $k(x-y)$ instead of the covariance function. 
This kernel, by translation invariance, always has a continuous spectrum as an integral operator on $L^2(\R^d,\d x)$. 
Restriction of the field to the domain $D$, however, only constrains $x$, and we therefore have to `cut off' the noise field ``$Z(y)$'' and restrict it to some larger domain $\Lambda$ containing $D$. 
With both variables in $k(x-y)$ restricted to $\Lambda$, Mercer's theorem (e.g.\ \cite[Theorem~1.80]{LordEtAl2014}) provides us with an expansion for $k(x-y)$ as an integral kernel on $L^2(\Lambda,\d x)$.  
If $k(x)$ is of Matérn type, the effect of the cut-off outside $\Lambda$ vanishes exponentially in the distance between $D$ and the boundary of $\Lambda$. 
The expansion of the kernel function restricted to $\Lambda$ then gives $k(x-y)=\sum_{i=1}^\infty \lambda_i e_i(x)e_i(y)$ with eigenfunctions $e_i(x)$ and eigenvalues $\lambda_i$ that depend on $\Lambda$. 
Rate estimates for the uniform convergence of this series are obtained via a thorough spectral analysis of the integral operator defined by $k(x-y)$ using circular embeddings and refining techniques from \cite{BachmayrEtAl2018}. 
Truncating the expansion at $N$, we obtain an approximation of $Z_k(x)$ by $Z_{k,N}(x)=\sum_{j=1}^N\lambda_j e_j(x)Z(e_j)$ which thus only depends on the finite-dimensional L\'evy distribution of the random vector $(Z(e_1),\ldots,Z(e_N))$. 
We obtain approximate solutions $u_N$ to the solution $u$ of \eqref{eq:random-diff} with $a(x)=T(Z_k(x))$ by replacing $Z_{k}(x)$ by $Z_{k,N}(x)$.
We then prove, under suitable conditions, convergence $u_N\to u$ in $L^n(\pspace;H^1(D))$. 
We also provide convergence rates for the combined  exponential decay in the cut-off to $\Lambda$ and the truncation of the series via a suitable choice of $\Lambda(N)$.

Our work thus provides a detailed exposition of Lévy generalized random fields and how they may be smoothed to yield random diffusion coefficients leading to well-posed problems for the random PDE. 
In addition, we show how the ususal Karhunen-Loève expansion approach can be modified to yield a convergent approximation sequence obtained from truncations of the smoothed Levy field. 

The paper is organised as follows: In Section \ref{sec:random_fields} we recall the theory of generalized random fields and state results on path properties in dual spaces to certain spaces of test functions. 
We also determine the amount of smoothing with Matérn kernels needed to map such distributions to continuous functions and thus prove continuity of the paths for smoothed generalized random fields.
In Section \ref{sec:LevyRandomField} we introduce L\'evy noise and present a classification theorem which states that noise fields with certain continuity properties allowing for 'integration over a region' and possessing finite expectation are L\'evy noise fields. 
We also recall the well-known decomposition of L\'evy noise into deterministic, Gaussian and Poisson contributions and derive the representation of Poisson noise as an infinite sum over compound Poisson point processes.  
Section \ref{section:existence_of_moments} establishes the existence of solutions to the random PDE \eqref{eq:random-diff} and their integrability properties. 
In Section \ref{section:approx} we present the convergence of solutions obtained from finite-dimensional approximations of smoothed L\'evy random fields and establish convergence rates. 
In the final Section \ref{sec:outlook} we briefly comment on possible future research directions. 
A number of technical results are presented in Appendices \ref{appendix b}--\ref{app:conv_rates}.

\section{Smoothing of Generalized Random Fields} \label{sec:random_fields}

In this section, we define L\'evy noise fields as a generalization of Gaussian random fields. 
In contrast to the latter, realizations of random fields which follow a L\'evy distribution cannot be represented as functions with values defined pointwise, and so more general mathematical concepts are needed. 
This is reflected in the term \emph{noise}, which besides connoting the perturbation of a signal also refers to the lack of spatial correlation of such random fields, a feature already exhibited by Gaussian white noise. 
Such a mathematical framework is provided by the theory of \emph{generalized random fields}, the definition and basic properties of which we recall for the reader's convenience below.

\subsection{Generalized Random Fields} \label{sec:minlos}

Rather than by points in a subset of $\mathbb R^d$, \emph{generalized} random fields are families of random variables indexed by elements of an abstract vector space $V$, 
which we will take to be a locally convex space over the real numbers.
Specifically, $V$ is a topological vector space possessing a base of convex zero-neighborhoods. 
The assumption of local convexity ensures the existence of a nontrivial dual space.

For a probability space $\pspace$ we denote by $L^0\pspace$ the vector space of Borel measurable random variables.
As usual, we do not distinguish notationally between a random variable and its equivalence class resulting from almost sure (a.s.) equality. 
Moreover, we set
\[
    \|X\|_{L^0} 
    := 
    \ev{ |X| \wedge 1 }
    =
    \int_\Omega (|X|\wedge 1) \;\d \prob,
    \qquad
    X \in L^0\pspace,
\]
where $a \wedge b := \min\{a,b\}$ for $a,b \in \mathbb R$,
and
\[
   d_0(X,Y) := \|X-Y\|_{L^0}, \qquad
   X, Y \in L^0\pspace.
\]
It is easily seen that $d_0$ is a (translation-invariant) metric on $L^0\pspace$ making $L^0\pspace$ a Hausdorff topological vector space. 
Moreover, since for any $\varepsilon \in (0,1)$ and $X \in L^0\pspace$ we have 
\[
   \varepsilon \, \prob(|X|>\varepsilon)
   \leq 
   \|X\|_{L^0}
   \leq 
   \prob(|X|>\varepsilon) + \varepsilon,
\]
it follows that convergence with respect to the metric $d_0$ coincides with convergence in probability. 
It is well known that the metric space $(L^0\pspace, d_0)$ is complete (see e.g.\ \cite[Lemma 3.6]{Kallenberg}).

\begin{definition}[Generalized Random Field] \label{def:gen_rf}
A \emph{generalized random field} $Z$ indexed by a locally convex topological vector space $V$ is a collection of real-valued random variables $\{Z(f)\}_{f\in V}$ on a common probability space $\pspace$ such that the following conditions hold:
\begin{itemize}
\item[(i)] 
Linearity: $Z(\alpha f + \beta g) = \alpha Z(f) + \beta Z(g)$ 
a.s.\ for all  $f,g \in V$ and $\alpha,\beta\in \mathbb R$.
\item[(ii)]
Stochastic continuity: $f\to f_0 \text{ in } V$ implies $Z(f) \to Z(f_0)$ in probability.
\end{itemize}
Thus, a generalized random field on $\pspace$ indexed by $V$ is a continuous linear mapping $Z:V\rightarrow L^0\pspace$, where $L^0\pspace$ is endowed with the metric $d_0$.
	
We call two generalized random fields $Z$ and $\tilde{Z}$ on probability spaces 
$\pspace$ and $(\tilde{\Omega},\tilde{\mathfrak{A}}, \tilde{\prob})$ indexed by $V$ \emph{equivalent (in law)} if their finite-dimensional distributions coincide, i.e., if
\begin{equation*}
	\prob\bigl(Z(f_1)\in A_1 \wedge \dots \wedge Z(f_n) \in A_n\bigr) 
	= 
	\tilde{\prob}
	\bigl(\tilde{Z}(f_1)\in A_1 \wedge \dots \wedge \tilde{Z}(f_n) \in A_n\bigr) 
\end{equation*}
holds for all $n\in \mathbb{N}, f_1, \ldots, f_n \in V \text{ and } A_1, \ldots, A_n\in \mathfrak{B}(\mathbb{R})$, where $\mathfrak{B}(\mathbb{R})$ denotes the Borel $\sigma$-algebra on $\mathbb R$.
\end{definition}

\begin{remark} \label{rem:gen_rf}
\leavevmode
\begin{itemize}
\item[(i)] 
For the (topological) dual $V'$ of a metrizable locally convex vector space $V$ and measurable $X:\pspace\rightarrow (V',\mathfrak{B})$, with $\mathfrak{B}$ a $\sigma$-algebra on $V'$ for which the evaluation maps $V'\rightarrow\mathbb R, u\mapsto u(f), f\in V,$ are measurable, an application of Lebesgue's dominated convergence theorem shows that
\[
   Z:V\rightarrow L^0\pspace, 
   \quad 
   f\mapsto (\omega\mapsto X(\omega)(f)) =: X(f) =: Z(f,\omega)
\]
is a generalized random field. 
In other words, in this setting $V'$-valued random variables are generalized random fields.
However, for a general (metrizable) locally convex space $V$ it is not true that every generalized random field $Z$ indexed by $V$ can be realized (up to equivalence) in the above way by a $V'$-valued random variable. 
This does, however, hold for \emph{nuclear} locally convex spaces $V$ by Minlos' Theorem, see below.
\item[(ii)] 
Let $\tilde{V}$ be a locally convex space and let the subspace $V\subseteq \tilde{V}$ be dense. 
If $Z:V\rightarrow L^0\pspace$ is a generalized random field, then, due to the fact that $L^0\pspace$ endowed with the topology of convergence in probability is a complete Hausdorff space, it follows that there is a unique continuous linear extension $\tilde{Z}:\tilde{V}\rightarrow L^0\pspace$ of $Z$ (see e.g.\ \cite[Lemma 22.19]{MeVo}). 
In particular, every generalized random field $Z$ on a locally convex space $V$ can be uniquely extended to a generalized random field on the completion of $V$.
\end{itemize}
\end{remark}


In finite dimensions, Bochner's theorem \cite[Theorem 1.4.3]{Rudin1962} establishes that distributions of $\mathbb{R}^d$-valued random variables are in one--to--one correspondence with continuous, positive definite functions $\varphi : \mathbb{R}^d\to \mathbb C$ with $\varphi(0) = 1$ by way of the Fourier transform as $\ev{\e^{iX(f)}}= \varphi(f)$ with $f\in \mathbb{R}^d$ and $X(f) = X\cdot f$ denoting the Euclidean inner product on $\mathbb{R}^d$. 
Although, as mentioned above, not every generalized random field indexed by a locally convex space $V$ can be represented by a $V'$-valued random variable, the one--to--one correspondence between generalized random fields indexed by $V$ (up to equivalence in law) and characteristic functionals on $V$ remains valid in this general setting.

\begin{definition}[Characteristic Functional]
A \emph{characteristic functional} on a locally convex space $V$ is a mapping $\varphi : V \to \mathbb{C}$ with the following properties
\begin{enumerate}[(i)]
\item  
$\varphi(0)=1$,
\item
$\varphi$ is continuous,
\item
$\varphi$ is positive definite, i.e., the matrix $[\varphi(f_i-f_j)]_{i,j=1}^n$ is Hermitian and positive semidefinite for all $n \in \mathbb N$ and $f_1, \ldots, f_n\in V$.
\end{enumerate}
\end{definition}

The following basic result on generalized random fields can be found e.g.\ in \cite[Theorem 2.4.5]{ito84}.
Note that positive definite functions on a locally convex space are continuous if and only if they are continuous at $0$, which holds if and only if they are uniformly continuous, i.e., if for every $\varepsilon>0$  there is a continuous seminorm $p$ on $V$ such that $|\varphi(f)-\varphi(g)|<\varepsilon$ whenever $p(f-g)<1$ \cite{BergForst}.

\begin{theorem}
\label{thm:correspondence characteristic functional and generalized random field}
Let $V$ be a locally convex space and $\varphi : V \to \mathbb{C}$ a characteristic functional. 
Then there exists a generalized random field $Z$ indexed by $V$ which is unique (up to equivalence in law) and satisfies $\varphi(f) = \ev{\e^{iZ(f)}}, f\in V$.
Conversely, for any generalized random field $Z$ indexed by $V$, its Fourier transform $\varphi(f) := \ev{\e^{iZ(f)}}, f\in V$, is a characteristic functional.
\end{theorem}



As mentioned in Remark \ref{rem:gen_rf} (i), it is not always possible to represent a generalized random field $Z$ indexed by a locally convex space $V$ as a $V'$-valued random variable. 
A sufficent condition for this to hold is that the characteristic functional of $Z$ be continuous not only with respect to the topology given on $V$, but also in the Sazonov topology of $V$, which is the strongest of all multi-Hilbertian topologies which are Hilbert--Schmidt-weaker than the topology of $V$. 
For a precise definition of the Sazonov topology we refer to \cite{dalecky91} and \cite{ito84} (where it is referred to as the Kolmogorov-I-topology).

In general, the Sazonov topology on $V$ is strictly weaker than the original topology of $V$. 
However, a notable exception to this is the case when the locally convex space $V$ is nuclear. 
Recall that a locally convex space $V$ is nuclear if there is a directed family $\mathscr{P}$ of continuous Hilbert seminorms  on $V$ that generate its topology such that for every $p\in\mathscr{P}$ there exists $q\in\mathscr{P}$ with $p\leq q$ and such that the so-called canonical linking map $i_{q}^{p}:V_q\rightarrow V_p$, i.e., the extension of the inclusion from the pre-Hilbert space $(V/\{f\in V;\,q(f)=0\}, q)$ into the pre-Hilbert space $(V/\{f\in V;\,p(f)=0\}, p)$ to their respective completions $V_q$ and $V_p$, is a Hilbert--Schmidt operator.

In order to formulate the version of Minlos' Theorem which will be crucial for our considerations, we introduce the following notation. 
For a continuous seminorm $p$ on a locally convex space $V$ we denote as above by $V_p$ the local Banach space corresponding to $p$, i.e., the completion of the quotient $V/\{f\in V : \,p(f)=0\}$ equipped with the quotient norm associated with $p$. 
By abuse of notation we denote the quotient norm as well as the norm on $V_p$ again by $p$. 
Then the dual space $V_p'$ of $V_p$ can be identified in a canonical way with the subspace $\{ \omega\in V' : \,\exists\,C>0\,\forall\,f\in V:\,|\omega(f)|\leq Cp(f)\}$ of $V'$. 
Finally, we denote the Borel $\sigma$-algebra on $V'$ generated by the weak*--topology $\sigma(V',V)$ by $\mathfrak{B}(V')$. 
Due to the Banach-Alaoglu-Bourbaki Theorem, for every continuous seminorm $p$ on $V$ and every $n\in\mathbb N$ the set 
$\{\omega\in V' :  |\omega(f)|\leq n p(f) \; \forall \, f\in V\}$ is $\sigma(V',V)$-compact which implies $V_p'\in\mathfrak{B}(V')$. 
For the following version of Minlos' Theorem, see \cite[Proof of Theorem III.1.1]{dalecky91} combined with \cite[Theorem I.3.4]{VaTaCh}.

\begin{theorem}[Minlos]\label{thm:minlos}
Let $V$ be a nuclear space and $V'$ its topological dual. 
For a functional $\varphi : V\rightarrow \mathbb C$ the following are equivalent:
\begin{enumerate}[(i)]
\item
$\varphi$ is a characteristic functional.
\item
There is a probability measure $\mu$ on $(V',\mathfrak{B}(V'))$ such that its Fourier transform $\hat{\mu}$ coincides with $\varphi$, where
\begin{equation} \label{eq:FT-measure}
     \hat{\mu}(f) := \int_{V'} \e^{i\omega(f)} \; \mu(\d \omega),
    \qquad
    f\in V.
\end{equation}
\end{enumerate}
Moreover, for a characteristic functional $\varphi$ the probability measure $\mu$ in \oge{\eqref{eq:FT-measure}} is uniquely determined.

Additionally, if for a characteristic functional $\varphi$ on a nuclear space $V$ there is a continuous Hilbert seminorm $p$ on $V$ such that $\varphi$ is continuous with respect to $p$, then  for the corresponding unique probability measure $\mu$ on $(V',\mathfrak{B}(V'))$ we have that $\mu(V_q')=1$ for every continuous Hilbert seminorm $q$ on $V$ for which the canonical linking map $i_q^p:V_q\rightarrow V_p$ is Hilbert-Schmidt.
\end{theorem}

\begin{remark}\label{rem:minlos}
\leavevmode
\begin{enumerate}[(i)]
\item
For an arbitrary locally convex space $V$ and any probability measure $\mu$ on $(V',\mathfrak{B}(V'))$, the mapping
\[
   (V',\mathfrak{B}(V'),\mu)\rightarrow\mathbb R, 
   \qquad 
   \omega\mapsto \omega(f)
\]
defines a (scalar) random variable for each $f\in V$. 
Therefore, the mapping
\[
   Z : V\rightarrow L^0(V',\mathfrak{B}(V'),\mu),
   \qquad
   f \mapsto (\omega\mapsto\omega(f))
\]
defines a generalized random field indexed by $V$ which is called the \emph{canonical process} associated with $\mu$. 
It should be noted that the canonical process satisfies a stronger continuity property than an arbitrary generalized random field since $(Z(f_\iota))_{\iota\in I}$ converges also pointwise on $V'$ (in particular $\mu$-almost everywhere) to $Z(f)$ whenever $(f_{\iota})_{\iota\in I}$ is a net converging to $f$ in $V$.
\item
Let $\varphi$ be a characteristic functional on the nuclear space $V$ which is continuous with respect to the continuous Hilbert seminorm $p$ and let $\mu$ be the corresponding probability measure on $(V',\mathfrak{B}(V'))$. 
Moreover, let $q\geq p$ be a continuous Hilbert seminorm on $V$ such that the canonical linking map $i_q^p$ is Hilbert-Schmidt. 
It is straightforward to show that the trace $\sigma$-algebra $\mathfrak{B}(V')\cap V_q'$ coincides with $\mathfrak{B}(V_q')$, the Borel $\sigma$-algebra of $V_q'$ generated by the weak*--topology $\sigma(V_q',V_q)$. 
Thus, for the canonical process
\[
   Z : V_q\rightarrow L^0(V_q',\mathfrak{B}(V_q'),\mu_{|V_q'})
\]
associated with the restriction $\mu_{|V_q'}$ it holds that whenever $D\subseteq\mathbb R^d$ is open and a mapping 
\[
   D\rightarrow V_q,
   \quad
   x\mapsto f_x
\]
is continuous, $(Z(f_x))_{x\in D}$ is a random field indexed by $D$ which has almost surely continuous paths. 
The characteristic function of the random variable $Z(f_x)$ is given by $\varphi(f_x)$ for each $x\in D$; note that the characteristic functional $\varphi$ is uniformly continuous with respect to $p$ by assumption and thus can be extended in a unique way to a uniformly continuous functional on $V_p\supseteq V_q$. 
\end{enumerate}
\end{remark}

We follow the approach outlined in Remark \ref{rem:minlos} (ii) below for the space $\spS(\R^d)$ of Schwartz functions on $\R^d$ as the index space of generalized random fields $Z$ whose characteristic functionals are continuous with respect to a specific norm. 
These random fields $Z$ will then be convolved with Mat\'ern kernels to yield random continuous functions on $\R^d$ with known pointwise distributions.

\subsection{Generalized Random Fields Indexed by \texorpdfstring{$\mathscr{S}$}{Lg} and Their Convolution With Mat\'ern Kernels} \label{sub:Schwartz} 

We now set the stage for a more precise characterization of the path properties of L\'evy noise fields. 
We denote by $\mathscr{S} = \mathscr{S}(\R^d)$ the space of (real-valued) rapidly decreasing smooth functions on $\R^d$ endowed with its standard topology so that $\mathscr{S}$ is a separable nuclear Fr\'echet space, see e.g.\ \cite{MeVo}. 
Clearly, $\spS$ is a subspace of $L^2(\R^d)$.

In view of Minlos' Theorem~\ref{thm:minlos}, it will be important for us to know when linking maps between local Hilbert spaces of $\spS$ are Hilbert-Schmidt. 
To facilitate this determination, we introduce a  sequence of (semi-)norms on $\spS$ which generate the same locally convex topology (cf.\ 
\cite[Section I.1.3]{ito84}, 
\cite[Example 29.5 (2)]{MeVo}, or 
\cite[Appendix to Section V.3]{ReSi} for the case $d=1$) 
but for which this property can be easily verified. 
For $k\in\N_0$ we denote by $h_k$ the $k$-th Hermite function on $\R$, defined as
\[
   h_k(x)
   :=
   (2^k k!\sqrt{\pi})^{-1/2}
   (-1)^k \e^{x^2/2} \biggl(\frac{\d}{\d x}\biggr)^k \e^{-x^2},
   \quad
   x \in \mathbb R,
\]
and for $\valpha\in\N_0^d$ we denote by $h_\valpha:=h_{\alpha_1}\otimes\dots\otimes h_{\alpha_d}$ the tensorized Hermite function $h_\valpha(x):=\prod_{j=1}^d h_{\alpha_j}(x_j)$ on $\R^d$. 
As is well known, the $(h_\valpha)_{\valpha\in\N_0^d}$ form an orthonormal basis of $L^2(\R^d)$. 
Denoting the inner product on $L^2(\R^d)$ by $(\cdot,\cdot)$, we observe that for every $p\in\R$ the set
\[
    \spS_p 
    :=
    \biggl\{
    f\in L^2(\R^d) : 
    |f|_p^2:=\sum_{\valpha\in\N_0^d}(2|\valpha|+d)^{2p}|(f,h_\valpha)|^2<\infty
    \biggr\}
\]
is a subspace of $L^2(\R^d)$ containing $\spS$, $|\cdot|_p$ is a norm on $\spS_p$ with associated inner product
\[
   (f,g)_p
   :=
   \sum_{\valpha\in\N_0^d} (2|\valpha|+d)^{2p} (f,h_\valpha) (h_\valpha,g),
   \quad
   f,g\in \spS_p,
\]
and $\spS_q\subseteq \spS_p$ with continuous (even contractive) inclusion for every $p\leq q$. 
Furthermore, we have that 
\[
    \spS
    = 
    \bigcap_{p\in\R} \spS_p
    =
    \bigcap_{p\geq 0} \spS_p
    =
    \bigcap_{p\in\N_0} \spS_p
\] 
and $(|\cdot|_p)_{p\in\R}$ is an increasing family of norms on $\spS$ which generates the standard topology on $\mathscr{S}$.

We use the above set of seminorms to construct Hilbert--Schmidt embeddings:
\begin{proposition} \label{nuclear linking map}
For each $p\in\R$ and $\ell>\frac{d}{2}$ the linking map
\[
    i_{p+\ell}^p:(\spS_{p+\ell},|\cdot|_{p+\ell})\rightarrow(\spS_p,|\cdot|_p)
\]
from the local Hilbert space $\spS_{p+\ell}$ to the local Hilbert space $\spS_p$ is Hilbert-Schmidt.
\end{proposition}
\begin{proof}
Defining
\[
   h_{p,\valpha} := \prod_{j=1}^d(2\alpha_j+1)^{-p} h_{\alpha_j},
   \qquad
   p\in\R, \quad \valpha\in\N_0^d,
\]
it follows that for fixed $p\in\R$ the family $(h_{p,\valpha})_{\alpha\in\N_0^d}$ is an orthonormal basis of the pre-Hilbert space $(\spS,|\cdot|_p)$ whose completion we denote by $(\spS_p,|\cdot|_p)$. 
Because for $\vbeta\in\N_0^d, \ell\geq 0$ we have
\[
   |h_{p+\ell,\vbeta}|_p^2
   =
   \sum_{\valpha\in\N_0^d} (2|\valpha|+d)^{2p}| (h_{p+\ell,\vbeta}, h_\valpha)|^2
   =
   \frac{(2|\vbeta|+d)^{2p}}{\prod_{j=1}^d(2\beta_j+1)^{2(p+\ell)}}
\]
and because there is $1\leq j\leq d$ with $\beta_j\geq |\vbeta|/d$, it follows that
\[
   |h_{p+\ell,\vbeta}|_p^2
   \leq
   \frac{(2|\vbeta|+d)^{2p}}{(2\frac{|\vbeta|}{d}+1)^{2(p+\ell)}}
   =
   \frac{d^{2(p+\ell)}}{(2|\vbeta|+d)^{2\ell}}.
\]
Noting that for given $k\in\N_0$ the number of $\vbeta\in\N_0^d$ for which $|\vbeta|=k$ is equal to $\binom{k+d-1}{k}$, we conclude
\begin{align*}
   \sum_{\vbeta\in\N_0^d}|h_{p+\ell,\beta}|_p^2
   &\leq 
   d^{2(p+\ell)}\sum_{k=0}^\infty \binom{k+d-1}{k}\frac{1}{(2k+d)^{2\ell}}\\
   &=
   \frac{d^{2(p+\ell)}}{(d-1)!} 
      \sum_{k=0}^\infty \frac{(k+d-1)!}{k!}\frac{1}{(2k+d)^{2\ell}}\\
   &\leq
   \frac{d^{2(p+\ell)}}{(d-1)!} \sum_{k=0}^\infty (2k+d)^{d-1-2\ell}
\end{align*}
which proves the assertion.
\end{proof}
Clearly, for $f \in \spS$
\[
   \tripleBar{f}
   :=
   \big(\|f\|_{L^1(\R^d)}^2+\|f\|_{L^2(\R^d)}^2\big)^{1/2}
   =
   \big(\|f\|^2_{L^1(\R^d)}+|f|_0^2\big)^{1/2}
\]
defines a norm on $\spS$ which is continuous in its standard topology defined 	by the increasing family of norms $|\cdot|_p, p \in \mathbb R$, defined just before Proposition~\ref{nuclear linking map}. 
For $m\in\N, m > d/2,$ we set
\[
   c_m := \int_{\R^d} \frac{\d x}{(1+|x|^2)^m} < \infty.
\]
For another suitable constant $C_m$ we conclude using H\"older's inequality that, for all $f\in\spS$,
\begin{equation}\label{mixed L^1-L^2-norm estimate}
\begin{split}
   \tripleBar{f}^2
   &\leq 
   c_m\int_{\R^d}(1+|x|^2)^m|f(x)|^2\,\d x + \int_{\R^d}|f(x)|^2\,\d x \\
   &\leq 
   (1+c_m)2^{m-1} \int_{\R^d}(1+|x|^{2m})|f(x)|^2\,\d x  \\
   &\leq 
   (1+c_m)(2d)^{m-1}\int_{\R^d} \biggl(1+\sum_{j=1}^d x_j^{2m}\biggr)|f(x)|^2\,\d x\\
   &=
   (1+c_m)(2d)^{m-1}\biggl(|f|_0^2+\sum_{j=1}^d|x_j^m f|_0^2\biggr) \\
   &\leq
   (1+c_m)(2d)^mC_m|f|_{\frac{m}{2}}^2.
\end{split}
\end{equation} 
In the last step we have used the estimate that for each $m \in \mathbb N$ there exists $C_m > 0$ such that 
\[
   |x_j^m f|_0^2 \le C_m |f|_{\frac{m}{2}}^2,
   \quad\text{ for all } 1 \le j \le d \text{ and for all } f \in \spS,
\]
which follows easily by induction from the well-known three-term recurrence relation 
\[
   x_j h_\valpha(x) 
   = 
   \sqrt{\frac{\alpha_j}{2}} h_{\valpha - \ve_j}(x)
   +
   \sqrt{\frac{\alpha_j+1}{2}} h_{\valpha + \ve_j}(x),
   \qquad 1 \le j \le d, \; \valpha \in \mathbb N_0^d, \; x \in \mathbb R^d
\]
satisfied by the Hermite functions. 
Here $\ve_j=(\delta_{\ell,j})_{1\leq \ell\leq d}$ denotes the $j$-th unit coordinate vector in $\R^d$. 
Combining the above considerations we can now easily prove the following theorem.
\begin{theorem}\label{preparation for Levy}
Let $\varphi:\spS\rightarrow\C$ be a positive definite functional which is continuous with respect to the norm $\tripleBar{\cdot}$ and which satisfies $\varphi(0)=1$.
Then there is a unique probability measure $\mu$ on $(\spS',\mathfrak{B}(\spS'))$ such that $\hat{\mu} = \varphi$. 
Moreover, $\mu(\spS'_q)=1$ if $q>\frac{3d}{4}$.
\end{theorem}
\begin{proof}
It follows from inequality (\ref{mixed L^1-L^2-norm estimate}) and the continuity of $\varphi$ that $\varphi$ is continuous with respect to $|\cdot|_{\frac{m}{2}}$ whenever $m>d/2$. 
Because $i^{m/2}_{m/2+\ell}$ is Hilbert-Schmidt for every $\ell > d/2$ by Proposition \ref{nuclear linking map}, the assertion follows from Minlos' Theorem~\ref{thm:minlos}.
\end{proof}

For a tempered distribution $\omega\in\spS'$ and a rapidly decreasing function $f\in\spS$ the convolution
\[
   \omega*f:\R^d\rightarrow\R, 
   \quad 
   y\mapsto \langle\omega, \tau_y (f^\vee)\rangle
   = 
   \langle\omega_x,f(y-x)\rangle
\]
is a smooth function, where as usual we denote by $u(g)=\langle u,g\rangle$ the application of $u\in\spS'$ to $g\in\spS$.
In addition, $(\tau_y g)(x):=g(x-y)$ denotes the translation of $g$ by $y \in \R^d$, 
$g^\vee (x):=g(-x)$ the reflection of $g$ at the origin, 
and the subscript $\omega_x$ indicates that the tempered distribution $\omega$ acts on test functions depending on the variable $x$.
	
Similarly, for $q\in\N_0$ and $\omega\in\spS_q'$, whenever $f\in\spS_q$ is a function such that $\tau_y(f^\vee)\in\spS_q$ for every $y\in\R^d$, the convolution
\[
   \R^d\rightarrow\R, 
   \quad
   y\mapsto \langle \omega_x, \tau_y (f^\vee)\rangle
\]
is defined and is obviously continuous whenever the mapping
\[
    \R^d\rightarrow \mathscr{S}_q(\R^d),
    \quad
    y\mapsto \tau_y (f^\vee)
\]
is continuous. 
Therefore, whenever $\varphi$ is a $\tripleBar{\cdot}$-continuous characteristic functional on $\spS$ with associated probability measure $\mu$ on $(\spS',\mathfrak{B}(\spS'))$, it follows from Theorem \ref{preparation for Levy} that for $\mu$-almost all $\omega$  the convolution $\omega*f$ is a well-defined function on $\R^d$ for each $f\in \spS_q, q>\frac{3d}{4},$ for which $\tau_y(f^\vee)\in\spS_q$ for every $y\in\R^d$, and this convolution yields a continuous function on $\R^d$ whenever
\[
    \R^d\rightarrow \mathscr{S}_q,
    \quad
    y\mapsto \tau_y (f^\vee)
\]
is continuous. 
We are particularly interested in the case when $f$ is a Mat\'ern kernel.

\begin{definition}\label{def:matern}
For $\alpha\in\R$ and $m>0$ we introduce the function
\[
   \hat k_{\alpha,m} :  \R^d \to \R,
   \qquad
   \xi \mapsto  \frac{1}{(|\xi|^2+m^2)^\alpha}
\]
and define the \emph{Mat\'ern kernel (with parameters $\alpha$ and $m$)} as the inverse Fourier transform 
\[
    k_{\alpha,m} := \spF^{-1} (\hat k_{\alpha,m}).
\]
Note that $\hat k_{\alpha,m}$ is a polynomially bounded smooth function and thus belongs to $\spS'$, hence its inverse Fourier transform is well-defined.
\end{definition}
The proof of the following lemma is somewhat technical and can be found in Appendix~\ref{appendix b}.

\begin{lemma}\label{lem:matern}
For $q\in\N_0$, $\alpha\in\R$, and $m>0$ in the following statements, (i) implies (ii), (ii) implies (iii), and (iii) implies (iv).
\begin{enumerate}[(i)]
\item 
$\alpha>\frac{d}{4}+q+\max\bigl\{0,\frac{q-3}{2}\bigr\}$.
\item
For every $y \in \R^d$ the translation $\tau_y(k_{\alpha,m}^\vee)$ lies in $\spS_q$ and the mapping
\[
    \R^d\rightarrow(\spS_q, |\cdot|_q), \quad
    y\mapsto \tau_y(k_{\alpha,m}^\vee)
\]
is  continuous.
\item
For every $y \in \R^d$ the translation $\tau_y(k_{\alpha,m}^\vee)$ lies in $\spS_q$.
\item
$\alpha>\frac{d}{4}+q$.
\end{enumerate}
In particular, if $q\in\{0,1,2,3\}$ then (ii), (iii), and (iv) above are equivalent.
\end{lemma}

We can now state sufficient conditions on the amount of smoothing required in order that a random field with a $\tripleBar{\cdot}$-continuous characteristic functional have continuous realizations after smoothing by convolution with $k_{\alpha,m}$.

\begin{theorem}\label{thm:matern}
Let $\varphi$ be a positive definite $\tripleBar{\cdot}$-continuous functional on $\spS=\spS(\R^d)$ with $\varphi(0)=1$. 
Then there is a unique probability measure $\mu$ on $(\spS',\mathfrak{B}(\spS'))$ satisfying $\hat{\mu} = \varphi$ and such that for all $\alpha>d+\max\{0,\frac{3d-12}{8}\}$, every $m>0$, the function
\begin{equation} \label{smoothed-field}
   \R^d\rightarrow\R, \quad
   y\mapsto \omega * k_{\alpha,m}(y)
   =
   \bigl\langle
     \omega,\tau_y(k_{\alpha,m}^\vee)
   \bigr\rangle
\end{equation}
is defined and continuous for $\mu$-almost all $\omega \in \spS'$. 
Moreover, for fixed $y\in\R^d$ the distribution of the random variable
\[
   (\spS',\mathcal{B}(\spS'),\mu)\rightarrow (\R,\mathfrak{B}(\R)),
   \quad
   \omega\mapsto \omega * k_{\alpha,m}(y)
\]
has the Fourier transform $\varphi\big(\tau_y(k_{\alpha,m}^\vee)\big)$.
\end{theorem}
\begin{proof}
By Theorem \ref{preparation for Levy} there is a unique probability measure $\mu$ on $(\spS',\mathfrak{B}(\spS'))$ such that $\hat{\mu}=\varphi$ and $\mu(\spS_q')=1$ whenever $q>\frac{3d}{4}$. 
Now, for 
$\alpha>\frac{d}{4}+\frac{3d}{4}+\max\{0,\frac{\frac{3d}{4}-3}{2}\}
=d+\max\{0,\frac{3d-12}{8}\}$ 
there is $q>\frac{3d}{4}$ such that $\alpha>\frac{d}{4}+q+\max\{0,\frac{q-3}{2}\}$ so that by Lemma \ref{lem:matern} the mapping
\[
    \R^d\rightarrow (\spS_q,|\cdot|_q), 
    \quad y\mapsto \tau_y (k_{\alpha,m}^\vee)
\]
is correctly defined and continuous which, since $\mu(\spS_q')=1$ and hence $\mu(\spS' \setminus \spS_q')=0$, implies that for $\mu$-almost all $\omega\in\spS'$
\[
    \R^d\rightarrow\C,
    \quad
    y\mapsto
    \left\langle 
    \omega,\tau_y (k_{\alpha,m}^\vee)
    \right\rangle
    =
    \omega * k_{\alpha,m}(y)
\]
	is the composition of continuous functions and therefore continuous.
	
Finally, since $\varphi$ is $\tripleBar{\cdot}$--continuous it follows from inequality (\ref{mixed L^1-L^2-norm estimate}) that $\varphi$ is also $|\cdot|_{p}$-continuous for every $p>\frac{d}{4}$. In particular, $\varphi$ is $|\cdot|_q$-continuous for $q$ as above. Because $\tau_y(k_{\alpha,m}^\vee)$ belongs to $\spS_q$, the $|\cdot|_q$-completion of $\spS$, there is a sequence $(f_n)_{n\in\N}$ in $\spS$ which converges to $\tau_y(k_{\alpha,m}^\vee)$ with respect to $|\cdot|_q$. 
Applying Lebesgue's dominated convergence theorem thus yields
\begin{equation*}
\begin{split}
	\varphi(\tau_y(k_{\alpha,m}^\vee))
	&=
	\lim_{n\rightarrow\infty} \varphi(f_n)
	=
	\lim_{n\rightarrow\infty}
	\int_{\spS'} \e^{i\langle \omega, f_n\rangle} \; \mu(\d\omega)\\
	&=
	\lim_{n\rightarrow\infty}
	\int_{\spS_q'} \e^{i\langle\omega,f_n\rangle} \mu(\d\omega)
	=
	\int_{\spS_q'} 
	\e^{i\langle\omega,\tau_y(k_{\alpha,m}^\vee)\rangle} \;\mu(\d \omega),
\end{split}
\end{equation*}
which proves the theorem.
\end{proof}



\section{L\'evy Random Fields} \label{sec:LevyRandomField}

In this section we employ the setting introduced in Section~\ref{sec:random_fields} to construct smoothed L\'evy noise fields.
These---composed with suitable transformations---will then be used as random coefficient functions in the diffusion equation.

\subsection{Classification of Noise Fields} 
In this section we introduce and investigate the class of L\'evy noise fields.

\begin{definition}\label{def:levy}
Let $b\in\R, \sigma^2\geq 0$, and let $\nu$ be a $\sigma$-finite Borel measure on $\R\backslash\{0\}$ satisfying $\int_{\R\backslash\{0\}}\min\{1,s^2\}\; \nu(\d s)<\infty$. 
Then the function
\[
   \psi:\R\rightarrow\C,
   \quad
   \psi(t)
   :=
   ibt - \frac{\sigma^2t^2}{2}
   + 
   \int_{\R\backslash\{0\}}
  \left( \e^{i t s} - 1 - i t s\mathds{1}_{\{|s|\leq 1\}}(s)\right)\, \nu(\d s)
\]
is called the \emph{L\'evy characteristic} with \emph{characteristic triplet} $(b,\sigma^2,\nu)$.
	
A generalized random field $Z$ indexed by $\spS$ is called a \emph{L\'evy noise field} if there is a characteristic triplet $(b,\sigma^2,\nu)$ such that for the characteristic functional $\varphi$ of $Z$ there holds
\[
   \varphi(f) 
   = 
   \exp\biggl( \int_{\R^d}(\psi\circ f)(x)\,\d x \biggr),
   \qquad
   f\in\spS,
\]
where $\psi$ is the L\'evy characteristic associated with $(b,\sigma^2,\nu)$. 
(In particular, this assumes $\psi\circ f\in L^1(\R^d)$ for all $f\in\spS$.) 
We then say that $Z$ is \emph{associated with the characteristic triplet $(b,\sigma^2,\nu)$}.
\end{definition}

A classical reference for Lévy noise fields is \cite{gelfand64}. 
See also \cite{albeverio96,albeverio05}.

\begin{lemma} \label{lemma:decomposition}
A L\'evy noise field $Z$ can be decomposed as $Z = Z_D + Z_G + Z_J$ with 
deterministic part $Z_D$, 
Gaussian white noise $Z_G$, 
and pure jump noise $Z_J$, 
each of which are independent random fields with characteristic functionals 
\begin{align*}
    \varphi_{Z_D}(f) &= \exp\left( ib \int_{\R^d}f(x) \d x\right), \\ 
    \varphi_{Z_G}(f) &= \exp\left({-\frac{1}{2}\sigma^2\|f\|^2_{L^2(\R^d)}}\right), \text{ and } \\ 
    \varphi_{Z_J}(f) &= \exp\left({\int_{\R^d}\ \int_{\R \setminus\{0\}}
                    \e^{isf} - 1 - isf\mathds{1}_{\{|s|\leq 1\}} \nu(\d s) \d x}\right),
\end{align*}    
respectively.
\end{lemma}
\begin{proof}
This is a direct consequence of Definition \ref{def:levy}, which implies the factorization of the characteristic functional $\varphi_Z(f) = \varphi_{Z_D}(f) \varphi_{Z_G}(f) \varphi_{Z_J}(f)$.
\end{proof}

In view of Lemma~\ref{lemma:decomposition}, L\'evy noise is seen to be a generalization of Gaussian noise, to which it simplifies when the pure jump part $Z_J$ is omitted.

\begin{proposition}\label{prop:levy}
Let $\psi$ be a L\'evy characteristic with triplet $(b,\sigma^2,\nu)$ in which the Lévy measure satisfies $\int_{\R\backslash\{0\}}|s|\mathds{1}_{\{|s|>1\}}(s) \; \nu(\d s) < \infty$. 
Then,
\[
    \varphi:\spS\rightarrow\C,
    \qquad
    \varphi(f) := \exp\biggl(\int_{\R^d} (\psi\circ f)(x)\,\d x\biggr)
\]
is a correctly defined characteristic functional which is continuous with respect to the norm $\tripleBar{\cdot}$. 
In particular, there is a L\'evy noise field $Z$ (unique up to equivalence in law) associated with $(b,\sigma^2,\nu)$. 
Moreover, $Z$ is continuous with respect to $\tripleBar{\cdot}$.
\end{proposition}

\begin{proof}
For $f\in\spS$ we have $f\in L^1(\R^d)\cap L^2(\R^d)$ so that
\begin{align*}
	&\int_{\R^d}\int_{\R^\backslash\{0\}}
	  |\e^{is f(x)}-1-isf(x)\mathds{1}_{\{|s|\leq 1\}}(s)| \;\nu(\d s)\d x\\
	=&
	\int_{\R^d}\int_{\{0<|s|\leq 1\}}
	|\e^{isf(x)}-1-isf(x)| \; \nu(\d s) \d x
	+
	\int_{\R^d}\int_{\{|s|> 1\}}|\e^{isf(x)}-1| 
	\; \nu(\d s) \d x\\
	\leq&
	\int_{\R^d}\int_{\{0<|s|\leq 1\}}
	\frac{|s|^2|f(x)|^2}{2} \; \nu(\d s) \d x
	+
	\int_{\R^d}\int_{\{|s|> 1\}}|s| |f(x)| \;\nu(\d s) \d x\\
	\leq&
	\frac{1}{2}
	\int_{\R\backslash\{0\}}
	\min\{1,s^2\} \;\nu(\d s)\,\|f\|^2_{L^2(\R^d)}
	+ \int_{\{|s|> 1\}}|s| \;\nu(\d s)\,\|f\|_{L^1(\R^d)}
\end{align*}
yields $\psi\circ f\in L^1(\R^d)$ and 
\begin{equation*}
\begin{split}
   \|\psi\circ f\|_{L^1(\R^d)}
   \leq 
   &\big( 
   |b|
   + \int_{\{|s|> 1\}}|s| \;\nu(\d s)\big)\|f\|_{L^1(\R^d)} \\
   &+ 
   \left(
   	  \frac{\sigma^2+\int_{\R\backslash\{0\}}\min\{1,s^2\}\;\nu(\d s)}{2}
   \right)
   \|f\|^2_{L^2(\R^d)}.
\end{split}   
\end{equation*}
Thus,
\[
   \varphi:\spS \rightarrow\C,
   \quad
   \varphi(f):=\exp\left( \int_{\R^d} (\psi\circ f)(x) \;\d x \right)
\]
is correctly defined. 
Since $\varphi(0)=1$, the previous inequality implies that $\varphi$ is continuous at $0$ with respect to the norm $\tripleBar{\cdot}$.
	
Finally, the restriction of $\varphi$ to $\mathscr{D}(\R^d)$ is positive definite, see \cite[Theorem 6 p. 283]{gelfand64}. Since $\varphi$ is $\tripleBar{\cdot}$-continuous at $0$, the restriction of $\varphi$ to $\mathscr{D}(\R^d)$ is (uniformly) $\tripleBar{\cdot}$-continuous. Because the latter subspace of $\mathscr{S}$ is $\tripleBar{\cdot}$-dense in $\mathscr{S}$, $\varphi$ is positive definite. Therefore, by Theorem \ref{thm:correspondence characteristic functional and generalized random field} (and inequality (\ref{mixed L^1-L^2-norm estimate})) there is a generalized random field $Z$ indexed by $\spS$ (which is continuous with respect to the $\tripleBar{\cdot}$-norm) whose Fourier transform is $\varphi$, which proves the proposition. 
\end{proof}

\begin{remark} \label{rem:levy}
Convolving a compactly supported continuous function on $\R^d$ with an approximate identity shows that $\mathscr{D}(\R^d)$ is $\tripleBar{\cdot}$-dense in the compactly supported continuous functions on $\R^d$. 
Hence $\spS$ is dense in $L^1(\R^d)\cap L^2(\R^d)$ when the latter space is equipped with the norm $\tripleBar{\cdot}$. 
As noted in Remark \ref{rem:gen_rf}, it thus follows that for every $\tripleBar{\cdot}$-continuous L\'evy noise field $Z$ there is a unique generalized random field indexed by $L^1(\R^d)\cap L^2(\R^d)$ which extends $Z$. 
We denote this extension again by $Z$. 
In particular, for a Borel subset $\Lambda$ of $\R^d$ with finite Lebesgue measure, we can define the (non-normalized) $\Lambda$-average of the L\'evy noise field $Z$ by $Z(\mathds{1}_\Lambda)$.
\end{remark}

\begin{definition}[Stationary noise field]
A generalized random field $Z$ indexed by $\spS$ is called
\begin{enumerate}[(i)]
\item
a \emph{noise field} if for any choice of index functions $f_1, \ldots, f_n\in \spS$ with mutually disjoint supports the random variables $Z(f_1), \ldots, Z(f_n)$ are independent;
\item
\emph{stationary} if for every $f\in \spS$ and each $a\in\R^d$ the random variables $Z(f)$ and $Z(f_a)$ have the same probability distribution, i.e., $Z(f) \sim Z(f_a)$, where $f_a(x) = (\tau_a f)(x) = f(x-a)$; 
\item
a \emph{stationary noise field} if it is both a noise field and stationary.
\end{enumerate}
\end{definition}

Noise fields can be arbitrarily singular, as the distributional derivative of a noise field is again a noise field. 
In many situations, one would like to take spatial averages $Z(\mathds{1}_A)$ of a noise field for a bounded and measurable set $A\subseteq \R^d$ and also ensure that this quantity has finite expectation. 
Note that for such $A$, the indicator function $\mathds{1}_A\in V=L^1(\R^d)\cap L^2(\R^d)$, but $\mathds{1}_A\not \in\spS$ which rules out all too singular distributional noises. 
In the setting outlined above, we obtain the following characterization of stationarity:

\begin{theorem} \label{thm:levy}
Let $Z$ be a generalized random field on $\pspace$ indexed by $\spS$ which is $\tripleBar{\cdot}$-continuous. 
Assume that the unique $\tripleBar{\cdot}$-continuous extension of $Z$ to $L^1(\R^d)\cap L^2(\R^d)$ satisfies $Z(f)\in L^1\pspace$ for all $f\in L^1(\R^d)\cap L^2(\R^d)$.
Then the following are equivalent.
\begin{enumerate}[(i)]
\item $Z$ is a L\'evy noise field.
\item $Z$ is a stationary noise field.
\end{enumerate}
\end{theorem}

\begin{proof}
Assuming first that $Z$ is a L\'evy noise field, there is a characteristic triplet $(b,\sigma^2,\nu)$ with associated L\'evy characteristic $\psi$ such that the Fourier transform $\varphi$ of $Z$ satisfies $\varphi(f)=\exp\big(\int_{\R^d} (\psi\circ f) \,\d x\big), f\in\spS$. 
By the translation invariance of Lebesgue measure we conclude that for all $a\in\R^d$ there holds $\varphi(f_a)=\varphi(f)$, i.e., the random variables $Z(f_a)$ and $Z(f)$ have the same characteristic function and therefore $Z(f) \sim Z(f_a)$. 
Hence $Z$ is a stationary field.
	
Now, for $f_1,\ldots,f_n\in\spS$ with disjoint supports it follows for all $(\kappa_1,\ldots,\kappa_n)\in\R^n$ and every $x\in\R^d$ that $\psi\big(\sum_{j=1}^n\kappa_j f_j(x)\big)=\sum_{j=1}^n\psi(\kappa_j f_j(x))$ because at most one of the summands is different from $0$ and $\psi(0)=0$. 
Hence,
\begin{align*}
	\ev{\e^{i\sum_{j=1}^n\kappa_j Z(f_j)}}
	&=
	\ev{\e^{iZ(\sum_{j=1}^n \kappa_j f_j)}}
	=
	\exp\biggl(\int_{\R^d}\psi\bigl(\sum_{j=1}^n\kappa_j f_j(x)\bigr)\,\d x\biggr)\\		
	&=
	\prod_{j=1}^n\exp\biggl(\int_{\R^d}\psi(\kappa_j f_j(x))\,\d x\biggr)
	=
	\prod_{j=1}^n \ev{\e^{i\kappa_j Z(f_j)}},
\end{align*}
i.e., the Fourier transform of the joint distribution of the random variables $Z(f_1),\ldots,$ $Z(f_n)$ equals the product of the characteristic functions of the $Z(f_j)$, thus $Z(f_1),\ldots,$ $Z(f_n)$ are independent, so that (i) implies (ii).
	
In order to show that (ii) implies (i) we first observe that the $\tripleBar{\cdot}$-continuity of $Z$ allows, as described in Remark \ref{rem:levy}, to extend $Z$ uniquely to a generalized random field on $L^1(\R^d)\cap L^2(\R^d)$, where we equip the latter with the norm $\tripleBar{\cdot}$. 
We denote this extension again by $Z$. 
Since $\tau_a$ is a continuous linear operator on $\spS$ with respect to the norm $\tripleBar{\cdot}$, it follows by the $\tripleBar{\cdot}$-density of $\spS$ in $L^1(\R^d)\cap L^2(\R^d)$ that $Z(f_a)\sim Z(f)$ for all $f\in L^1(\R^d) \cap L^2(\R^d)$.
	
By definition, a box in $\R^d$ is a set of the form $\prod_{j=1}^d[\beta_j,\gamma_j)$ where $\beta_j, \gamma_j\in\R, 1\leq j\leq d$ with $\gamma_j-\beta_j>0$ independent of $j$, the so-called (side) length of the box. 
Let $\Lambda:=\prod_{j=1}^d[\beta_j,\gamma_j)$ be a box of length $L>0$. 
We subdivide $\Lambda$ into $n^d$ non-intersecting boxes $\Lambda_\ell$, each of side length $L/n$. 
Thus for each $1\leq \ell,k\leq n^d$ we have $Z(\mathds{1}_{\Lambda_\ell})\sim Z(\mathds{1}_{\Lambda_k})$ by the (extended) stationarity of $Z$. 
It follows from the construction that there are $\beta_j,\gamma_j\in\R, \gamma_j-\beta_j=L/n, 1\leq j\leq d$ as well as $a^{(1)},\ldots,a^{(n^d)}\in\R^d$ such that
\[
    \Lambda_\ell = \prod_{j=1}^d [a_j^{(\ell)}+\beta_j, a_j^{(\ell)}+\gamma_j),
    \qquad
    \ell = 1,2,\dots,n^d.
\]
For $\varepsilon\in (0,L/2n)$ we define
\[
   \Lambda_\ell^\varepsilon
   :=
   \prod_{j=1}^d
   [a_j^{(\ell)}+\beta_j+\varepsilon, a_j^{(\ell)}+\gamma_j-\varepsilon),
   \qquad
   \ell = 1,2,\dots,n^d.
\]
Moreover, let $\phi\in C^\infty_c(\R^d)$ be such that $\mbox{supp}\,\phi\subset(-1,1)^d$, $\phi\geq 0$, and $\int_{\R^d}\phi\,\d x=1$. 
We define $\phi_\varepsilon(x):=\varepsilon^{-d}\phi(x/\varepsilon), \varepsilon>0$. 
Then, for $\varepsilon\in (0,L/2n)$ it follows that 
$\phi_\varepsilon*\mathds{1}_{\Lambda_\ell^\varepsilon}, 1\leq \ell\leq n^d,$ 
are functions in $C^\infty_c(\R^d)$ satisfying
\begin{enumerate}[(1)]
\item 
$\forall\,1\leq \ell \leq n^d, \varepsilon\in(0,L/2n):\,\mbox{supp}\,\phi_\varepsilon*\mathds{1}_{\Lambda_\ell^\varepsilon}\subseteq \Lambda_\ell$,
\item
$\forall\,1\leq \ell\leq n^d, \varepsilon\in(0,L/2n):\,\sup_{x\in\R^d}|\phi_\varepsilon*\mathds{1}_{\Lambda_\ell^\varepsilon}|\leq 1$.
\end{enumerate}
In particular, for fixed $\varepsilon\in (0,L/2n)$ the functions 
$\phi_\varepsilon*\mathds{1}_{\Lambda_\ell^\varepsilon}, 1\leq \ell\leq n^d$, 
have mutually disjoint supports and, by Lebesgue's dominanted convergence theorem,
$\lim_{\varepsilon\rightarrow 0}\phi_\varepsilon*\mathds{1}_{\Lambda_\ell^\varepsilon}=\mathds{1}_{\Lambda_\ell}$ with respect to $\tripleBar{\cdot}$.

By the $\tripleBar{\cdot}$-continuity of $Z$, $Z(\phi_\varepsilon*\mathds{1}_{\Lambda_\ell^\varepsilon})\to Z(\mathds{1}_{\Lambda_\ell})$ in probability if $\varepsilon\searrow 0$, see Remark \ref{rem:gen_rf} (ii).
Consequently, also the vector $(Z(\phi_\varepsilon*\mathds{1}_{\Lambda_1^\varepsilon}),\ldots,Z(\phi_\varepsilon*\mathds{1}_{\Lambda_{n^d}^\varepsilon}))$ converges in probability to $(Z(\mathds{1}_{\Lambda_1}),\ldots, Z(\mathds{1}_{\Lambda_{n^d}}))$. 
As the random variables $Z(\phi_\varepsilon*\mathds{1}_{\Lambda_\ell^\varepsilon})$ are independent, their joint characteristic function factors into a product of individual characteristic functions, each converging to the characteristic function of the corresponding $Z(\mathds{1}_{\Lambda_\ell})$ (as convergence in probability is stronger than convergence in law, which is equivalent to the point-wise convergence of characteristic functions). 
Thus, the joint characteristic function of $(Z(\mathds{1}_{\Lambda_1}),\ldots, Z(\mathds{1}_{\Lambda_{n^d}}))$, which coincides with the limit of the joint characteristic function of $(Z(\phi_\varepsilon*\mathds{1}_{\Lambda_1^\varepsilon}),\ldots,Z(\phi_\varepsilon*\mathds{1}_{\Lambda_{n^d}^\varepsilon}))$, factors into a product of the characteristic functions of $Z(\mathds{1}_{\Lambda_\ell})$. 
This implies that the random variables $Z(\mathds{1}_{\Lambda_\ell})$ are independent. 

Defining
\[
    B_{\ell,n}
    :=
    \left[\beta_1+(\ell-1)\frac{L}{n},\beta_1+\ell\frac{L}{n}\right)
    \times
    \prod_{j=2}^n[\beta_j,\gamma_j),
    \qquad
    1\leq \ell\leq n,
\]
we obtain a partition of $\Lambda$ into $n$ sets of which each is a disjoint union of a mutually disjoint subfamily of the $\Lambda_1,\ldots,\Lambda_{n^d}$ such that $Z(\mathds{1}_{B_{1,n}}),\ldots,Z(\mathds{1}_{B_{n,n}})$ are i.i.d.\ random variables. 
Obviously, 
$Z(\mathds{1}_\Lambda)=\sum_{\ell=1}^n Z(\mathds{1}_{B_{\ell,n}})$ 
and, since $n\in\N$ was chosen arbitrarily, $Z(\mathds{1}_\Lambda)$ has an infinitely divisible probability law. 
    Thus, by the L\'evy-Khinchine Theorem \cite[Theorem 8.1]{Sato13} there is a uniquely determined characteristic triplet $(b_\Lambda,\sigma^2_\Lambda,\nu_\Lambda)$ with associated L\'evy characteristic $\psi_\Lambda$ such that
\[
   \ev{\e^{i Z(\mathds{1}_\Lambda)}}
   =
   \e^{|\Lambda|\psi_\Lambda(\kappa)}
   \quad\text{ and }\quad
   \ev{\e^{i\kappa Z(\mathds{1}_{\Lambda_\ell})}}
   =
   \e^{|\Lambda_\ell|\psi_\Lambda(\alpha)},
\]
for all $\kappa \in \R$ and $\ell=1,\dots,n^d$, where for a Borel set $B\subseteq\R^d$ we denote by $|B|$ its Lebesgue measure.

Let now $\Lambda'$ be a box of length $L'>0$ such that $L/L'$ is a rational number $n/m, n,m\in\N$. 
As above, we subdivide $\Lambda$ into $n^d$ mutually disjoint boxes $\Lambda_\ell$ of side length $L/n$ and  $\Lambda'$ into $m^d$ mutually disjoint boxes $\Lambda_k'$ of side length $L'/m$. 
Because $L/n=L'/m$, it follows from the (extended) stationarity of $Z$ that the random variables $Z(\mathds{1}_{\Lambda_\ell})$ and $Z(\mathds{1}_{\Lambda_k'})$ have the same distribution, $1\leq \ell\leq n^d, 1\leq k\leq m^d$. 
This implies
\[
   \e^{|\Lambda_\ell|\psi_\Lambda(\kappa)}
   =
   \e^{|\Lambda_k'|\psi_{\Lambda'}(\kappa)}
   \qquad
   \forall
   \kappa\in\R, 1\leq \ell\leq n^d, 1\leq k\leq m^d,
\]
so that by $|\Lambda_l|=(L/n)^d=(L/m)^d=|\Lambda_k'|$ and the continuity of the L\'evy characteristics $\psi_{\Lambda}$ and $\psi_{\Lambda'}$ it follows that there is $k\in\Z$ with $\psi_{\Lambda}(\kappa)=\psi_{\Lambda'}(\kappa)+2\pi\,ik$. 
Since $\psi_{\Lambda}(0)=0=\psi_{\Lambda'}(0)$ we conclude $\psi_{\Lambda}=\psi_{\Lambda'}$. Hence, there is a characteristic triplet $(b,\sigma^2,\nu)$ with associated L\'evy characteristic $\psi$ such that for all boxes $\Lambda$ with rational side length 
$\ev{\e^{i\kappa Z(\mathds{1}_\Lambda)}} = \e^{|\Lambda|\psi(\kappa)}, \kappa\in\R$. 
Because $Z(\mathds{1}_\Lambda)\in L^1\pspace$ it follows from \cite[Example 25.12]{Sato13} that $\int_{\{|s|> 1\}}|s|\; \nu(\d s)<\infty$ so that by Proposition \ref{prop:levy}
\[
    \varphi_\psi:\spS\rightarrow\C,
    \quad
    f\mapsto\exp\Big(\int_{\R^d}(\psi\circ f)(x)\,\d x\Big)
\]
is a correctly defined, positive definite functional which is $\tripleBar{\cdot}$-continuous and which can be extended in a unique way to a $\tripleBar{\cdot}$-continuous characteristic functional on $L^1(\R^d) \cap L^2(\R^d)$.	
	
Now let $\Lambda^{(1)},\ldots,\Lambda^{(n)}$ be mutually disjoint boxes in $\R^d$ of respective side lengths $L_j\in\Q$. 
By the same arguments as above for the $\Lambda_1,\ldots,\Lambda_n$ we obtain via mollification of the indicator functions of suitably shrunk boxes and the fact that $Z$ is a noise field, that $Z(\mathds{1}_{\Lambda^{(1)}}),\ldots, Z(\mathds{1}_{\Lambda^{(n)}})$ are independent. 
Considering the simple function $f=\sum_{j=1}^n\kappa_j\mathds{1}_{\Lambda^{(j)}}$, we obtain
\begin{equation}\label{fragile equality}
\begin{split}
  \varphi(f)
  &=
  \ev{\e^{iZ(f)}}
  =
  \prod_{j=1}^n \e^{|\Lambda^{(j)}|\psi}
  =
  \prod_{j=1}^n \e^{\int_{\Lambda^{(j)}}\psi(\kappa_j)\,\d x}
  =
  \e^{\int_{\R^d}\sum_{j=1}^n\psi(\kappa_j)\mathds{1}_{\Lambda^{(j)}}\,\d x}\\
  &=
  \e^{\int_{\R^d} (\psi\circ f)(x)\,\d x}
  =
  \varphi_\psi(f),
\end{split}
\end{equation}
where we have used again that for functions with mutually disjoint (essential) supports $f_1,\ldots,f_n\in L^1(\R^d) \cap L^2(\R^d)$ we have $\psi(\sum_j f_j)=\sum_j\psi(f_j)$ due to $\psi(0)=0$.
	
Finally, since simple functions of the above form are $\tripleBar{\cdot}$-dense in $L^1(\R^d)\cap L^2(\R^d)$ and $\varphi$ as well as $\varphi_\psi$ are $\tripleBar{\cdot}$-continuous it follows from (\ref{fragile equality}) that $\varphi(f)=\varphi_\psi(f)$ for all $f\in L^1(\R^d) \cap L^2(\R^d)$. 
In particular, $Z$ is a L\'evy noise field. 
\end{proof}

For computational purposes such as quadrature or Karhunen-Lo\`eve expansion, knowledge about the expectation of polynomial expressions in the random fields is needed. 
It is thus essential to calculate the moments of L\'evy noise fields.

\begin{proposition} \label{prop:moments}
Let $Z$ be a $\tripleBar{\cdot}$-continuous L\'evy noise field with characteristic triplet $(b,\sigma^2, \nu)$. 
Suppose the L\'evy measure $\nu$ is such that the following integrals exist and are finite

\[
   b_1 := \int_{\{|s|>1\}}s \,\nu(\d s)
   \qquad \text{ and } \qquad 
   b_n := \int_{\mathds{R}\setminus\{0\}} s^n \,\nu(\d s), \quad n\in \N, \quad n\geq2.
\] 
Then $Z(f)$ has moments of all orders for every $f\in\spS$ and
\begin{equation*}
	\ev{\prod^n_{j=1}Z(f_j)}
	=
	\sum_{\substack{I\in \mathscr{P}^{(n)}\\ I = \{I_1, \ldots, I_k\}}} 
	\prod_{\ell=1}^k c_{|I_\ell|}\int_{\R^d} \prod_{j\in I_\ell} f_j \,\d x.
\end{equation*}
Here, $\mathscr{P}^{(n)}$ is the collection of all partitions of $\{1, \ldots, n\}$ into non-intersecting and non-empty sets $\{I_1, \ldots, I_k\}$, where $k$ is arbitrary. 
$|I_\ell|$ denotes the number of elements in $I_\ell$ and $c_n$ is a sequence of constants defined as
\begin{equation*}
	c_n 
	= 
	\begin{cases}
		b + b_1 &: n = 1,\\
		\sigma^2 + b_2 &: n = 2,\\
		b_n &: n\geq 3.
	\end{cases}
\end{equation*}
\end{proposition}
\begin{proof}
Note that this is the cumulant expansion of moments. 
It can be obtained by application of F\`aa di Bruno's formula to $\varphi(f)=\exp\{\int \psi\circ f \d x\}$. 
For details, see e.g.\ \cite[Proposition 3.6]{albeverio96}. A complete proof can be found in Appendix \ref{sec:moment_gauss}.
\end{proof}

\subsection{Smoothed Stationary Noise Fields}

So far we have considered $\tripleBar{\cdot}$-continuous stationary noise fields $Z$ which by definition are indexed by $\mathscr{S}(\R^d)$ but can be extended uniquely to generalized random fields indexed by $L^1(\R^d)\cap L^2(\R^d)$ and these extensions will again be denoted by $Z$. 
However, we are interested in random functions on (an open, bounded subset $D$ of) $\R^d$ which will serve as conductivity coefficients for a stationary diffusion equation. 
We therefore introduce the following notion.

\begin{definition}[Smoothed Random Fields] \label{def:smoothRF}
For a $\tripleBar{\cdot}$-continuous stationary noise field $Z$ on the probability space 
$\pspace$ and a function $k\in L^1(\R^d)\cap L^2(\R^d)$ we define the 
\emph{smoothed random field (with window function, smoothing function, or smoothing kernel $k$)} as the family of random variables
\[
  Z_k(x) := Z(k_x)\in L^0\pspace, \quad  x\in\R^d,
\]
where $k_x:=\tau_x (k^\vee)=k(x-\cdot)$.
More generally, we shall call a bivariate $k: \R^d \times \R^d \to \R$ a \emph{smoothing function (window function)} if $ k(x,\cdot) \in L^1(\R^d) \cap L^2(\R^d)$ for every $x \in \R^d$ and $\R^d\rightarrow(L^1(\R^d)\cap L^2(\R^d),\tripleBar{\cdot}),x\mapsto k(x,\cdot)$ is continuous. For a bivariate smoothing function $k$ we set (using the same notation) $k_x := k(x,\cdot)$ and define the smoothed random field with smoothing function $k$ as the family of random variables
$
  Z_k(x) := Z(k_x)\in L^0\pspace, \;  x\in\R^d.
$

\end{definition}

\begin{remark} \label{rem:smoothRF}
\hspace{2ex}
\begin{itemize}
\item[(i)] 
Since by Minlos' Theorem \ref{thm:minlos} every generalized random field $Z$ on $\pspace$ indexed by $\mathscr{S}(\R^d)$ is given by a $\mathscr{S}'(\R^d)$-valued random variable (again denoted by $Z$), it follows for a window function $k\in\mathscr{S}(\R^d)$ that
\[
   \R^d\rightarrow\R,
   \quad
   x\mapsto Z(k_x)=\langle Z,k(x-\cdot)\rangle = (Z*k)(x)
\]
is $\prob$-almost surely a smooth function as a convolution of a random tempered distribution with a Schwartz function.
\item[(ii)] 
For stationary noise fields $Z$ it follows from the definition that, for an arbitrary window function $k\in\mathscr{S}(\R^d)$, the random variables of the associated smoothed random field $Z_k$ are identically distributed, i.e., $Z_k(x_1)\sim Z_k(x_2)$ for every $x_1,x_2\in\R^d$. 
Moreover, whenever $x_1,\ldots, x_n\in\R^d$ are such that $\tau_{x_1}f,\ldots,\tau_{x_n}f$ have mutually disjoint supports, then the random variables $Z_k(x_1),\ldots,Z_k(x_n)$ are independent. 
It follows by standard arguments already employed in the proof of Theorem \ref{thm:levy} that for a $\tripleBar{\cdot}$-continuous stationary noise field $Z$ the ``noise field property" as well as  stationarity hold not only for $f_1,\ldots,f_n, f\in\mathscr{S}(\R^d)$ but for arbitrary $f_1,\ldots,f_n, f\in L^1(\R^d)\cap L^2(\R^d)$.
\item[(iii)] 
By Theorem \ref{preparation for Levy} it follows that every $\tripleBar{\cdot}$-continuous stationary noise field $Z$ is given by an $\mathscr{S}'_q(\R^d)$-valued random variable, for an arbitrary $q>\frac{3d}{4}$, where we use the notation introduced in Section \ref{sub:Schwartz}. Therefore, if $k\in\mathscr{S}_q(\R^d)$ is a function such that
\begin{equation}\label{eqn:translates in S_q}
		\R^d\rightarrow \mathscr{S}_q(\R^d),
		\qquad
		x\mapsto \tau_x k 
\end{equation}
is correctly defined, we may consider the smoothed random field $Z_k$ with window function $k$ even if $k\notin L^1(\R^d)\cap L^2(\R^d)$. Moreover, if the function in (\ref{eqn:translates in S_q}) is also continuous (when $\mathscr{S}_q(\R^d)$ is endowed with the Hilbert space norm $|\cdot|_q$) the resulting smoothed random field is almost surely a continuous function on $\R^d$. For Mat\'ern kernels as window functions we considered this smoothing procedure in Theorem \ref{thm:matern}.
\item[(iv)] 
One can also view the smoothed L\'evy field $Z_{k_{\alpha,m}}$ as the distributional solution of the linear stochastic pseudodifferential equation
\[
	(-\Delta+m^2)^\alpha Z_{k_{\alpha,m}}(f) = Z(f), \qquad f\in \mathscr{S},
\]   
where $Z$ is a L\'evy noise field and $\Delta$ denotes the Laplacian, see \cite{albeverio96,albeverio05}.
In this sense our approach directly extends that given in \cite{Lindgren2011}  for sampling Gaussian fields.
\end{itemize}
\end{remark}

\subsection{Examples}

We now present some examples for smoothed L\'evy random fields.

\subsubsection{Gaussian Fields} 

Gaussian random fields are obtained as a special case of Lévy fields by setting $\nu = 0$. 
The generalized random field associated with the characteristic triplet $(b, \sigma^2, 0)$ results in a stationary, uncorrelated noise field denoted by $G$ with characteristic functional 
$\varphi_G(f)
=
\exp\left(ib\int_{\R^d}f(y) \,\d y-\frac{\sigma^2}{2}\|f\|_{L^2(\R^d)}^2\right)$. 
Since $b$ corresponds to a deterministic background field, we obtain classical white noise for $b=0$. 
For $f\in L^1(\R^d)\cap L^2(\R^d)$, the corresponding random variable $G(f)$ has variance $\sigma^2 \|f\|_{L^2(\R^d)}$.

For the Gaussian random field smoothed with a window function $k\in L^1(\R^d)\cap L^2(\R^d)$ it follows that $G(k_{x})$ has a Gaussian distribution with mean $b\int_{\R^d}k(y) \, \d y$ and variance $\sigma^2\int_{\R^d} k^2(y) \, \d y$ for each $x\in\R^d$. 
Moreover, a straightforward calculation shows

\begin{equation} \label{eqa:CovFromSmooth}
   \Cov(G_k(x_1), G_k(x_2)) 
   = 
   \sigma^2\int_{\R^d}k(x_1-\tau)k(x_2-\tau) \,\d \tau
   =
   \sigma^2(k^\vee*k)(x_1-x_2).
\end{equation}
In particular, setting $k=k_{\alpha,m}$ we obtain  
$\Cov(G_k(x_1),G_k(x_2)) = \sigma^2 k_{2\alpha,m}(x_1-x_2)$, 
which is the usual Matérn covariance function with smoothness parameter $2\alpha$. 

\begin{remark} \label{rem:ContGauss}
In connection with Lemma \ref{lemma:mat-hold} one can see from the preceding that the lower bound $\alpha >d$ in $d=1,2,3$ obtained in Theorem \ref{thm:matern} is not optimal for the Gaussian case, where $2\alpha > d$ is sufficient to obtain a continuous modification of $Z_{k_{\alpha,m}}(x)$ by the Kolmogorov continuity criterion for random fields  \cite{kunita90}.  
\end{remark}

\subsubsection{Compound Poisson Random Fields}

As an example for a Lévy noise field, we consider a compound Poisson random field, which is a special case with finite jump activity.

\begin{definition}[Compound Poisson Random Field] 	\label{def:CPnoise}
Let $\nu$ be a finite Lévy measure on $\R\backslash\{0\}$. 
Setting $b := \int_{\{0<|s|\leq 1\}} s \, \d\nu$ we call the generalized random field associated with the characteristic triplet $(b,0,\nu)$ a \emph{compound Poisson random field}, denoted by $P$.
\end{definition}

The associated characteristic functional is given by
\begin{equation} \label{eqa:CF_CompoundPoisson}
   \varphi_P(f) 
   = 
   \exp
   \left(
   \int_{ \mathbb{R}^d} \int_{\mathbb{R} \setminus\{0\}} 
   \left(\e^{is f(x)} - 1 \right)
   \nu(\d s) \,\d x  
   \right), 
   \qquad
   f \in L^1(\R^d)\cap L^2(\R^d).
\end{equation}
Let $f\in L^1(\R^d)\cap L^2(\R^d)$ with essential support in a region $\Lambda\subseteq \R^d$ with $|\Lambda|<\infty$. 
As the L\'evy measure $\nu$ is finite, we define an intensity parameter 
$\lambda := \nu(\R\backslash\{0\})$ and  obtain a probability measure $\tilde \nu$ on $\R$ by setting 
$\tilde \nu := \lambda^{-1}\nu$ and $\tilde\nu(\{0\}) := 0$.

Now, let $\pspace$ be a probability space, let $N_\Lambda$ be a Poisson-distributed random variable with intensity $\lambda|\Lambda|$, and let $(X_1,S_1), (X_2,S_2),\ldots$ be a sequence of $\R^d\times\R\backslash\{0\}$-valued random variables which are identically distributed with $(X_1,S_1)\sim \frac{\d x}{|\Lambda|}\otimes\tilde\nu$, where $\d x$ is restricted to  $\Lambda$, and such that $N_\Lambda, (X_1,S_1), (X_2,S_2),\ldots$ are independent. 
We define $P_\Lambda := \sum_{j=1}^{N_\Lambda} S_j\delta_{X_j}$. 
Let in addition $\{\Lambda_j\}_{j\in \N}$ denote a partition of $\R^d$ such that any compact set intersects at most finitely many $\Lambda_j$ and let the $P_{\Lambda_j}$ be mutually independent. 
We set $P=\sum_{j=1}^\infty P_{\Lambda_j}$. For $f\in \mathscr{S}$ with compact support let $I=\{j\in \N: \Lambda_j\cap {\rm supp} f\not=\emptyset\}$. 
Then, 
\begin{align}\label{eqa:CF_calc}
\begin{split}
    \varphi_P(f)
    &=
    \ev{\e^{iP(f)}}
    =
    \ev{\e^{i\sum_{j\in I}P_{\Lambda_j}(f)}}
    =
    \prod_{j\in I}\ev{\e^{iP_{\Lambda_j}(f)}}\\
    &=
    \prod_{j\in I} \sum_{\ell_j=1}^\infty \prob (N_{\Lambda_j}=\ell_j) \prod_{r_j=1}^{\ell_j}\ev{\e^{iS^{(j)}_{r_j}f(X^{(j)}_{r_j})}}\\
    &=
    \prod_{j\in I} \sum_{\ell_j=1}^\infty \e^{-\lambda|\Lambda_j|} \frac{(\lambda|\Lambda_j|)^{\ell_j}}{\ell_j!} 
    \left(
    \int_{\Lambda_j} \int_{\R} 
    \e^{isf(x)}\, \tilde\nu(\d s)\, \frac{\d x}{|\Lambda_j|}
    \right)^{\ell_j}\\
    &=
    \prod_{j\in I}
    \exp
    \left(
    \lambda\int_{\Lambda_j}\int_{\R} (\e^{isf(x)}-1)\, \tilde\nu(\d s)\, \d x
    \right)\\
    &=
    \exp
    \left(
    \int_{\mathds{R}^d} \int_{\mathds{R}\setminus\{0\}} 
    (\e^{is f(x)} - 1) \,\nu(\d s) \,\textup{d}x\right).
\end{split}
\end{align}
As the set of compactly supported Schwartz test functions is dense in $\mathscr{S}$, we conclude (see Remark \ref{rem:gen_rf} (ii)) that we can extend $P(f)$ to $\mathscr{S}$. As the characteristic functionals coincide, the so constructed random field $P(f)$ coincides with the field from Theorem \ref{thm:minlos} up to equivalence in law. Using the Borel-Cantelli lemma, it is an easy exercise to show that also in this representation, the locally finite and discrete signed measure $P$ actually is a tempered distribution $\prob$-a.s..

For the smoothed compound Poisson noise field, we obtain $P_k=\sum_{j=1}^\infty P_{\Lambda_j,k}$. 
Let us assume that $k$ is a continuous, bounded function in $L^1(\R^d)$ (and hence also in $L^2(\R^d)$), then
\[
    P_{\Lambda_j,k}(x) = \sum_{\ell=1}^{N_{\Lambda_j}}S_\ell^{(j)}k(x-X_\ell^{(j)})
\]
shows that $P_{\Lambda_j,k}(x)$ is represented as a continuous function. 
Although this does not immediately imply that  this is also true for $P_k(x)$, it seems likely that this is true if $k(x)$ has some global uniform continuity and decays sufficiently fast, see \cite[Theorem 3.2]{albeverio05} for a related argument. 

\begin{remark}\label{rel:levy_pos}
As a useful estimate, we note that the absolute value of the signed measure $|P|$ is given by $\sum_{j=1}|P_{\Lambda_j}|$ and $|P_{\Lambda_j}|=\sum_{\ell=1}^{N_{\Lambda_j}}|S_\ell^{(j)}| \,  \delta_{X_\ell^{(j)}}$ holds $\prob$-a.s. 
We therefore conclude that for almost every $x \in \mathbb R^d$ there holds
\begin{equation} \label{eqa:bound}
    |P_k(x)|\leq |P|_{|k|}(x)
\end{equation}
$\prob$-almost surely. Note that the right hand side, for fixed $x$, is almost surely finite, as clearly $|P|$ also is a compound Poisson noise field with characteristic triplet $(b^+,0,\nu^+)$, where $\nu^+$ is the image measure of $\nu$ under the mapping $s\mapsto |s|$. 
Clearly, $\nu^+$ is supported on $(0,\infty)$. 
\end{remark}

\subsubsection{L\'evy Noise of Infinite Activity} \label{sec:inf_act}

Under the assumption $\int_{\{|s|\leq 1\}} |s| \,\nu(\d s)<\infty$, the deterministic compensator term for the small jumps $\int_{\R\setminus\{0\}} s \mathds{1}_{\{|s|\leq 1\}}(s)\,\nu( \d s) $ can still be subsumed into the constant $b$ and the expression for the characteristic functional \eqref{eqa:CF_CompoundPoisson} remains valid. 
This case includes important examples such as (bi-) gamma distributions with $\nu(\d s)=v\mathds{1}_{\{s>0\}}(s)\frac{\e^{-w s}}{s}\,\d s$ 
($\nu(\d s)=v\frac{\e^{-w |s|}}{|s|}\, \d s$), $v,w>0$. 
As in this case and in others the jump measure $\nu$ is infinite, the representation given in the compound Poisson case has to be extended as follows:
The sets $\Theta_0=\{s\in \R:|s|>1\}$ and $\Theta_\ell=\{s\in \R:\frac{1}{\ell}\geq |s|>\frac{1}{\ell+1}\}$ form a partition of $\R\setminus\{0\}$. 
Then the L\'evy measures $\nu_\ell(\d s)=\mathds{1}_{\Theta_\ell}(s) \nu(\d s)$ 
are all finite and define independent compound Poisson processes $P_\ell$. 
With a calculation similar to \eqref{eqa:CF_calc}, we deduce that \eqref{eqa:CF_CompoundPoisson} is still valid and the same applies 
to \eqref{eqa:bound}, where $|P|=\sum_{\ell=1}^{\infty}|P_\ell|$. 
Also, $|P|$ is a L\'evy noise with triplet $(b^+,0,\nu^+)$ and the r.h.s. of \eqref{eqa:bound} is $\prob$-a.s. true $\forall x\in \R^d$ also in this case. 
Figure~\ref{fig:paths} shows sample paths of Gaussian, Poisson (compound Poisson with $\nu=\delta_1$) and bi-directional gamma noise fields with identical covariance.

\begin{figure}
\centering
\includegraphics[height=.18\textheight]{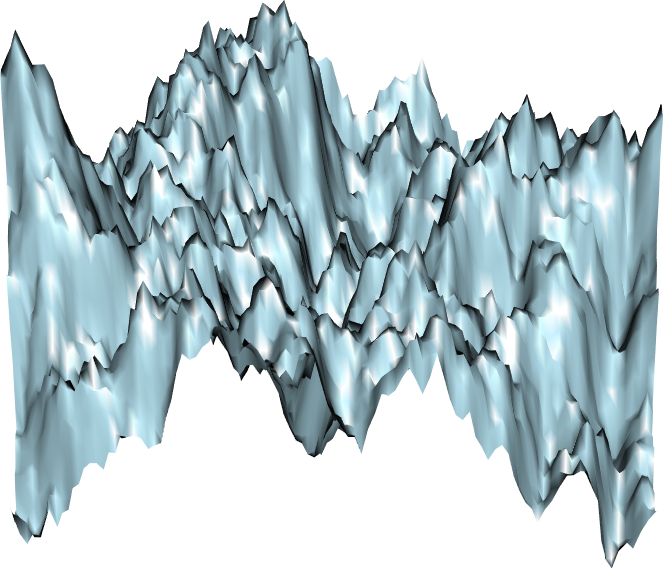}.  \hspace{.1cm}
\includegraphics[height=.18\textheight]{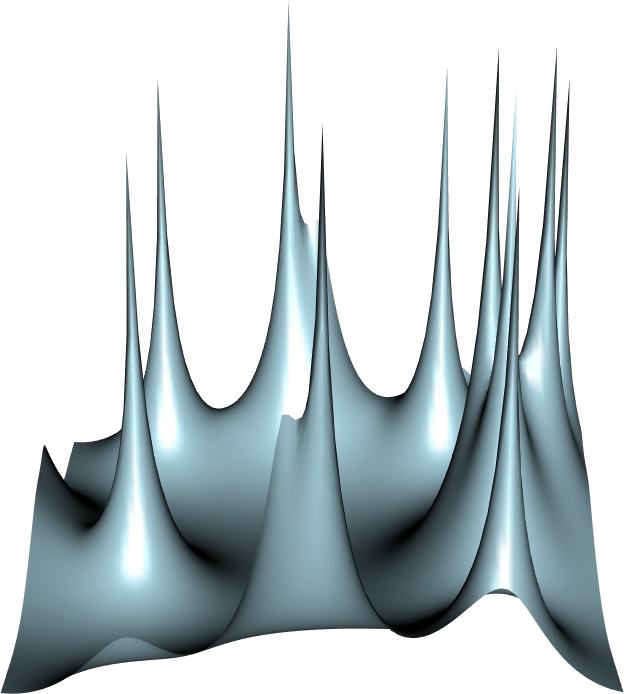} \hspace{.1cm}
\includegraphics[height=.18\textheight]{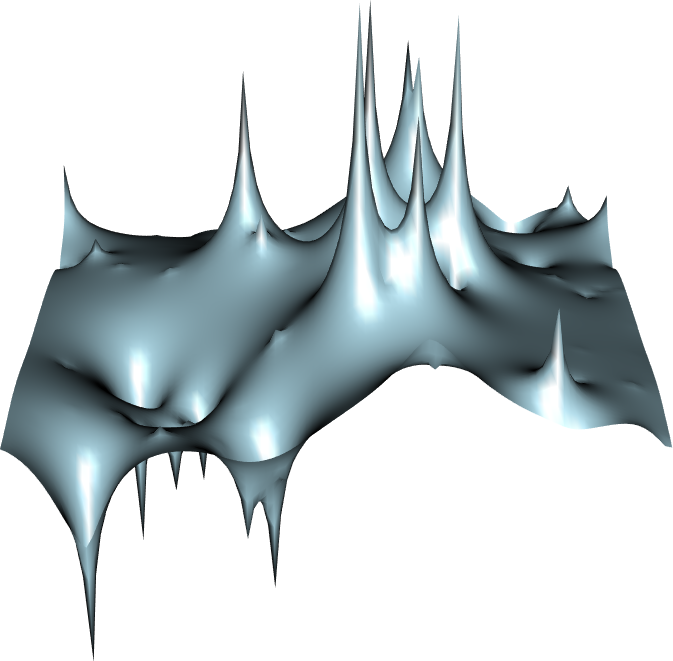}
\caption{Realizations of smoothed noise fields: Gaussian (left), Poisson (middle) and bi-gamma (right) -- each with the same Mat\'ern covariance function.}
    \label{fig:paths}
\end{figure}

\section{Existence of Moments of Solutions of the Random Diffusion Equation}
\label{section:existence_of_moments}

In this section we first construct pathwise solutions to the linear stationary diffusion equation with transformed smoothed Lévy random fields as coefficients. 
After establishing a connection to extreme value theory, based on an estimate due to Talagrand for the Gaussian part combined with a large deviation-type estimate for the Poisson part, we identify sufficient conditions for the moments of the solution exist.

\subsection{Pathwise Existence and Measurability of Solutions}

Given a domain $D \subset \mathbb{R}^d$, i.e.\ $D$ is open, bounded, and connected with a Lipschitz boundary $\partial D$, a measurable partition of its boundary $\partial D = \partial_D \cup \partial_N$ such that $\partial_D \cap \partial_N = \emptyset$ and 
such that $\partial_D$ has positive surface measure,
we consider the boundary value problem for the stationary diffusion equation
\begin{equation} \label{eq:diff_eq}
	\left\{ 
	\begin{array}{rll}
		-\nabla \cdot (a \nabla u) = f & \mbox{ in }D,\\
		u = g_D                        & \mbox{ along }\partial_D,\\
		\vn \cdot a\nabla u = g_N    & \mbox{ along }\partial_N,
	\end{array}
	\right.
\end{equation}
where $f\in L^2(D)$ is a given source term, 
$g_D \in H^{\frac{1}{2}}(\partial_D)$ denotes given Dirichlet boundary data, $g_N \in H^{-\frac{1}{2}}(\partial_N)$  given Neumann boundary data, and $\vn$ denotes the unit outward normal vector along $\partial D$. 
The coefficient function $a \in L^\infty(D)$ models the conductivity throughout the domain $D$. 
As usual, we interpret (\ref{eq:diff_eq}) in the weak sense.

The boundary value problem \eqref{eq:diff_eq} models a great variety of phenomena in the physical sciences, among these groundwater flow in a porous medium governed by Darcy's law, which expresses the (pointwise) volumetric flux  as a function of the hydraulic head $u$ by $-a(x)\nabla u(x)$.
In such a setting, the precise value of the conductivity coefficient is typically uncertain, e.g.\ derived from sparse information based on limited observations.
Modeling such uncertainty by introducing a probability distribution on the set of admissible coefficient functions $a$ results in a random PDE.

We now model the coefficient function $a$ as a transformed smoothed random field $a(x) = T(Z_k(x))$ with a suitable Borel-measurable real-valued function $T$ and a $\tripleBar{\cdot}$-continuous stationary noise field $Z$ smoothed by a window function $k\in L^1(\R^d)\cap L^2(\R^d)$.
Our goal is the estimation of quantities of interest associated with the solution $u$ of the random boundary value problem such as statistical moments, the probability of certain events or the expected or  maximal flow through a subdomain or boundary.

As a first step, we establish the pathwise existence and uniqueness of solutions. 
For each $\omega\in\Omega$, the assumption
\begin{equation} \label{ellipticity inequality}
	0 
	< 
	\essinf_{x\in D} a(x,\omega) 
	\leq 
	\esssup_{x\in D} a(x,\omega) 
	< 
	\infty
\end{equation}
ensures that the differential operator in the boundary value problem \eqref{eq:diff_eq} obtained by setting $a$ to be a realization $a(\cdot,\omega)$ of the random field $a=T\circ Z_k$ is strictly elliptic. 
Hence, there exists a unique $u=u(\cdot,\omega)\in H^1(D)$ which solves \eqref{eq:diff_eq} with $a=a(\cdot,\omega)$. 

\begin{lemma} \label{lemma:weak solution}
\hspace{2ex}
\begin{enumerate}
\item[a)] 
For $a\in L^\infty(D)$ with $\essinf a>0$, $f\in L^2(D)$, $g_D\in H^{\frac{1}{2}}(\partial_D)$, and $g_N\in H^{-\frac{1}{2}}(\partial_N)$, the problem \eqref{eq:diff_eq} has a unique solution $u\in H^1(D)$. 
Moreover, there is a constant $C\geq 1$ independent of $a,f, g_D$, and $g_N$ such that
\begin{equation} \label{eq:a-prio}
	\|u\|_{H^1(D)}
	\leq 
	C \, \frac{1+\|a\|_\infty}{\essinf a}
	\Big(
	\|f\|_{L^2(D)}  + \|g_D\|_{H^{\frac{1}{2}}(\partial_D)} + \|g_N\|_{H^{-\frac{1}{2}}(\partial_N)}
	\Big).
\end{equation}
One can choose $C=(1+C_P^2)\max\{1,2\|E\|, \| \trace \|\}$, where $C_P>0$ only depends on $D$ and $\partial_D$ and where $E:H^{1/2}(\partial_D)\rightarrow H^1(D)$ denotes an extension operator and $\trace:H^1(D)\rightarrow H^{1/2}(\partial D)$ the trace operator.
\item[b)] 
Let $Z$ be a $\tripleBar{\cdot}$-continuous generalized random field and $k\in L^1\cap L^2(\R^d)$ a window function such that the random field $(Z_k(x))_{x\in\R^d}$, has almost surely continuous paths. 
Then for a strictly positive, locally Lipschitz continuous function $T$ on $\R$, we have for the random conductivity $a:=T\circ Z_k\in L^\infty(D)$ as well as $\essinf a > 0$ almost surely. 
Denoting the (almost surely) existing solution of \eqref{eq:diff_eq} with conductivity function $a(\cdot,\omega)$ by $u(\cdot,\omega)$, the mapping $\omega\mapsto u(\cdot,\omega)$ is an $H^1(D)$-valued, Borel-measurable random variable.
\item[c)] 
Let $Z$ be a $\tripleBar{\cdot}$-continuous generalized random field and let $k_{\alpha,m}$ be a Mat\'ern kernel with $\alpha>d+\max\{0,\frac{3d-12}{8}\}$. 
Then the random field $(Z_{k_{\alpha,m}}(x))_{x\in\R^d}$ has almost surely continuous paths. The same assertion holds for $\alpha > d/2$ in case $Z$ is a Gaussian random field.  
\end{enumerate}
\end{lemma}

\begin{proof}
Existence and uniqueness in a) are well known and can be found in many textbooks on elliptic boundary value problems. 
However, as we could not find a reference for the a priori bound \eqref{eq:a-prio}, we provide a brief sketch of its proof for the reader's convenience. 
We seek $u\in H^1(D)$ with $\trace(u) = g_D$ on $\partial_D$ and
\begin{eqnarray}\label{eq:diff_eq_weak}
	\int_D a  \nabla u \cdot \nabla v \,\d x
	&=&
	\int_{\partial_N} g_N \trace(v) \,\d\sigma 
	+
	\int_D [fv - a \nabla (E g_D) \cdot \nabla v] \,\d x\;( =: \ell(v))
\end{eqnarray}
for all $v\in H^1_D(D):=\{w\in H^1(D);\,\trace(v)=0 \mbox{ on }\partial_D\}$, where $E$ denotes an extension operator from $H^{\frac{1}{2}}(\partial_D)$ to $H^1(D)$. 
By \cite[Theorem 6.1.5.4 (page 358)]{Triebel1992} (and \cite[Theorem 3.29, Theorem 3.30]{McLean2000}), the left-hand side of \eqref{eq:diff_eq_weak} defines an inner product $(\cdot,\cdot)_a$ on the closed subspace $H^1_D(D) \subset H^1(D)$ whose associated norm $\|\cdot\|_a$ satisfies, for some suitable $C_P>0$,
\begin{equation} \label{equivalence of Sobolev norms}
	\sqrt{\frac{\essinf a}{1+C_P^2}} \|v\|_{H^1(D)}
	\leq
	\|v\|_a\leq \|a\|_\infty\|v\|_{H^1(D)}
	\qquad \forall \,v \in H^1_D(D).
\end{equation}
Applying Riesz' Representation Theorem to the continuous linear functional $\ell$ on the right in \eqref{eq:diff_eq_weak} gives a unique $v_\ell \in H^1_D(D)$ with $(v_\ell,v)_a = \ell(v) \;\forall \,v \in H^1_D(D)$ and
\begin{equation*}
	\|v_\ell\|_{H^1 (D)}
	\leq 
	\frac{1+C_P^2}{\essinf a} 
	\Big(
	\|f\|_{L^2(D)} + \|a\|_\infty\|E\|\|g_D\|_{H^{\frac{1}{2}}(\partial_D)}
	+ \|g_N\|_{H^{-\frac{1}{2}}(\partial_N)}\|\trace\|
	\Big).
\end{equation*}
Hence $u := v_\ell + Eg_D$ is the unique (weak) solution of \eqref{eq:diff_eq} and the desired inequality follows with $C:=(1+C_P^2)\max\{1,2\|E\|, \|\trace\|\}$.
	
To prove b), we first show that $\pspace\rightarrow C(\bar{D}), \omega\mapsto Z_k(\cdot,\omega)$ is measurable with respect to the Borel $\sigma$-algebra generated by the $\|\cdot\|_\infty$-norm. 
In fact, as $Z_k(x)\in L^0\pspace$, $x\in\R^d$, for any $q\in C(\bar{D})$ and $ \varepsilon>0$ we have that
\[
   Z_k^{-1}\big(B_\varepsilon(q)\big)
   =
   \{\|Z_k-q\|_\infty\leq\varepsilon\}=\bigcap_{x\in\bar{D}\cap\Q^d}\{|Z_k(x)-q(x)|
   \leq \varepsilon\}
\]
is measurable.
Since $(C(\bar{D}),\|\cdot\|_\infty)$ is separable, every open $U\subseteq C(\bar{D})$ is a countable union of open balls $B_\varepsilon(q)$, so the above implies that $\{Z_k\in U\}$ is measurable for any open $U\subseteq C(\bar{D})$.
	
Furthermore, due to the local Lipschitz continuity of $T$, $q\mapsto T\circ q$ is $\|\cdot\|_\infty$-continuous on $C(\bar{D})$ and thus $\|\cdot\|_\infty$-Borel measurable. 
To see that for fixed $f\in L^2(D), g_D\in H^{\frac{1}{2}}(\partial_D)$, and $g_N\in H^{-\frac{1}{2}}(\partial_N)$ the solution map
\[
   C_+(\bar{D})
   :=
   \{a\in C(\bar{D});\,\inf a>0\}\rightarrow H^1(D),
   \quad 
   a\mapsto u_a
\]
is continuous, where $u_a$ denotes the unique solution to \eqref{eq:diff_eq} with conductivity $a$, cf.\ \cite{Hackbusch2017} or the methods applied in Section~5. 
Thus, $u\in L^0(\Omega, H^1(D))$, i.e.\ b) holds.
	
	Finally, c) is an immediate consequence of Theorem \ref{thm:matern} and Remark \ref{rem:ContGauss}.
\end{proof}

\subsection{Integrability of the Solution}

In this subsection we investigate the integrability of solutions to the boundary value problem \eqref{eq:diff_eq} with random diffusion coefficient $a$ given by a transformed smoothed Lévy noise field. 
More generally, we are interested in the existence of moments of the Sobolev norm of solutions.
Our first result in this direction states that the moments of the weak solution $u$ can be estimated using the extreme values of the random diffusion coefficient $a=T\circ Z_k$.

\begin{lemma} \label{lemma:p-est}
Let $Z$ be a $\tripleBar{\cdot}$-continuous generalized random field and let $k\in L^1(\R^d)\cap L^2(\R^d)$ be a window function such that $(Z_k(x))_{x\in\R^d}$ has almost surely continuous paths. 
Moreover, let $T$ be locally Lipschitz on $\R$ such that with $h\geq 0$, $B,\rho>0$ it holds that $B^{-1} \e^{-\rho |z|^h}\leq T(z)\leq B \e^{\rho|z|^h}$ for all $z\in\R$.
	
Then, for the random conductivity $a:=T\circ Z_k$ there is $C\geq 1$ such that for all $f\in L^2(D)$, $g_D\in H^{\frac{1}{2}}(\partial_D)$, and $g_N\in H^{-\frac{1}{2}}(\partial_N)$ the solution $u$ to the random boundary value problem \eqref{eq:diff_eq} satisfies
\begin{equation*}
	\ev{\lVert u \rVert_{H^1(D)}^n}
	\leq 
	\tilde{C}^n 2^{n-1}(B^n+B^{2n}) 
		\, 
		\sum_{j=0}^\infty e^{2n\rho(j+1)^h} \, 
		\prob(\sup_{x\in D}\lvert Z_k(x) \rvert \geq j),
	\qquad \forall\,n\in\N
\end{equation*}
with $\tilde{C}=C(\|f\|_{L^2(D)}+\|g_D\|_{H^{\frac{1}{2}}(\partial_D)}+\|g_N\|_{H^{-\frac{1}{2}}(\partial_N)})$.
\end{lemma}

\begin{proof}
By assumption, there is a $\prob$-null set $N \in {\mathfrak A}$ such that $(Z_k(x))_{x\in\R^d}$ has continuous paths on $N^c$. 
Thus, by setting the following functions equal to zero on $N$ as necessary, both
\[
    \|a\|_\infty = \sup_{x\in D}T(Z_k(x)) = \sup_{x\in D\cap\Q^d}T(Z_k(x))
\]
and
\[
   \essinf a = \inf_{x\in D}T(Z_k(x)) = \inf_{x\in D\cap\Q^d}T(Z_k(x))
\]    
are measurable. 
Applying Lemma \ref{lemma:weak solution} a) and the law of total probability yields
\begin{equation*}
\begin{split}
	&\tilde{C}^{-n} \ev{\| u\|_{H^1(D)}^n}
	\leq 
	\ev{\left(\frac{1+\sup_{x\in D} T(Z_k(x))}{\inf_{x\in D}T(Z_k(x))}\right)^n} \\
	&\leq
	\ev{\frac{(1+B\sup_{x\in D} \e^{\rho|Z_k(x)|^h})^n}
	         {(B\inf_{x\in D}e^{\rho|Z_k(x)|^h})^{-n}}}\\
	&\leq 
	\ev{\frac{2^{n-1}+2^{n-1}B^n \e^{n\rho\sup_{x\in D}|Z_k(x)|^h}}
	        {B^{-n} \e^{-n\rho\sup_{x\in D}|Z_k(x)|^h}}}
	\leq 
	2^{n-1} (B^n+B^{2n}) \ev{\e^{2n\rho \sup_{x\in D}|Z_k(x)|^h}}\\
	&\leq 
	2^{n-1}(B^n+B^{2n})
	\sum_{j=0}^\infty 
	\ev{\e^{2n\rho(j+1)^h} | j \leq \sup_{x\in D}|Z_k(x)| < j+1}
	\prob(\sup_{x\in D}|Z_k(x)|\geq j)\\
	&\leq 
	2^{n-1}(B^n+B^{2n})
	\sum_{j=0}^\infty \e^{2n\rho(j+1)^h} \prob(\sup_{x\in D}|Z_k(x)|\geq j).
\end{split}
\end{equation*}
~
\end{proof}

Lemma~\ref{lemma:p-est} shows the need for estimating the probabilities $\prob(\sup_{x\in D}\lvert Z_k(x) \rvert  \geq j), j \in \mathbb N$. 
For a smoothed Lévy noise field $Z_k$ we obtain such an estimate by decomposing $Z_k$ into its Gaussian part $G_k$ and its Poisson part $P_k$ and then separately estimating the extreme values of each. 
For the Gaussian part, the following result due to Talagrand will be crucial.

\begin{lemma}{(Talagrand, \cite[Thm.\ 2.4]{Tal94})} \label{lemma:tal}
Let $(G(x))_{x\in D}$ be a centered Gaussian field with a.s.\ continuous paths and let $\bar{\sigma}^2 = \sup_{x\in D} \boldsymbol{\mathsf  E}[G(x)^2]$.
Consider the canonical distance $d_c(x,y) := \ev{(G(x)-G(y))^2}^{1/2}$
on $D$ and let $N(D,d_c,\varepsilon)$ be the smallest number of $d$-open balls with $d_c$-radius $\varepsilon$ needed to cover $D$. 
Assume that for some constant $A > \bar{\sigma}$, some $v > 0$ and $0 \leq 
\varepsilon_0 \leq \bar{\sigma}$, the number $N(D,d_c,\varepsilon)$ is bounded above by $(A/\varepsilon)^v$ whenever $\varepsilon\in (0,\varepsilon_0)$.
	
Then there is a universal constant $K>0$ such that for $g \geq \bar{\sigma}^2 \big[ (1+\sqrt{v})/\varepsilon_0\big]$ we have
\begin{equation} \label{eq:tal}
	\prob \bigg( \sup_{x\in D}\lvert G(x) \rvert \geq g \bigg) 
	\leq 
	2 \,\bigg( \frac{KAg}{\sqrt{v}\bar{\sigma}^2} \bigg)^v 
	\Phi \Big( -\frac{g}{\bar{\sigma}} \Big) 
	\leq
	\,\left( \frac{KAg}{\sqrt{v}\bar{\sigma}^2} \right)^v \e^{-\frac{g^2}{2\bar{\sigma}^2}},
\end{equation}
where $\Phi$ denotes the CDF of the standard normal distribution.
If $\varepsilon_0 = \bar{\sigma}$, the condition on $g$ is $g \geq \bar{\sigma} \big[ 1 + \sqrt{v} \big]$.
\end{lemma}

We continue with a technical result which will be needed below.

\begin{lemma} \label{lemma:mat-hold}
Let $D\subset \mathbb{R}^d$ be open and bounded, $\alpha > d/2$ and $m > 0$. 
Then the following holds:
\begin{enumerate}[(i)]
\item 
For $0 < \eta < 2\alpha-d$ there exists $C=C(m,\eta,\alpha)>0$ such that for all $x,y \in \mathbb{R}^d$
\begin{equation*}
    \lvert k_{\alpha, m}(x) - k_{\alpha, m}(y) \rvert 
    \leq 
    C(m,\eta,\alpha) \, \lvert x - y \rvert^\eta.
\end{equation*}
\item 
If $\alpha > d/2$, then $|k_{\alpha,m}|$ is bounded, decreases like $f(x)=\e^{-m|x|}$, and the mapping $y \mapsto \sup_{x\in D}|\tau_{y} \left(k_{\alpha, m}(x)\right)| \in L^1(\mathbb{R}^d) \cap L^\infty(\mathbb{R}^d)$. 
\end{enumerate}
\end{lemma}
\begin{proof}
\begin{enumerate}[(i)]
\item 
For fixed $\eta \in (0,1)$ and all $z, w \in \mathbb{C}$ with $|z - w| \leq 2$ we have 
$|z - w| \leq 2^{1-\eta} |z - w|^\eta$; moreover, we have $| \e^{-i\xi \cdot x} - \e^{-i\xi \cdot y}| \leq 2$
for all $x, y, \xi \in \mathbb{R}^d$ and, by the mean value theorem,
$|\e^{-i\xi\cdot x} - \e^{-i\xi \cdot y}| \leq  |\xi \cdot (x-y)|$.
Combining these, we obtain
\begin{align*}
    | k_{\alpha, m}(x) - k_{\alpha, m}(y) | 
    &= 
    \frac{1}{(2\pi)^d}
    \left|
    \int_{\mathbb{R}^d} 
    \frac{\e^{-i\xi \cdot x} - \e^{-i\xi \cdot y}}
         {(| \xi |^2 + m^2)^\alpha}
    \, \d\xi\right| \\[.5em]
    &\leq 
    \frac{2^{1-\eta}}{(2\pi)^d} 
    |x - y|^\eta 
    \int_{\mathbb{R}^d} 
    \frac{|\xi|^\eta}{(|\xi|^2 + m^2)^\alpha} \, \d\xi.
\end{align*}
The last integral converges if $0<\eta < 2\alpha - d$.\vspace{.5em}
    
\item By applying the Hankel transform one can see that
\begin{equation*}
    \spF^{-1} (\hat k_{\alpha,m})(x) 
    =
    \frac{(|x|/m)^{\alpha - d/2} K_{\alpha - d/2}(|x|m)}
         {2^{\alpha -1}\Gamma(\alpha)(2\pi)^{d/2}}
\end{equation*}
where $K$ is the modified Bessel function of second kind. 
For a fixed $v > 0$, 
$K_v(|x|) \sim \frac{1}{2}\Gamma(v)(\frac{1}{2}|x|)^{-v}$ 
for $|x|\to 0$ and 
$K_v(|x|) \sim \sqrt{\pi/(2|x|)} \e^{-|x|}$ for $|x| \to \infty$. 
This implies that $|k_{\alpha, m}|$ is bounded and decreases as $\e^{-m|x|}$. 
Therefore since $D$ is relatively compact, 
$\sup_{x\in D} \left|\tau_y k_{\alpha,m}(x)\right|$ 
is bounded and exponentially decreasing as well, which implies the assertion.
\end{enumerate}
\end{proof}

The next result is formulated in a more general way than needed in this section. However, the general result will be used as stated in Section~\ref{section:approx} below. 
The following assumption will used repeatedly in the following. 
Recall that for a Borel measure $\nu$ on $\R\backslash\{0\}$ we denote by $\nu_+$ its image measure on $\R_+$ under $|\cdot|$. 
\begin{assumption} \label{assum:Levy}
Let $Z$ be a $\tripleBar{\cdot}$-continuous L\'evy field with characteristic triplet $(b, \sigma^2,\nu)$, such that $\nu$ is a L\'evy measure satisfying $\int_{\R\backslash\{0\}} |s|\,\nu(\d s)<\infty$ and $\int_\R(\e^{\beta s}-1)\nu_+(\d s)<\infty$ for some $\beta>0$.
\end{assumption}

The following proposition gives a Chernov-type exponential upper bound for the supremum over the Poisson part.

\begin{proposition}\label{prop:poisson_est}
Let $P$ be a compound Poisson field, i.e., a Lévy field with characteristic triplet given by $\left(\int_{\{0<|s|\leq 1\}} s\,\nu(\d s), 0,\nu\right)$ with a finite measure $\nu$, which satisfies Assumption~\ref{assum:Levy}.
Moreover, let $D\subseteq\R^d$ be open and bounded and let $k_\iota:\R^d\times\R^d\rightarrow\R, \iota\in I,$ be a family of smoothing functions such that with $\tilde{k}_\iota(y):=\sup_{x\in D}|k_\iota(x,y)|, y\in \R^d, \iota\in I$ the following conditions hold:
\begin{itemize}
\item[i)] 
$\forall\,\iota\in I:\tilde{k}_\iota\in L^1(\R^d)\cap L^\infty(\R^d)$.
\item[ii)] 
$\kappa_\infty:=\sup_{\iota\in I}\|\tilde{k}_\iota\|_{L^\infty(\R^d)}<\infty$ as well as $\kappa_1:=\sup_{\iota\in I}\|\tilde{k}_\iota\|_{L^1(\R^d)}<\infty$.
\end{itemize}
Then, for all $\iota\in I$ and $\tau\in (0,1)$ there holds for every $p>0$ 
\begin{align*}
\begin{split}
    \prob & \left(\sup_{x\in D} |P_{k_\iota}(x)|\geq p\right) \\
    &\leq
    \exp\left(
    \frac{\beta\kappa_1}{\kappa_\infty}
    \left(\e^\beta\int\limits_{\{0<s\leq 1\}}|s|\nu_+(\d s)
    +
    \frac{1}{\beta \e (1-\tau)}
    \int\limits_{\{s>1\}} 
    \e^{\beta s}\nu_+(\d s)\right)\right)
    \e^{-\frac{\beta}{\kappa_\infty}\tau p}.
\end{split}
\end{align*}
\end{proposition}

\begin{proof}
For $\iota\in I$ we define $\kappa_\iota:=\|\tilde{k}_\iota\|_{L^\infty(\R^d)}$ as well as
\[
   f_\iota:(0,\infty)\rightarrow[0,\infty], 
   \quad
   f_\iota(\vartheta)
   :=
   \int_{\R^d}\int_{\R_+} 
   (\e^{\vartheta s \tilde{k}_\iota(y)}-1)\nu_+(\d s)\d y
\]
and
\[
   \theta_\iota:(0,\infty)\rightarrow\R\cup\{\infty\},
   \quad
   \theta_\iota(p):=\sup_{\vartheta>0}\vartheta p-f_\iota(\vartheta).
\]
    Then $f_\iota$ is a convex increasing function and $\theta_\iota$ is its Legendre transform (Fenchel transform, conjugate function).
    
    With the notation from Remark \ref{rel:levy_pos} and a calculation analogous to \cref{eqa:CF_calc}, for $\vartheta>0$ we obtain, abbreviating $P_\iota(x):=P_{k_\iota}(x), \iota\in I, x\in D,$
	\begin{eqnarray*}
		\boldsymbol{\mathsf  E}[\e^{\vartheta \sup_{x\in D}\vert
			P_\iota(x) \rvert}]&\leq&  \boldsymbol{\mathsf  E}[\e^{\vartheta \sup_{x\in D}\lvert P\rvert_{|k_\iota|}(x) }] \leq \boldsymbol{\mathsf  E}[\e^{\vartheta \sum_j \sum_{l=1}^{N_{\Lambda_j}}|S_l^{(j)}|\tilde{k}_\iota(X_l^{(j)})}]\\
			&=&\e^{\int_{\mathbb{R}^d} \int_{\mathbb{R}_+}(\e^{\vartheta s \tilde{k}_\iota(y) }-1)\, \nu_+(\d s)\, \d y}.
	\end{eqnarray*}
	Applying Markov's inequality, this yields for $p>0$
	\begin{eqnarray}\label{estimate:Poisson part}
		\prob\left(\sup_{x \in D} \lvert P_\iota(x) \rvert \geq p\right) &=&
		\inf_{\vartheta>0}\prob\left(\e^{\vartheta \sup_{x \in D} \lvert P_\iota(x) \rvert} \geq \e^{\vartheta p}\right) \leq
		\inf_{\vartheta>0}\frac{\boldsymbol{\mathsf  E}[\e^{\vartheta \sup_{x\in D}\vert P_\iota(x) \rvert}]}{\e^{\vartheta p}}\nonumber\\
		&\leq& \inf_{\vartheta >0} \, \e^{\int_{\mathbb{R}^d} \int_{\mathbb{R}_+}(e^{\vartheta s \tilde{k}_\iota(y) }-1)\, \nu_+(\d s)\, \d y - \vartheta p}\\ 
		&\leq& \, \e^{-\sup\limits_{\vartheta>0}\{\vartheta p - \int_{\mathbb{R}^d} \int_{\mathbb{R}_+}(\e^{\vartheta s \tilde{k}_\iota(y) }-1)\, \nu_+(\d s)\, \d y\} }
		= \e^{-\theta_\iota(p)}.\nonumber
	\end{eqnarray}
Using the hypothesis on $\nu_+$, for $0\leq \vartheta < \frac{\beta}{\kappa_\iota}$ we derive
\begin{align} \label{eq:est_poisson_part}
\begin{split}
	    f_\iota(\vartheta)
	    &=
	    \int_{\mathbb{R}^d}\int_{\mathbb{R}_+}
	    \left(
		\e^{\vartheta s \tilde{k}_\iota(y) }-1
		\right)
		\, \nu_+(\d s)\, \d y \\
	    &=
	    \int_{\mathbb{R}^d}
	    \left(
	    \int_{\{0< s \leq 1\}}
	    +\int_{\{s>1\}}
	    \right)
	    \left(	\e^{\vartheta s \tilde{k}_\iota(y) }-1	\right)
			\, \nu_+(\d s) \d y \\
		&\leq
	        \int_{\mathbb{R}^d}
	        \left(
	        \int_{\{0< s \leq 1\}}
	        +\int_{\{s>1\}}
	        \right)
	        \e^{\vartheta s \tilde{k}_\iota(y) }\vartheta s \tilde{k}_\iota(y)
			\, \nu_+(\d s) \d y \\
	    &\leq 
	    \vartheta \lVert \tilde{k}_\iota \rVert_{ L^1(\mathbb{R}^d)} \left(\e^{\vartheta \kappa_\iota}
			\int_{\{0 < s \leq 1 \}}
			\lvert s \rvert \, \nu_+(\d s)
			+  \int_{\{s>1\}}s\e^{\vartheta s\kappa_\iota}
			\, \nu_+(\d s)\right)\\
		&= 
	    \vartheta \lVert \tilde{k}_\iota \rVert_{ L^1(\mathbb{R}^d)} \left(\e^{\vartheta \kappa_\iota}
			\int_{\{0 < s \leq 1 \}}
			\lvert s \rvert \, \nu_+(\d s)
			+  \int_{\{s>1\}}\e^{\beta s} s \e^{-(\beta-\vartheta \kappa_\iota)s}
			\, \nu_+(\d s)\right)\\
		&\leq \vartheta\kappa_1 \left(\e^{\vartheta \kappa_\iota}
			\int_{\{0 < s \leq 1 \}}
			\lvert s \rvert \, \nu_+(\d s)
			+  \frac{1}{(\beta-\vartheta\kappa_\iota)\e}\int_{\{s>1\}}\e^{\beta s}
			\, \nu_+(\d s)\right),
\end{split}
\end{align}
where in the last step we have used the elementary fact that for $\alpha>0$ we have
$\max_{s>0}s \e^{-\alpha s}=\frac{1}{\alpha \e}$. 
Thus, $f_{\iota|[0,\beta/\kappa_\iota)}$ is finite and for arbitrary $\tau\in (0,1)$ we obtain from (\ref{eq:est_poisson_part}) and $\kappa_\infty\geq\kappa_\iota$ for $\vartheta=\tau\frac{\beta}{\kappa_\infty}$ from the definition of $\theta_\iota$
\begin{align*}
\begin{split}
	\theta_\iota(p)
	&\geq 
	\tau\frac{\beta}{\kappa_\infty}p-f_\iota\left(\tau\frac{\beta}{\kappa_\infty}\right)\\
	&\geq 
	\tau\frac{\beta}{\kappa_\infty}p
	-\tau\frac{\beta}{\kappa_\infty}\kappa_1
	\left(\e^{\tau\beta} \int_{\{0<s\leq 1\}} |s| \, \nu_+(\d s)
	+  \frac{1}{(\beta-\tau\beta\frac{\kappa_\iota}{\kappa_\infty})\e}
	\int_{\{s>1\}}\e^{\beta s} \, \nu_+(\d s)\right)
\end{split}    
\end{align*}
for every $p>0$. 
The assertion now follows by combining the previous inequality with \eqref{estimate:Poisson part}.
\end{proof}

We are finally ready to present this section's main result. 

\begin{theorem} \label{thm:levy_int}
Let the Lévy field $Z$ satisfy Assumption~\ref{assum:Levy}.
Moreover, let $k:\R^d\times\R^d\rightarrow\R, k(x,y):=k_{\alpha,m}(x-y)$ with $2\alpha>d$ and let $T$ be locally Lipschitz such that for $h\in [0,1]$, $B,\rho>0$ we have $B^{-1} \e^{-\rho |z|^h}\leq T(z)\leq B \e^{\rho|z|^h}$ for all $z\in\R$. 

Then, for the solution $u$ of the random boundary value problem \eqref{eq:diff_eq} with random conductivity function $a=T\circ Z_k$ we have $u \in L^n(\Omega; H^1(D))$, for any $n \in \mathbb{N}$ if $h < 1$ and for $n < \beta / 2\kappa\rho$ if $h=1$, where 
$\kappa := \sup_{x\in D,y\in \R^d}|k_{\alpha,m}(x-y)|$.

In particular, all moments of $u$ exist if $h\leq 1$ and $\int_{\mathbb{R}_+} (\e^{\beta s}-1)\, \nu_+(\d s) < \infty$ for all $\beta> 0$.   
\end{theorem}

\begin{proof}
	We first show that without loss of generality, we may assume that $Z$ has the characteristic triplet $(b',\sigma^2,\nu)$ with $b':=\int_{\{0<|s|\leq 1\}} s\,\nu(ds)$. 
	Indeed, $(b',\sigma^2,\nu)$ is a characteristic triplet whose associated L\'evy noise field $\tilde{Z}$ is $\tripleBar{\cdot}$-continuous by Proposition \ref{prop:levy}. 
	Moreover, for arbitrary $\alpha\in \R$, $T_\alpha(z):=T(z+\alpha)$ is locally Lipschitz, and with $\tilde{\rho}:=\max\{1, 2^{h-1}\}\rho$, $\tilde{B}:=B e^{\tilde{\rho}|\alpha|^h}$ we have
	\[
	   \tilde{B}^{-1}e^{-\tilde{\rho}|z|^h}
	   \leq 
	   T_\alpha(z)
	   \leq 
	   \tilde{B} e^{\tilde{\rho}|z|^h}.
	\]
	For the special case $\alpha_k:=(b-b')\int_{\R^d}k(y)\,dy$ we obtain $a=T\circ Z_k=T_{\alpha_k}\circ\tilde{Z}_k$. 
	Hence, replacing $T$ by $T_{\alpha_k}$ and $Z$ by $\tilde{Z}$, we may indeed assume that $Z$ has the characteristic triplet $(b',\sigma^2,\nu)$. 
	Therefore, we have $Z=G+P$, where $G$ is the $\tripleBar{\cdot}$-continuous generalized centered Gaussian field with characteristic triplet $(0,\sigma^2,0)$ and $P$ is the $\tripleBar{\cdot}$-continuous L\'evy field with characteristic triplet $(b',0,\nu)$.
	
Let $d_c$ be the canonical distance of the centered Gaussian field $(G_k(x))_{x\in D}$ which has almost surely continuous paths by Theorem \ref{thm:matern}. 
We fix $\eta\in (0,2\alpha-d)$ as well as $a>\mbox{diam}(D)$ and set $\bar{\sigma}^2:=\sup_{x\in D}\ev{G_k(x)^2}=\sigma^2\|k\|_{L^2(\R^d)}^2$.
	
With the aid of Lemma \ref{lemma:mat-hold} i), for a suitable constant $C_1=C_1(m,\eta, 2\alpha)>0$, we have for arbitrary $x,y\in D$
\begin{equation} \label{Gaussian covariance Hoelder}
    \begin{split}
	    d(x,y)^2
	    &=
	    \Var(G_k(x)-G_k(y)) \\
	       &=
	    \Var(G_k(x))-\Var(G_k(y))-2\Cov(G_k(x),G_k(y)). \\
	    &=
	    2\sigma^2(k_{2\alpha,m}(0)-k_{2\alpha,m}(x-y))
	    \leq 
	    2\sigma^2 C_1|x-y|^\eta.
    \end{split}
    \end{equation}
    Then, with $C'^2:=2\sigma^2 C_1$, we have for all $\varepsilon > 0$ and $x \in \mathbb{R}^d$
    \begin{equation}\label{inclusion of balls}
    \begin{split}
	    B_{\lvert \cdot \rvert, (\frac{\varepsilon^2}{C'^2})^{\frac{1}{\eta}}}(x) 
	    &:= \
	    \left\{ y \in \mathbb{R}^d\ : \ \lvert x - y \rvert < 
			\Big(\frac{\varepsilon^2}{C'^2}\Big)^{\frac{1}{\eta}} \right\}\\ 
	    &\subseteq \
	    \{ y \in \mathbb{R}^d\ : \ d_c(x,y) < \varepsilon \} 
	    =: B_{d_c, \varepsilon}(x).
    \end{split}
    \end{equation}
Since $D$ is bounded, we can cover $D$ with a finite number $N$ of open balls  $B_{\lvert \cdot \rvert, (\frac{\varepsilon^2}{C'^2})^{\frac{1}{\eta}}}(x)$. By the choice of $a$, this number $N$ is bounded by $(C'^{\frac{2}{\eta}}a/\varepsilon^{\frac{2}{\eta}})^d=(C' a^{\eta/2}/\varepsilon)^{2d/\eta}$. By \eqref{inclusion of balls} we thus obtain for all $\varepsilon>0$
\begin{equation*}
	N(D, d_c,\varepsilon) 
	\leq 
	(C' a^{\eta/2}/\varepsilon)^{2d/\eta},
\end{equation*}
so that $d_c$ satisfies the covering property of Talagrand's Lemma \ref{lemma:tal} with $v:=2d/\eta$ and $A:=\max\{C'a^{\eta/2},\bar{\sigma}+1\}$ for every $\varepsilon>0$. 
Thus, by Talagrand's Lemma \ref{lemma:tal}, with the universal constant $K>0$, for every $g\geq \bar{\sigma} (1 + \sqrt{v})$
\begin{equation}\label{estimate Gaussian part}
	\prob\left(\sup_{x\in D}|G_k(x)|\geq g\right)
	\leq 
	\left( \frac{KA g}{\sqrt{v}\bar{\sigma}^2} \right)^v \e^{-\frac{g^2}{2\bar{\sigma}^2}}.
\end{equation}
Next, we observe that by Lemma \ref{lemma:mat-hold} ii), we have 
\[
   \tilde{k}(y)
   :=
   \sup_{x\in D}|k(x,y)| = \sup_{x\in D}|k_{\alpha,m}(x-y)| \in L^1(\R^d)\cap L^\infty(\R^d),
\]
so that by Proposition~\ref{prop:poisson_est} applied to the family of smoothing functions consisting only of $k$, that for arbitrary $\zeta\in (0,1)$ there is a constant $C_\zeta>0$ depending only on 
$\zeta$, $\|\tilde{k}\|_{L^1(\R^d)}$, $\|\tilde{k}\|_{L^\infty(\R^d)}$, $\beta$, and $\nu$ 
such that for every $p>0$
\begin{equation} \label{estimate Poisson part}
	    \prob\left(\sup_{x\in D}|P_k(x)|\geq p\right)
	    \leq 
	    C_\zeta \e^{-\frac{\beta}{\kappa}(1-\zeta)p}.
\end{equation}
Taking into account that $Z=G+P$, it follows from Lemma \ref{lemma:p-est} together with (\ref{estimate Gaussian part}) and (\ref{estimate Poisson part}) that for every $\zeta\in (0,1)$ we have with 
	$D_\zeta:=\max\{\left(\frac{K A}{\bar{\sigma}^2\sqrt{v}}\right)^v,C_\zeta,1\}$
\begin{eqnarray} \label{est:moments}
	\boldsymbol{\mathsf  E}[\lVert u \rVert_{H^1(D)}^n] 
	&\leq&
	\tilde{C}^n2^{n-1}(B^n+B^{2n}) 
  	\, \sum_{z=0}^\infty \e^{2n\rho(z+1)^h} \, \prob\left(\sup_{x\in D}\lvert Z_k(x) \rvert \geq z\right)\nonumber\\
	&\leq \ 
	&\tilde{C}^n2^{n-1}(B^n+B^{2n}) D_\zeta \Bigg\{ \sum_{z=0}^{\left \lfloor{\bar{\sigma}[1+\sqrt{v}]/\zeta}\right \rfloor } \e^{2n\rho(z+1)^h}\\
	&&+ 
	\sum_{z=\left \lfloor{\bar{\sigma}[1+\sqrt{v}]/\zeta}\right \rfloor+1 }^\infty \e^{2n\rho(z+1)^h}\left( z^v \e^{-\frac{\zeta^2z^2}{2\bar{\sigma}^2}} + \e^{-\frac{\beta}{\kappa}(1-\zeta)z} \right) \Bigg\}.\nonumber
\end{eqnarray}
Thus, in case $h < 1$ the above series converges. 
In case $h=1$, the above series converges if $n < (1-\zeta)\beta / 2\kappa\rho$. Hence, by choosing $\zeta$ sufficiently close to zero, in case $h=1$ the series converges for all $n <\beta / 2\kappa\rho$.
\end{proof}

\begin{remark}
\hspace{2ex}
\begin{itemize}
	\item[(i)] 
	By Theorem \ref{thm:levy_int}, in the case of $h = 1$ we get all moments up to an order that depends on $\beta$. 
	The larger $\beta$ is, the more moments $u$ has w.r.t. the Sobolev norm.
	\item[(ii)] 
	If we assume existence of the Laplace transform for $\nu$, $\int_{\mathbb{R}_+} \e^{\beta s}\, \nu_+(\d s) < \infty$ for some $\beta > 0$, we exclude  noises with infinite activity like Gamma noise. 
	We therefore employ the more general condition 
	$\int_{\mathbb{R}_+} (\e^{\beta s}-1)\, \nu_+(\d s) < \infty$.
	\item[(iii)] 
	In the special case where the smoothed Lévy noise field $Z_k$ is a Gaussian field without a compound Poisson noise component, we have $\theta(p)=\infty$ for all $p>0$ so that (\ref{est:moments}) gives us the existence of all moments if $h< 2$. Moreover, in case of $h=2$, we then obtain the existence of moments of order $n<1/(4\rho\sigma^2\|k\|_{L^2 (\R^d)}^2)$. 
	This improves \cite{Charr12}, where this result was shown for $h = 1$.
	\end{itemize}
\end{remark}

\section{Approximability of Solutions of the Random Diffusion Equation} \label{section:approx}

In this section we approximate the random diffusion coefficient $a$ in \eqref{eq:diff_eq} by a finite modal expansion, thus reducing the coefficient from an infinite-dimensional Lévy random field to a finite-dimensional Lévy random vector.
We prove that, under similar assumptions as for integrability, solutions of the diffusion equation with approximate diffusion coefficient converge in the Bochner space $L^n(\pspace;H^1(D))$ to that of the original equation. 
The remaining problem of the quadrature of high-dimensional uncorrelated, but possibly not independent, L\'evy distributions is left for future work. 

\subsection{Dependence on Random Coefficient}

Before we can give convergence results, we need to control the change in the solution $u$ that stems from a change in the coefficients. 
This change e.g.\ can be due to a finite-dimensional approximation of Karhunen-Loève type, as  will be the case below in Section~\ref{sec:KL}. 
The results in this subsection are of independent interest and can be used, e.g., to control the error in statistical estimation of the law and smoothing function of the random field.  Also, the results easily generalize to arbitrary continuous random fields and differentiable transformations $T(z)$ which are exponentially bounded from below and above.

Consider a smoothed L\'evy random field $Z_k$ with continuous paths and smoothing function $k:\R^d\times\R^d\rightarrow\R$. 
Let $k = k_N + r_N$ be any decomposition of the smoothing function $k$ such that 
$\lim_{N\to\infty} k_N(x,\cdot) = k(x,\cdot)$ with respect to $\tripleBar{\cdot}$ for every $x\in\R^d$. 
We define the random field $Z_N(x) := Z_{k_N}(x)$, $N\in \N$, to be an approximation to $Z_k(x)$ and $R_N(x) := Z_{r_N}(x)$ the corresponding remainder such
that $Z_k(x)=Z_{N}(x)+R_N(x)$. 
We assume that $Z_{N}(x)$ has continuous paths on $\bar D$ and consequently so does $R_N(x)$. 
This yields an approximating diffusion coefficient $T(Z_{N}(x))$ in equation \eqref{eq:diff_eq} with associated random solution $u_{N}$ to the corresponding weak problem.

To prove convergence of the weak solution $u_N\to u$ in $L^n(\pspace,H^1(D))$, $n \in \mathbb N$, as $N\to\infty$, we will derive an estimate based on an interpolated diffusion
equation with diffusion coefficient $T(Z_{N,t}(x))$ where $Z_{N,t}(x) := Z_{k_N+tr_N}(x) = Z_N(x) + t R_N(x)$ with $t \in [0,1]$. 
The resulting weak form of equation \cref{eq:diff_eq} with approximating diffusion coefficient and homogenized Dirichlet boundary conditions with weak solution $u_{0_{N,t}}\in H^1_D(D)=\{v\in H^1(D);\,\trace(v)=0\text{ on }\partial_D\}$ is characterized by
\[
b_{N,t}(u_{0_{N,t}},v) = \ell_{N,t}(v) \quad  \forall v\in H^1_D(D),
\]
with
\[
b_{N,t}(u,v) 
:= 
\int_D 
T(Z_{N,t}(x))\nabla u(x) \cdot \nabla v(x) \, \d x,
\quad
u,v \in H^1_D(D),
\]
and
\[
\ell_{N,t}(v)
:=
\int_D \left[ f(x)v(x) - T(Z_{N,t}(x))\nabla Eg_D(x) \cdot \nabla v(x) \right]\,\d x
+
\int_{\Gamma_N}g_N(x)v(x)\,\d\sigma,
\]
where $Eg_D\in H^1(D)$ is an extension of $g_D$. 
The weak solution of \eqref{eq:diff_eq} with inhomogeneous Dirichlet boundary condition then has the form $u_{N,t} = u_{0_{N,t}} + Eg_D$.
From now on we additionally assume that the transformation $T$ is continuously differentiable. 
It can then be shown that $t\mapsto u_{0_{N,t}}$ (and thus $t\mapsto u_{N,t}=u_{0_{N,t}}+Eg_D$) is differentiable with respect to the weak topology \cite{bittner,Schwartz1954}.  
We denote the derivative by $\dot{u}_{0_{N,t}}$ and $\dot{u}_{N,t}$, respectively. 
Moreover, setting 
\begin{eqnarray*}
	\dot{b}_{N,t}(u,v)
	&:=&
	\int_D T'(Z_{N,t}(x)) R_N(x)\nabla u \cdot \nabla v(x) \d x\\
	\dot{\ell}_{N,t}(v)
	&:=&
	-\int_D T'(Z_{N,t}(x)) R_N(x) \nabla Eg_D(x) \cdot \nabla v(x) \,\d x
\end{eqnarray*}
for $u,v \in H^1_D(D)$, one can show that
\begin{equation} \label{eq:Sensitivity}
	b_{N,t}(\dot{u}_{0_{N,t}},v) = \dot{\ell}_{N,t}(v)- \dot{b}_{N,t}(u_{0_{N,t}},v)
	\qquad
	\forall v\in H^1_D(D).
\end{equation}
As one can show using \eqref{eq:Sensitivity}, $t\mapsto \dot{u}_{N,t}$ is continuous with respect to the strong $H^1$-topology, so we conclude that $t\mapsto u_{N,t}$ is also differentiable with respect to the strong topology and has the derivative $\dot{u}_{N,t}$ \cite{bittner}.   

\begin{lemma} \label{lemma:app-a-prio}
	Let $Z$ be a $\tripleBar{\cdot}$-continuous generalized random field and $k:\R^d\times\R^d\rightarrow\R$ a smoothing function such that $(Z_k(x))_{x \in D}$ has a.s.\ continuous paths. 
	Moreover, let $Z_{N,t}$ be as above and assume that $T$ is continuously differentiable. 
	Let $u_{N,t}$ be the solution to \eqref{eq:diff_eq} with random conductivity $a := T\circ Z_{N,t}$.
	Then there holds
	\begin{eqnarray}\label{eqa:sensSol}
		\| \dot{u}_{N,t} \|_{H^1(D)} 
		&\leq& 
		C \,\sup_{x\in D} |T'(Z_{N,t}(x))| \, \sup_{x\in D} |R_w(x)| 
		\bigg(
		\frac{1+\sup_{x\in D} |T(Z_{N,t}(x))|}
		{(\inf_{x\in D} | T(Z_{N,t}(x)) |)^2}	\nonumber\\
		&&+ \frac{1}{\inf_{x\in D} | T(Z_{N,t}(x)) |}\bigg)
		\big(\|f\|_{L^2(D)} + \|g_D\|_{H^{1/2}(\partial_D)}+\|g_N\|_{H^{-1/2}(\partial_N)}\big),
	\end{eqnarray}
	where $C=(1+C_P^2)^2\max\{1,2\|E\|,\|\trace\|\}$ with $E:H^{1/2}(\partial_D)\rightarrow H^1(D)$ denoting an extension operator, $\trace:H^1(D)\rightarrow H^{1/2}(\partial D)$  the trace operator, and where $C_P>0$ only depends on $D$ and $\partial_D$.
\end{lemma}

\begin{proof}
As noted above, we can write the weak solution to \cref{eq:diff_eq} as $u_{N,t} = u_{0_{N,t}} + Eg_D$ and therefore we have 
	$\|\dot{u}_{N,t}\|_{H^1(D)} =  \| \dot{u}_{0_{N,t}} \|_{H^1(D)}$.
Setting $\tilde{C}:=(1+C_P^2)\max\{1,2\|E\|,\|\mbox{tr}\|\}$, we combine (a generalization of) Poincar\'e's inequality (cf.\ \eqref{equivalence of Sobolev norms}), \eqref{eq:Sensitivity}, the definition of $\dot{b}_{N,t}, \dot{\ell}_{N,t}$, as well as inequality \eqref{eq:a-prio} to obtain
\begin{eqnarray*}
		&&\frac{\inf_{x\in D}|T(Z_{N,t}(x))|}{(1+C_P^2)}\|\dot{u}_{0_{N,t}}\|^2_{H^1(D)}
		\leq 
		b_{N,t}(\dot{u}_{0_{N,t}},\dot{u}_{0_{N,t}})
		=
		|\dot{\ell}_{N,t}(\dot{u}_{0_{N,t}})-\dot{b}_{N,t}(u_{0_{N,t}},\dot{u}_{0_{N,t}})|\\
		&\leq& 
		\int_D |T'(Z_{N,t}(x))R_N(x) 
		\nabla (Eg_D+u_{0_{N,t}})(x) \cdot \nabla\dot{u}_{0_{N,t}}(x)| \,\d x\\
		&\leq&
		\sup_{x\in D}|T' (Z_{N,t}(x))| \sup_{x\in D}|R_N(x)|
		\,\|\dot{u}_{0_{N,t}}\|_{H^1(D)}\|u_{0_{N,t}} + Eg_D\|_{H^1(D)}\\
		&\leq&
		\sup_{x\in D}|T'(Z_{N,t}(x))| \sup_{x\in D}|R_N(x)|\,
		\|\dot{u}_{0_{N,t}}\|_{H^1(D)} 
		\Big(\|E\|\,\|g_D\|_{H^{1/2}(\partial_D)}\\
		&&+\,
		\tilde{C}\,
		\frac{1 + \sup_{x\in D}|T(Z_{N,t}(x))|}
		{    \inf_{x\in D}|T(Z_{N,t}(x))|}
		\big(
		\|f\|_{L_2(D)} + \|g_D\|_{H^{1/2}(\partial_D)}+\|g_N\|_{H^{-1/2}(\partial_N)}
		\big)\Big).
\end{eqnarray*}
The assertion \eqref{eqa:sensSol} now follows on dividing by 
$\frac{\inf_{x\in D}|T(Z_{N,t}(x))|}{(1+C_P^2)} \| \dot u_{0,t}\|_{H^1(D)}$.
\end{proof}

Using the above lemma, we next see that we can estimate the effect of the perturbation $R_N$ by a term that is exponentially growing in the extreme values of $Z_{N,t}(x)$ and a moment in the perturbation. 
The following assumption on the function $T$ will be used repeatedly in the following.

\begin{assumption} \label{assum:T}
For the continuously differentiable function $T:\mathbb R \to \mathbb {R_+}$ there exist $\rho, B>0, h\in (0,1]$ such that for all $z \in \mathbb R$ there holds
\begin{equation} \label{eqa:consitions_T_Tdot}
	B^{-1} \e^{-\rho |z|^h}
	\leq 
	T(z)
	\leq  
	B \e^{\rho |z|^h}
	\qquad \text{ and } \qquad 
	|T'(z)| \leq  B \e^{\rho |z|^h}.
\end{equation}
\end{assumption}

\begin{lemma} \label{lemma:prob_error_Estimate}
Under the assumptions of Lemma~\ref{lemma:app-a-prio} in addition to Assumption~\ref{assum:T} we have that for any $\varrho > 1$ and $\frac{1}{\varrho}+\frac{1}{\varrho'}=1$, $n\geq 1$,
	\begin{equation} \label{eqa:est_by_moment}
		\ev{\| u - u_N \|^n_{H^1(D)}}
		\leq 
		\bar C
		\;
		\ev{\sup_{x\in D} |R_N(x)|^{n\varrho'}}^{\frac{1}{\varrho'}}
		\sup_{t\in[0,1]}
		\ev{{\e}^{4\varrho\rho n\sup_{x\in D} |Z_{N,t}(x)|}}^{\frac{1}{\varrho}}.
	\end{equation}
	where $\bar C = C^n(B^2+B^3)^n(\|f\|_{L^2(D)}+\|g_D\|_{H^{1/2}(\partial_D)}+\|g_N\|_{H^{-1/2}(\partial_N)})^n$ with $C$ from Lemma~\ref{lemma:app-a-prio}.
\end{lemma}

\begin{proof}
	Using the properties of the Bochner integral for Banach space-valued functions and Jensen's inquality for the ordinary integral over $[0,1]$, we obtain
	\begin{align*}
		\ev{\| u-u_N \|^n_{H^1(D)}}
		&=
		\ev{\left\| \int_0^1\dot u_{N,t} \, \d t \right\|^n_{H^1(D)}}
		\leq 
		\ev{\left(\int_0^1\left\|\dot u_{N,t} \right\|_{H^1(D)}\, \d t\right)^n}\\
		&\leq 
		\ev{ \int_0^1\left\|\dot u_{N,t} \right\|_{H^1(D)}^n\, \d t }
		= 
		\int_0^1 \ev{ \left\|\dot u_{N,t} \right\|_{H^1(D)}^n } \, \d t,
	\end{align*}
	where we made use of the fact that we can interchange the order of integration for nonnegative integrands. 
	Applying now Lemma \ref{lemma:app-a-prio} and H\"older's inequality, we easily obtain \eqref{eqa:est_by_moment}.
\end{proof}

\subsection{Convergence of Solution Moments} \label{sub:suff_conditions}

From Lemma~\ref{lemma:prob_error_Estimate} it is clear that establishing convergence $u_N\to u$ in $L^n(\pspace,H^1(D))$ can be based on two steps: 
\begin{itemize}
\item[(i)] 
A bound on the Laplace transform 
$
\ev{{\e}^{4\varrho\rho n\sup_{x\in D}| Z_{N,t}(x)|}}
$    
of the extreme value of $Z_{N,t}$ which is uniform in $N$ and $t$ and
\item[(ii)] 
a proof that $\ev{\sup_{x\in D}| R_N(x)|^{n\varrho'}}\to 0$ as $N\to\infty$.
\end{itemize}
In this subsection we identify suitable conditions on $k$ and $k_N$ which imply (i) and (ii). 
We then apply this to the natural generalization of the Karhunen-Loève expansion to smoothed L\'evy fields.
We first prove a uniform version of Talagrand's Lemma \ref{lemma:tal}.

\begin{proposition}\label{prop:gauss_est}
	Let $G$ be a generalized centered Gaussian field, i.e.\ a $\tripleBar{\cdot}$-continuous L\'evy field with characteristic triplet $(0,\sigma^2,0), \sigma>0$. Moreover, let $D\subseteq\R^d$ be open and bounded and let $k_\iota:\R^d\times\R^d\rightarrow\R, \iota\in I,$ be a family of smoothing functions such that for another smoothing function $k:\R^d\times\R^d\rightarrow\R$ the following hold:
	\begin{itemize}
		\item[i)] $\forall\,\iota\in I:\,\sup_{x\in D}\|k_\iota (x,\cdot)\|_{L^2(\R^d)}\leq \sup_{x\in D}\|k(x,\cdot)\|_{L^2(\R^d)}$.
		\item[ii)] The canonical distances $d_\iota, \iota\in I$ and $d_c$ of the centered Gaussian fields $(G_{k_\iota}(x))_{x\in D}, \iota\in I$, and $(G_k(x))_{x\in D}$, respectively, satisfy $d_\iota\leq d_c, \iota\in I,$ and $d_c$ satisfies the covering property of Talagrand's Lemma \ref{lemma:tal}.
		\item[iii)] The centered Gaussian fields $(G_{k_\iota}(x))_{x\in D}, \iota\in I$, and $(G_k(x))_{x\in D}$ all have almost surely continuous paths.
	\end{itemize}
	Then, with $\bar{\sigma}_\iota^2:=\sigma^2\sup_{x\in D}\|k_\iota(x,\cdot)\|_{L^2(\R^d)}$ and $\bar{\sigma}^2:=\sigma^2\sup_{x\in D}\|k(x,\cdot)\|_{L^2(\R^d)}$, there are constants $A>\bar{\sigma}^2, K>0,v,\gamma>0, \varepsilon_0\in (0,\bar{\sigma})$ such that for all $\iota\in I$
	\begin{align}\label{estimate:gaussian part}
		\forall\, g>\bar{\sigma}_\iota(1+\sqrt{v})/\varepsilon_0:\,\prob\left(\sup_{x\in D}|G_{k_\iota}(x)|\geq g\right)&\leq \left(\frac{K A g}{\sqrt{v}\bar{\sigma}_\iota^2}\right)^v\exp\left(-\frac{g^2}{2\bar{\sigma}_\iota^2}\right)\\
		&\leq \gamma\left(\frac{K A g}{\sqrt{v}\bar{\sigma}^2}\right)^v\exp\left(-\frac{g^2}{2\bar{\sigma}^2}\right)\nonumber.
	\end{align}
\end{proposition}

\begin{proof}
	We abbreviate $G_\iota(x):=G_{k_\iota}(x),\iota\in I, x\in D$. Since $k$ is a smoothing function, it follows that $\bar{\sigma}^2<\infty$ and due to hypothesis i) we have for all $\iota\in I$
	\begin{eqnarray*}
		\bar{\sigma}^2_\iota&:=&\sup_{x\in D}\ev{\left[G_\iota(x)\right]^2}=\sup_{x\in D}\sigma^2\|k_\iota(x,\cdot)\|_{L^2(\R^d)}\leq\sup_{x\in D}\sigma^2\|k(x,\cdot)\|_{L^2(\R^d)}\\
		&=&\sup_{x\in D}\ev{\left[G_k(x)\right]^2}=\bar{\sigma}^2
	\end{eqnarray*}
	Hypothesis ii) implies that the $d_c$-ball centered at $x\in\R^d$ with $d_c$-radius $\varepsilon>0$ is contained in the $d_\iota$-ball centered at $x$ with $d_\iota$-radius $\varepsilon$, so that in the notation of Talagrand's Lemma $N(D,d_\iota,\varepsilon)\leq N(D,d_c,\varepsilon)$. 
	Again by hypothesis ii) there are thus $A>\bar{\sigma}^2, v>0, \varepsilon_0\in (0,\bar{\sigma})$ such that
	$$\forall\,\iota\in I, \varepsilon\in (0,\varepsilon_0):\,N(D,d_\iota,\varepsilon)\leq \left(\frac{A}{\varepsilon}\right)^v.$$
	Since all Gaussians fields considered have almost surely continuous paths, applying Talagrand's Lemma once more, there is $K>0$ such that for all $\iota\in I$ and $g>\bar{\sigma}^2(1+\sqrt{v})/\varepsilon_0$
	\begin{eqnarray*}
		\prob\left(\sup_{x\in D}|G_\iota(x)|\geq g\right)&\leq&\left(\frac{KAg}{\sqrt{v}\bar{\sigma}_\iota^2}\right)^v\exp\left(-\frac{g^2}{2\bar{\sigma}_\iota^2}\right)\\
		&\leq&\left(\frac{KAg}{\sqrt{v}\bar{\sigma}^2}\right)^v\exp\left(-\frac{g^2}{2\bar{\sigma}^2}\right)\left(\frac{\bar{\sigma}^2}{\bar{\sigma}_\iota^2}\right)^v\exp\left(-\bar{\sigma}^2\frac{(1+\sqrt{v})^2}{2\varepsilon_0}\left(\frac{\bar{\sigma}^2}{\bar{\sigma}_\iota^2}-1\right)\right).
	\end{eqnarray*}
	Since $f:[1,\infty)\rightarrow\R, f(x):= x^v\exp(-\bar{\sigma}^2\frac{(1+\sqrt{v})^2}{2\varepsilon_0}(x-1))$ is bounded above, the claim follows from the previous inequality with $\gamma:=\sup_{x\geq 1}f(x)$.
\end{proof}

Following the two-step approach outlined above, we begin with a uniform estimate of the Laplace transform of the extreme values of $Z_{N,t}$ under suitable assumptions on the smoothing kernel $k$.
To this end, we define for bivariate kernel functions $k=k(x,y)$
\[
    \tilde k(y) := \sup_{x \in D} |k(x,y)|, \qquad y \in \R^d.
\]
\begin{definition} \label{def:unif_orth_dec}
  We say that a smoothing function $k : \R^d \times \R^d \to \R$ has an \emph{orthogonal approximation sequence} $k = k_N + r_N$, $N\in\N$, if $k_N$ and $r_N$ are smoothing functions with
\begin{itemize}
\item[(i)] 
$\int_{\R^d} k_N(x_1,y) r_N(x_2,y) \, \d y = 0$ for $x_1, x_2\in \R^d$;
\item[(ii)] 
$\max\{\|\tilde r_N\|_{L^1(\R^d)},\kappa_{r,N}\}\to 0$ as $N\to\infty$ where $\kappa_{r,N}=\sup_{x\in D,y\in \R^d} |r_N(x,y)|$ and $\tilde r_N$ is defined as above.
\end{itemize}
\end{definition}

\begin{lemma} \label{lem:Uniform}
	Let the L\'evy field $Z$ satisfy Assumption~\ref{assum:Levy}. 
	Moreover, let $k:\R^d\times\R^d\rightarrow\R$ be a smoothing function such that $\tilde{k}\in L^1(\R^d)\cap L^\infty(\R^d)$ and such that the canonical distance $d_c$ of $(G_k(x))_{x\in D}$ satisfies the covering property of Talagrand's Lemma \ref{lemma:tal}, where $G$ is the centered Gaussian part of $Z$, i.e.\ the $\tripleBar{\cdot}$-continuous L\'evy field with characteristic triplet $(0,\sigma^2,0)$. Furthermore, let $k=k_N+r_N, N\in \N,$ be an orthogonal approximation sequence for which the centered Gaussian fields $(G_{k_N}(x))_{x\in D}$ and $(G_{r_N}(x))_{x\in D}$, $N\in\N$, all have a.s.\ continuous paths and for which $\tilde{k}_N\in L^1(\R^d)\cap L^\infty(\R^d), N\in \N$. Additionally, let $\varrho>1$, $\rho>0$, and $n\in (0,\frac{\beta}{4\varrho\kappa\rho})$, where $\kappa:=\|\tilde{k}\|_{L^\infty(\R^d)}$.
	
	Then, there is $M\in\N$ such that
	\[
	    \sup_{N\geq M, t\in[0,1]}
	    \ev{\e^{4\varrho\rho n\sup_{x\in D}|Z_{N,t}(x)|}}
	    < \infty.
	\]
	In case $k_N(x,\cdot)$ and $r_N(x,\cdot), x\in \R^d,$ have disjoint supports for every $N\in\N$ one can choose $M=1$.
\end{lemma}

\begin{proof}
We define $b':=\int_{\{0<s\leq 1\}} s\nu(\d s)$ and denote by $P$ the L\'evy field associated with the characteristic triplet $(b',0,\nu)$. 
Then, $P$ is $\tripleBar{\cdot}$-continuous and for an arbitrary smoothing function 
$l:\R^d\times\R^d\rightarrow\R$ the smoothed field $Z_l$ satisfies
\[
   Z_l(x)
   =
   (b-b') \int_{\R^d}l(x,y)\d y + G_l(x)+P_l(x), \qquad x\in\R^d.
\]
With $G_{N,t}(x):=G_{k_N+t r_N}(x), P_{N,t}(x):=P_{k_N+tr_N}(x), N\in\N, t\in [0,1]$ it follows for arbitrary $B>2|b-b'|\,\|\tilde{k}\|_{L^1(\R^d)}$ and each $N\in\N, t\in [0,1]$ and every $\lambda\in (0,1)$:
\begin{equation} \label{estimate expected value}
\begin{aligned}
	 &\ev{\e^{4\varrho\rho n\sup_{x\in D}|Z_{N,t}(x)|}}
	 \leq 
	 \sum_{j=0}^\infty \e^{4\varrho\rho n(j+1)B}\prob\left(\sup_{x\in D}|Z_{N,t}(x)| \geq  jB\right)\\
	 &\leq 
	 \e^{4\varrho\rho n B}
	 + \sum_{j=1}^\infty 
	 \e^{4\varrho\rho n(j+1)B}\prob\left(\sup_{x\in D}|G_{N,t}(x)|+\sup_{x\in D}|P_{N,t}(x)|\geq (j-\frac{1}{2})B\right)\\
	 &\leq 
	 \e^{4\varrho\rho n B}\left(1+\sum_{j=1}^\infty \e^{4\varrho\rho n j B}\left[\prob(\sup_{x\in D}|G_{N,t}(x)|\geq (j-\frac{1}{2})\lambda B)\right.\right.\\
	    &\left.\left.+\prob(\sup_{x\in D}|P_{N,t}(x)|\geq(j-\frac{1}{2})(1-\lambda)B)\right]\right).
	\end{aligned}
\end{equation}

We next verify the hypothesis of Proposition~\ref{prop:gauss_est} for the family of smoothing functions $k_{N} + t \, r_N$, $ N \in \N, t \in [0,1]$ and the smoothing function $k$.
Using property (i) of an orthogonal approximation sequence, we set
\begin{align*}
   \bar\sigma^2_{N,t}
   &:=
   \sup_{x \in D} \ev{G_{N,t}(x)^2}
   =
   \sup_{x\in D} \sigma^2 
   \left(\int_{\R^d} |k_N(x,y)|^2 \,\d y + t^2 \int_{\R^d} |r_N(x,y)|^2 \,\d y \right)
   \\
   &\le
   \sup_{x \in D} \sigma^2 
   \left( \int_{\R^d} |k_N(x,y)|^2 \,\d y + \int_{\R^d} |r_N(x,y)|^2 \,\d y \right)
   =
   \sup_{x \in D} \ev{G_k(x)^2} =: \bar\sigma^2,
\end{align*}
which implies hypothesis (i) of Proposition~\ref{prop:gauss_est}.

Denoting the canonical distances of the centered Gaussian random fields $(G_k(x))_{x \in D}$ and $(G_{N,t}(x))_{x \in D}$ by $d_c$ and $d_{N,t}$, respectively, we obtain for arbitrary $N \in \N, t \in [0,1]$ and each $x_1, x_2 \in \R^d$, using property (i) of an orthogonal approximation sequence,
\begin{align*}
 &d_{N,t}(x_1,x_2) 
   =
   \left( \ev{( G_{N,t}(x_1) - G_{N,t}(x_2))^2} \right)^{1/2}\\
   &=
   \sigma
   \left(\int_{\R^d} \left( k_{N,t}(x_1,y) - k_{N,t}(x_2,y) \right)^2 \,\d y \right)^{1/2}\\
   &=
   \sigma
   \left(\int_{\R^d} \left( k_N(x_1,y) - k_N (x_2,y) \right)^2 \,\d y +t^2\int_{\R^d}\left( r_N(x_1,y) - r_N(x_2,y)\right)^2 \,\d y\right)^{1/2}\\
   &\leq
   \sigma
   \left(\int_{\R^d} \left( k_N(x_1,y) - k_N (x_2,y) \right)^2 \,\d y +\int_{\R^d}\left( r_N(x_1,y) - r_N(x_2,y) \right)^2 \,\d y\right)^{1/2}\\
   &=
   \sigma
   \left(\int_{\R^d} \left( k(x_1,y) - k(x_2,y) \right)^2 \,\d y\right)^{1/2}=d_c(x_1,x_2),
\end{align*}
implying hypothesis (ii) of Proposition~\ref{prop:gauss_est}. Since, by assumption on $k_N$ and $r_N$, hypothesis (iii) of Proposition~\ref{prop:gauss_est} is satisfied as well, it follows from Proposition~\ref{prop:gauss_est} that there are constants $A>\bar{\sigma}^2, K>0,v,\gamma>0, \varepsilon_0\in (0,\bar{\sigma})$ such that for all $N\in\N, t\in [0,1], \lambda\in (0,1)$
\begin{equation}\label{estimate Gaussian part II}
    \prob\left(\sup_{x\in D}|G_{k_\iota}(x)|\geq (j-\frac{1}{2})\lambda B\right)
    \leq
    \gamma\left(\frac{K A (j-\frac{1}{2})\lambda B}{\sqrt{v}\bar{\sigma}^2}\right)^v\exp(-\frac{((j-\frac{1}{2})\lambda B)^2}{2\bar{\sigma}^2}),
\end{equation}
whenever $j>\frac{1}{2}+\frac{\bar{\sigma}(1+\sqrt{v})}{\lambda B\varepsilon_0}$.

Next, because of
\[\forall\,N\in\N, t\in [0,1]:\,\|\tilde{k}_{N,t}\|_{L^1(\R^d)}\leq\|\tilde{k}\|_{L^1(\R^d)}+\|\tilde{r}_N\|_{L^1(\R^d)}\]
and property (ii) of an orthogonal approximation sequence \[\kappa_1:=\sup_{N\in\N, t\in [0,1]}\|\tilde{k}_{N,t}\|_{L^1(\R^d)}<\infty.\]
Moreover, since for all $N\in\N, t\in [0,1]$
\begin{equation}\label{estimate kappa}
    \kappa_{N,t}=\sup_{x\in D}\|k_N(x,\cdot)+t r_N(x,\cdot)\|_{L^\infty(\R^d)}\leq\kappa+(1-t)\kappa_{r,N}\leq \kappa+\kappa_{r,N}
\end{equation}
it follows from property (ii) of an orthogonal approximation sequence that for every $\varepsilon>0$ there is $M_\varepsilon\in\N$ such that
\[\forall\,N\in\N, N\geq M_\varepsilon, t\in [0,1]:\,\kappa_{N,t}\leq\kappa+\varepsilon.\]
Moreover, in case $k_N(x,\cdot)$ and $r_N(x,\cdot)$ have disjoint supports for every $N\in\N, x\in \R^d$, it follows that in \eqref{estimate kappa} we even have $\kappa_{N,t}\leq \kappa$.

Now, applying for fixed $\varepsilon>0$ Proposition~\ref{prop:poisson_est} to the family of smoothing functions $(k_N+t r_N)_{N\geq M_\varepsilon, t\in [0,1]}$ (respectively to $(k_N+t r_N)_{N\in\N, t\in [0,1]}$) gives that for every $\tau\in (0,1)$
\begin{equation} \label{estimate Poisson part II}
    \prob\left(\sup_{x\in D}|P_{N,t}(x)|\geq (j-\frac{1}{2})(1-\lambda)B\right)
    \leq 
    C_\tau \e^{-\frac{\beta}{\kappa+\varepsilon}(j-\frac{1}{2})(1-\lambda){\color{blue}\tau}B}
\end{equation}
for all $j\in \N, \lambda\in(0,1)$, whenever $N\geq M_\varepsilon$ (resp.\ $N\in\N$), $t\in [0,1]$ and where
\[
   C := 
   \exp\left( \frac{\beta\kappa_1}{\kappa} 
   \left(
   \e^\beta \int_{\{0<s\leq 1\}}|s|\nu_+(\d s) +
   \frac{1}{\beta \e (1-\tau)}
   \int_{\{s>1\}} \e^{\beta s} \nu_+(\d s)
   \right)
   \right).
\]
Finally, since $n\in (0,\frac{\beta}{4\varrho\kappa\rho})$ there are $\lambda_0\in (0,1)$ and $\varepsilon>0$ such that $n<\beta(1-\lambda_0)/(\kappa+\varepsilon)4\varrho\rho$. 
Then, with $B>2|b-b'|\|\tilde{k}\|_{L^1(\R^d)}$ so large that $2\bar{\sigma}(1+\sqrt{v})/\varepsilon_0 B\lambda_0<1/2$ it follows from \eqref{estimate expected value}, \eqref{estimate Gaussian part II}, and \eqref{estimate Poisson part II} that for every $\tau\in (0,1)$ and for all $N\geq M_\varepsilon$ (resp.\ $N\in\N$), $t\in[0,1]$
\begin{align*}
    \ev{\e^{4\varrho\rho n\sup_{x\in D}|Z_{N,t}(x)|}}
	\leq 
	\e^{4\varrho\rho n B}
	&\left(
	1 + \sum_{j=1}^\infty \e^{4\varrho\rho n j B}
	\left[ C_\tau \e^{-\frac{\beta}{\kappa+\varepsilon}(j-\frac{1}{2})(1-\lambda_0){\color{blue}\tau}B}\right.
	\right.\\
	&\left.\left.+\gamma \left(\frac{K A (j-\frac{1}{2})\lambda_0 B}{\sqrt{v}\bar{\sigma}^2}\right)^v
	\exp\left(-\frac{((j-\frac{1}{2})\lambda_0 B)^2}{2\bar{\sigma}^2}\right)\right]
	 \right).
\end{align*}
Choosing $\tau$ sufficiently close to $1$, by the same arguments as employed in the proof of Theorem \ref{thm:levy_int}, the series converges since $4\varrho\rho n < \beta(1-\lambda_0)/(\kappa+\varepsilon)$, hence the assertion follows.
\end{proof}

\begin{lemma} \label{lem:VanishingMoments}
	Let $Z, k, k_N,$ and $r_N, N\in\N$ be as in Lemma \ref{lem:Uniform}. Moreover, let $\varrho'>1$ and $n\geq 1/\varrho'$. Then, for every $\delta\in(0,1)$ there is a constant $C>0$ depending only on $\delta$, $Z$, $k$, and $n\varrho'$ such that
	\[\forall\,N\in\N:\,\|\sup_{x\in D}|R_N(x)|^{\varrho'}\|_{L^n\pspace}\leq C\left(\max\{\alpha_N,\alpha_N^{1-\delta}\}\right)^{\varrho'}\]
	where $\alpha_N:=\max\{\|\tilde{r}_N\|_{L^1(\R^d)},\kappa_{r,N}\}$. In particular, $\lim\limits_{N\rightarrow\infty}\ev{\sup_{x\in D}|R_N(x)|^{n\rho'}}=0$.
\end{lemma}

\begin{proof}
	With the same notation as in the proof of Lemma \ref{lem:Uniform} we have
	\[R_N(x)=(b-b')\int_{\R^d} r_N(x,y)\,\d y+ G_{r_N}(x)+P_{r_N}(x),\]
	so that by Jensen's inequality
	\begin{align}\label{est:expected}
	    \ev{\sup_{x\in D}|R_N(x)|^{n\varrho'}}\leq& 3^{n\varrho'-1}\left(
	    |b-b'|^{n\varrho'}\|\tilde{r}_N\|_{L^1(\R^d)}^{n\varrho'}+\ev{\sup_{x\in D}|G_{r_N}(x)|^{n\varrho'}}\right.\\
	    &+\left.\ev{\sup_{x\in D}|P_{r_N}(x)|^{n\varrho'}}\right).\nonumber
	\end{align}
	In order to estimate the Gaussian part, we precede as in the proof of Lemma \ref{lem:Uniform}. From property~(i) of an orthogonal approximation sequence we conclude
	\[\forall\,N\in\N, x\in D:\,\|r_N(x,\cdot)\|_{L^2(\R^d)}\leq\|k(x,\cdot)\|_{L^2(\R^d)}\mbox{ and }d_N\leq d_c,\]
	where $d_N$ and $d_c$ denote the canonical distances associated with $(G_{r_N}(x))_{x\in D}$ and $(G_k(x))_{x\in D}$, respectively. Therefore, by Proposition \ref{prop:gauss_est}, or better the first inequality in (\ref{estimate:gaussian part}), there are constants \[A>\bar{\sigma}^2:=\sigma^2\sup_{x\in D}\|k(x,\cdot)\|^2_{L^2(\R^d)}\geq\sigma^2\sup_{x\in D}\|r_N(x,\cdot)\|^2_{L^2(\R^d)}=:\bar{\sigma}^2_{r,N},\]
	$K,v>0,\varepsilon_0\in (0,\bar{\sigma})$ such that for every $N\in\N, g>\bar{\sigma}_{r,N}(1+\sqrt{v})/\varepsilon_0$
	\[\prob\left(\sup_{x\in D}|G_{r_N}(x)|\geq g\right)\leq\left(\frac{K A g}{\sqrt{v}\bar{\sigma}^2_{r,N}}\right)^v\exp(-\frac{g^2}{2\bar{\sigma}^2_{r,N}}).\]
	For $j\in\N$ with $j>\bar{\sigma}^\delta(1+\sqrt{v})/\varepsilon_0$ it holds $j\bar{\sigma}_{r,N}^{1-\delta}>\sigma_{r,N}(1+\sqrt{v})/\varepsilon_0$ so that with $M:=\max\{\lceil \bar{\sigma}^\delta(1+\sqrt{v})/\varepsilon_0\rceil, \lceil \bar{\sigma}^\delta\sqrt{v\frac{1+\delta}{\delta}}\rceil\}$
	\begin{align*}
	    \ev{\sup_{x\in D}|G_{r_N}(x)|^{n\varrho'}}\leq& \sum_{j=0}^\infty \left( (j+1)\bar{\sigma}_{r,N}^{1-\delta}\right)^{n\varrho'}\prob\left(\sup_{x\in D}|G_{r_N}(x)|\geq j\bar{\sigma}_{r,N}^{1-\delta}\right)\\
	    \leq& \sum_{j=0}^M\left((j+1)\bar{\sigma}_{r,N}^{1-\delta}\right)^{n\varrho'}\\
	    &+\sum_{j=M+1}^\infty \left((j+1)\bar{\sigma}_{r,N}^{1-\delta}\right)^{n\varrho'}\left(\frac{K A j}{\sqrt{v}\bar{\sigma}^{1+\delta}_{r,N}}\right)^v\exp(-\frac{j^2}{2\bar{\sigma}_{r,N}^{2\delta}})\\
	    \leq &\bar{\sigma}_{r,N}^{(1-\delta)n\varrho'}\left(\sum_{j=0}^M (j+1)^{n\varrho'}\right.\\
	    &\left. +2^{n\varrho'}\left(\frac{K A}{\sqrt{v}}\right)^v\sum_{j=M+1}^\infty j^{n\varrho'}\left(\frac{j}{\bar{\sigma}_{r,N}^{1+\delta}}\right)^v\exp\left(-\frac{j^2}{2\bar{\sigma}_{r,N}}^{2\delta}\right)\right)\\
	    \leq & \bar{\sigma}_{r,N}^{(1-\delta)n\varrho'}\left(\sum_{j=0}^M (j+1)^{n\varrho'}\right.\\
	    &\left. +2^{n\varrho'}\left(\frac{K A}{\sqrt{v}}\right)^v\sum_{j=M+1}^\infty j^{n\varrho'}\left(\frac{j}{\bar{\sigma}^{1+\delta}}\right)^v\exp\left(-\frac{j^2}{2\bar{\sigma}^{2\delta}}\right)\right)
	\end{align*}
	where in the last step we have used that for every $j>\lceil \bar{\sigma}^\delta\sqrt{v\frac{1+\delta}{\delta}}\rceil$ the functions $f_j:[0,\infty)\rightarrow\R, f_j(x):=(j x^{-(1+\delta)})^v\exp(-j^2/2x^{2\delta})$ are strictly increasing on $[0,\bar{\sigma}]$ and that $\bar{\sigma}_{r,N}\in [0,\bar{\sigma}]$ for all $N\in\N$. Therefore, since the above series converges, denoting by $C_1$ the expression in brackets on the right hand side of the above inequality and taking into account that
	\[\forall\,N\in\N:\,\bar{\sigma}^2_{r,N}=\sigma^2\sup_{x\in D}\int_{\R^d}|r_N(x,y)|^2\,\d y\leq\sigma^2\kappa_{r,N}\|\tilde{r}_N\|_{L^1(\R^d)}\leq\sigma^2 \alpha_N^2,\]
	we derive
	\begin{equation}\label{est:expected Gaussian}
	    \forall\,N\in\N:\,\ev{\sup_{x\in D}|G_{r_N}(x)|^{n\varrho'}}\leq\sigma^{1-\delta}C_1 \alpha_N^{(1-\delta)n\varrho'}.
	\end{equation}
	
	We next consider the Poisson part in (\ref{est:expected}). By H\"older's inequality, $P_{r_N}(x)\leq |P|_{|r_N|}(x)\leq |P|(\tilde r_N)$, cf.\ \eqref{eqa:bound}, Proposition \ref{prop:moments}, and $\|\tilde{r}_N\|_{L^1(\R^d)}^k \kappa_{r,N}^{l-k}\leq\alpha_N^l$ for every $0\leq k\leq l$,  we have for all $N\in\N$
	\begin{align}\label{est:expected Poisson}
	    &\ev{\sup_{x\in D}|P_{r_N}(x)|^{n\varrho'}}\leq \left(\ev{\sup_{x\in D}|P_{r_N}(x)|^{\lceil n\varrho'\rceil}}\right)^{\frac{n\varrho'}{\lceil n\varrho'\rceil}}\nonumber\\
	    \leq& \left(\ev{|P|(\tilde{r}_N(x)|^{\lceil n\varrho'\rceil}}\right)^{\frac{n\varrho'}{\lceil n\varrho'\rceil}}
	    =\left(\sum_{\substack{I\in \mathscr{P}^{(\lceil n\varrho'\rceil)}\\ I 
				= 
				\{I_1, \ldots, I_k\}}} 
		\prod_{\ell=1}^k c^+_{|I_\ell|}\int_{\R^d} \tilde r_N^{|I_\ell|} \,\d x\right)^{\frac{n\varrho'}{\lceil n\varrho'\rceil}}\nonumber\\
		\leq & \left(\sum_{\substack{I\in \mathscr{P}^{(\lceil n\varrho'\rceil)}\\ I 
				= 
				\{I_1, \ldots, I_k\}}} 
		\prod_{\ell=1}^k c^+_{|I_\ell|}\|\tilde{r}_N\|_{L^1(\R^d)}\kappa_{r,N}^{|I_l|-1}\right)^{\frac{n\varrho'}{\lceil n\varrho'\rceil}}\\
		= & \left(\sum_{\substack{I\in \mathscr{P}^{(\lceil n\varrho'\rceil)}\\ I 
				= 
				\{I_1, \ldots, I_k\}}} 
		\prod_{\ell=1}^k c^+_{|I_\ell|}\;\|\tilde{r}_N\|_{L^1(\R^d)}^k\kappa_{r,N}^{\lceil n\varrho'\rceil-k}\right)^{\frac{n\varrho'}{\lceil n\varrho'\rceil}}
		\leq  \left(\sum_{\substack{I\in \mathscr{P}^{(\lceil n\varrho'\rceil)}\\ I 
				= 
				\{I_1, \ldots, I_k\}}} 
		\prod_{\ell=1}^k c^+_{|I_\ell|}\;\right)^{\frac{n\varrho'}{\lceil n\varrho'\rceil}}\alpha_N^{n\varrho'},\nonumber
	\end{align}
	where $\mathscr{P}^{(\lceil n\varrho'\rceil)}$ denotes the collection of all partitions on $\{1,\ldots, \lceil n\varrho'\rceil\}$ into non-intersecting, none-empty sets $I_1,\ldots, I_k, 1\leq k\leq \lceil n\varrho'\rceil$, and $c^+_{|I_l|}$ are suitable non-negative numbers, compare Proposition \ref{prop:moments}. Note that the constants $c^+_{|I_\ell|}$ are taken w.r.t.\ the modified L\'evy measure $\nu_+$ associated with $|P|$ instead of $\nu$ associated with $P$.
	
	Therefore, setting $C_2$ to be the factor in front of $\alpha_N^{n\varrho'}$ in the previous inequality, we have
	\begin{equation}
	    \forall\,N\in\N:\,\ev{\sup_{x\in D}|P_{r_N}(x)|^{n\varrho'}}\leq  C_2 \alpha_N^{n\varrho'}.
	\end{equation}
	Combining (\ref{est:expected}), (\ref{est:expected Gaussian}), and (\ref{est:expected Poisson}) we finally obtain
	\[\forall\,N\in\N:\,\ev{\sup_{x\in D}|R_N(x)|^{n\varrho'}}\leq 3^{n\varrho'-1}(|b-b'|+\sigma^\delta C_1+C_2)\left(\max\{\alpha_N,\alpha^{1-\delta}_N\}\right)^{n\varrho'}\]
	which proves the claim.
\end{proof}

Combining Lemmas \ref{lemma:prob_error_Estimate}, \ref{lem:Uniform} and \ref{lem:VanishingMoments} we now obtain the following convergence result:
\begin{theorem} \label{theo:approxSol}
	Let $Z$ be a L\'evy field satisfying Assumption~\ref{assum:Levy}. 
	Moreover, let $k:\R^d\times\R^d\rightarrow\R$ be a smoothing function such that $\tilde{k}\in L^1(\R^d)\cap L^\infty(\R^d)$ and such that the canonical distance $d_c$ of $(G_k(x))_{x\in D}$ satisfies the covering property of Talagrand's Lemma \ref{lemma:tal}, where $G$ is the centered Gaussian part of $Z$. 
	Furthermore, let $k=k_N+r_N, N\in \N,$ be an orthogonal approximation sequence for which the centered Gaussian fields $(G_{k_N}(x))_{x\in D}$ and $(G_{r_N}(x))_{x\in D}$, $N\in\N$, all have a.s.\ continuous paths and for which $\tilde{k}_N\in L^1(\R^d)\cap L^\infty(\R^d), N\in \N$.
	
	Let $u$ and $u_N, N\in\N,$ be the solution of \eqref{eq:diff_eq} with random conductivity $T\circ Z_k$ and $T\circ Z_{k_N}$, respectively, where $T$ satisfies Assumption~\ref{assum:T}. Assume that with $\kappa:=\|\tilde{k}\|_{L^\infty(\R^d)}$ we have $\beta>4\kappa\rho$. 
	Then for all $n\in [1,\frac{\beta}{4\kappa\rho})$, $\varrho\in (1,\frac{\beta}{4\kappa\rho n})$, and $\delta\in (0,1)$ there is a constant $C'>0$ and $M\in \N$ such that for all $N\geq M$ we have
	\[\|u-u_N\|_{L^n(\pspace; H^1(D))}\leq C'\max\{\alpha_N,\alpha_N^{1-\delta}\},\]
	where $\alpha_N=\max\{\|\tilde{r}_N\|_{L^1(\R^d)},\kappa_{r,N}\}$. In case $k_N(x,\cdot)$ and $r_N(x,\cdot), N\in\N,$ have disjoint supports for every $x\in\R^d$, one can choose $M=1$.
	
	In particular, $(u_N)_{N\in\N}$ converges to $u$ in $L^n(\pspace; H^1(D))$. 
	The constant $C'$ depends only on $B$, $Z$, $k$, $\frac{n\varrho}{\varrho-1}$, $\|f\|_{L^2(D)}$, $\|g_D\|_{H^{\frac{1}{2}}(\partial_D)}$, $\|g_N\|_{H^{-\frac{1}{2}}(\partial_N)}$, and $C$, the constant from Lemma \ref{lemma:app-a-prio}. 
\end{theorem}

\subsection{Series Expansion of Lévy Coefficients} \label{sec:KL}
We provide a two-step procedure for approximating smoothed L\'evy random fields and the associated solutions of \eqref{eq:diff_eq}, resulting in a finite-dimensional approximation of $Z_k$ as a natural generalization of the Karhunen--Loève expansion for L\'evy fields.
In particular, we shall employ two specific orthogonal approximation sequences, one by restricting to sets in a compact exhaustion of $\mathbb R^d$ and the other by truncated Mercer expansion. 

In the first step we restrict the second argument of the kernel $k(x,y)$ to a set $\Lambda_N$ from a compact exhaustion $(\Lambda_N)_{N\in\N}$ of $\R^d$, i.e., $(\Lambda_N)_{N\in\N}$ is a sequence of compact subsets of $\R^d$ with $\Lambda_N\subset \mbox{int}(\Lambda_{N+1}), N\in\N$, and $\cup_{N}\Lambda_N=\R^d$. 
This is necessary as only $x$ is restricted to $D$, whereas $y$ transports the effect of noise source terms from  locations $y\not \in D$ into $D$.

For a $\tripleBar{\cdot}$-continuous L\'evy field $Z$ it follows immediately that $Z^N(f):=Z(\mathds{1}_{\Lambda_N}f)$, $N\in\N,$ again defines a $\tripleBar{\cdot}$-continuous 
generalized random field, which, however, is no longer stationary.
For a Mat\'ern kernel $k_{\alpha,m}, m>0, \alpha>d+\max\{0,\frac{3d-12}{8}\}$ it therefore follows from Theorem \ref{thm:matern} that for the smoothing function $k(x,y)=k_{\alpha,m}(x-y)$ the smoothed fields $(Z^N_k(x))_{x\in D}, N\in\N,$ have a.s.\ continuous paths. 
Moreover, it was shown in the proof of Theorem \ref{thm:levy_int} that the canonical distance $d_c$ associated with the Gaussian field $(G_k(x))_{x\in D}$ (where $G$ denotes, as usual, the centered Gaussian part of $Z$) satisfies the covering property of Talagrand's Lemma \ref{lemma:tal} and, by Lemma \ref{lemma:mat-hold} (ii), there holds $\tilde{k}\in L^1(\R^d)\cap L^\infty(\R^d)$ with $\lim_{|y|\rightarrow\infty}\tilde{k}(y)=0$. 
Therefore, the assumptions of Theorem~\ref{theo:approxSol} are automatically satisfied for Mat\'ern kernels as smoothing functions.

\begin{corollary} \label{lem:cut_off}
  Let $Z, T$ and $u$ be given as in Theorem \ref{theo:approxSol}. Furthermore, let $k_{\alpha,m}$ be a Mat\'ern kernel with $\alpha > d$. For a compact exhaustion $(\Lambda_N)_{N\in\N}$ of $\R^d$ with $D\subset\Lambda_1$, we set $k_N(x,y):=k_{\alpha,m}(x,y)\mathds{1}_{\Lambda_N}(y)$ and denote by $u_N$ the solution of (\ref{eq:diff_eq}) with random conductivity $T\circ Z_{k_N}$.
    
    Then for all $n\in [1,\frac{\beta}{4\kappa\rho})$, $\varrho\in (1,\frac{\beta}{4\kappa\rho n})$, and $0 < m' < m$, there is a constant $C>0$ such that for all $N\in\N$ we have
	\begin{equation}\label{eq:conv-uN}
	  \|u-u_N\|_{L^n(\pspace; H^1(D))}\leq C \e^{-m'd_e(D,\Lambda_N^c)}
	\end{equation}
	where $d_e(D,\Lambda_N^c)$ denotes the Euclidean distance between $D$ and $\Lambda_N^c$.
\end{corollary}

\begin{proof}
We first verify that $k = k_N + r_N$ is an orthogonal approximation sequence in the sense of Definition~\ref{def:unif_orth_dec}.
Condition (i) obviously holds as the $y$-domains of $k_N$ and $r_N$ are disjoint. 
To verify condition (ii), we use the decay rate $k_{\alpha,m}(x,y) \leq C \e^{-m|x-y|}$ for $|x-y| \to \infty$ from Lemma~\ref{lemma:mat-hold} (ii). 
This implies for $0 < m' < m$
\begin{equation*}
	  \|\tilde r_N\|_{L^1(\R^d)} 
	  \leq 
	  C\int_{\Lambda_N^c} \sup_{x\in D} \e^{-m|x-y|}\,\d y 
	  = 
	  C \e^{-m d_c(D,\Lambda_N^c)}\int_{\Lambda_N^c} \sup_{x\in D} \e^{-(m-m')|x-y|}\,\d y 
\end{equation*}
where the last integral is finite as $m'<m$, and can be estimated by a constant as the integration area gets smaller for $N\to \infty$.
We next have 
\begin{equation*}
	\kappa_{r,N} 
	= 
	\|\tilde r_N\|_{L^\infty(R^d)} 
	= 
	\sup_{x\in D, y\in \Lambda_N^c} C \e^{-m|x-y|} 
	= 
	C \e^{-md_e(D,\Lambda_N^c)}
\end{equation*}
and thus
\begin{equation*}
	\alpha_N 
	= 
	\max\{\|\tilde r_N\|_{L^1(\R^d)}, \kappa_{r,N}\} 
	\leq 
	C \e^{-m'd_e(D,\Lambda_N^c)}.
\end{equation*}

We can now apply Theorem~\ref{theo:approxSol} and obtain	
\begin{equation*}
	\|u-u_N\|_{L^n(\pspace; H^1(D))}
	\leq 
	C'\max\{\alpha_N,\alpha_N^{1-\delta}\} 
	\leq 
	C'C \e^{-m'd_e(D,\Lambda_N^c)},
\end{equation*}
	where we merged $m'$ and $\delta$, as for $m'\in (0,m)$ and $\delta\in (0,1)$ we again have $(1-\delta)m'\in (0,m)$. Redefining the constant $C$ implies $\eqref{eq:conv-uN}$.
\end{proof}

\begin{remark} \label{rem:no-problem}
	\strut \\[-1em]
	\begin{itemize}
		\item[(i)] 
		At first sight, the discontinuous cut-off $\mathds{1}_{\Lambda_N}(y)$ appears to  contradict the assumptions needed for the continuity of the paths of $Z_{k_N}(x)$, which is part of the prerequisites of Theorem~\ref{theo:approxSol}. 
		Nevertheless, we may still obtain continuous realizations of $Z_{k_N}(x)$ from Theorem~\ref{thm:matern} as it is equivalent to apply the noise $Z$ to $k_N(x,\cdot)$ or to apply the noise $\mathds{1}_{\Lambda_N} Z$ to $k(x,\cdot)$. 
As $\varphi_N(f)=\e^{\int_{\Lambda_N}\psi(f)\, \d x}$ (and likewise for $R_N(y)$ with $\Lambda_N$ replaced by $\Lambda_N^c$), we see that this functional still is $\tripleBar{\cdot}$-continuous and therefore the results of Theorem~\ref{thm:matern} are compatible with the cut-off $\Lambda_N$. 
		\item[(ii)] 
		Likewise, the H\"older continuity of the covariance function $k_{2\alpha,m}$
		of the Gaussian part (see \eqref{Gaussian covariance Hoelder}) is immediately passed on to the truncated fields $G_{k_N}(x)$, $Z_{k_{N,t}}(x)$ and $R_N(x)$, as by Definition \ref{def:unif_orth_dec} (i) the canonical distance of all these fields is dominated by that of $G_k(x)$.  
	\end{itemize}
\end{remark}

As we are now able to restrict both $x$ and $y$ to a (sufficiently large) bounded domain $\Lambda$ at the cost of a small and controllable  error, we can now apply Mercer expansion of the smoothing kernel on $\Lambda$ and recall the following well-known result:  
\begin{theorem}[Mercer's Theorem, cf.\ {\cite[Theorem~1.80]{LordEtAl2014}}] \label{thm:mercer}
	Let $\Lambda$ be a compact subset of $\mathbb{R}^d$ and $k:\Lambda\times \Lambda\to \mathbb{R}$ be a continuous, positive definite kernel.  
	Associated with $k$ is a compact linear operator $K:L^2(\Lambda) \to L^2(\Lambda)$ defined by 
	\begin{equation*}
		[K\phi](x) = \int_\Lambda k(x,y)\phi(y)\,\d y.
	\end{equation*}
	Then there exist an orthonormal basis $\{ e_i \}_{i\in \mathbb{N}}$ of $L^2(\Lambda)$ consisting of eigenfunctions
	of $K$ such that the associated sequence of eigenvalues $\{ \lambda_i \}_{i\in \mathbb{N}}$ is 
	non-negative with zero as its only possible point of accumulation. 
	The eigenfunctions corresponding to positive eigenvalues are continuous on $\Lambda$ and $k$ has the representation
	\begin{equation*}
		k(x,y) = \sum_{i=1}^\infty \lambda_i e_i(x) e_i(y), 
		\qquad x,y \in \Lambda,
	\end{equation*}
	where the convergence is absolute and uniform.
\end{theorem}

\begin{remark}\leavevmode
\begin{itemize}
\item[(i)] 
In what follows we will obtain a finite-dimensional approximation of the smoothed random field by expanding the smoothing kernel $k$ rather than its covariance function $k^\vee *k$, as in a Karhunen-Loève (KL) expansion.  The difference to the standard KL-expansion lies in the fact that the covariance function of the truncated noise is
$\int_\Lambda k(x-z)k(y-z)\, \d z$ in contrast to $\int_{\R^d} k(x-z)k(y-z)\, \d z$ expanded in the standard KL-expansion, where $x,y\in\Lambda$.  
It is easily seen that the eigenvalues obtained for the expansion of the first covariance function are $\lambda_i^2$ and the eigenfunctions $e_i(x)$ remain the same, as this operator is the square to the integral operator defined by $k(x-y)$ on $L^2(\Lambda,\d y)$ . 

\item[(ii)] Of course also an expansion of the paths of $Z_k(x)$ in eigenfunctions of the second covariance operator is possible in principle. However this requires an expansion of the smoothing kernel $k(x-y)$ in $x\in\Lambda$ (or $D$) and prove uniformity and decay properties of this expansion in $y\in\R^d$. This approach seems more involved than the cut-off method used here, as the spectral properties of the integral operator induced by $k(x-y)$ can not be used. Also, the cut-off method seems efficient as for Mat\`ern kernels it does not lead to a worsening of rates of convergence in Theorem \ref{theo:ApproxKL} below.  
		\item[(iii)] 
		The assumptions of the above theorem clearly hold for $k(x,y)=k_{\alpha,m}(x-y)$ for $2\alpha>d$ when resticted to $\Lambda$ in both arguments $x,y$, cf.\ Lemma \ref{lemma:mat-hold}. Note that by Definition \ref{def:matern} the Fourier transform of $k_{\alpha,m}(x)$ is positive, which implies the positive definiteness of the kernel $ k_{\alpha,m}(x-y)$.
		\item[(iv)] 
		Note that the eigenfunctions $e_i(x)$ and the eigenvalues $\lambda_i$ depend on $\Lambda$. 
		For $\Lambda=\Lambda_N$ we use the notation $\lambda_{N,i}$ and $e_{N,i}(x)$. 
	\end{itemize}
\end{remark}

\begin{corollary} \label{lem:KL}
	Under the assumptions given in Theorem \ref{theo:approxSol}, let in addition $k$ be a positive definite kernel. 
	Then for fixed $N \in \mathbb N$ the decomposition 
	\[
	   k_N(x,y) = k(x,y)\mathds{1}_{\Lambda_N}(y) = k_{N,N'}(x,y) + r_{N,N'}(x,y),
	   \qquad x,y \in \Lambda_N,
	\]
	with 
	$k_{N,N'}(x,y) = \sum_{i=1}^ {N'}\lambda_{N,i} e_{N,i}(x) e_{N,i}(y)$ the truncated Mercer expansion (Theorem~\ref{thm:mercer}) with remainder $r_{N,N'}$ represents an orthogonal approximation sequence in the sense of Definition~\ref{def:unif_orth_dec} w.r.t.\ the approximation parameter $N' \in \mathbb N$. 
	Then for the solution $u_N$ of \eqref{eq:diff_eq} with smoothing kernel truncated in the $y$ variable  (cf.\ Corollary~\ref{lem:cut_off}) and the solution $u_{N,N'}$ associated with $Z_{k_{N,N'}}$, we have
	\begin{equation}
		\lVert u_N-u_{N,N'}\rVert_{L^n(\pspace; H^1(D))}
		\leq 
		\hat C \lvert \Lambda_N\rvert\kappa_{r,N,N'} \to 0
		\quad \text{ as } N' \to \infty,    
	\end{equation}
	where  $\kappa_{r,N,N'}=\sup_{x\in D,y\in\Lambda_N}\lvert\sum_{i=N'+1}^\infty\lambda_{N,i} e_{N,i}(x)e_{N,i}(y)\rvert$ and $\vert \Lambda_N\rvert >1$.
\end{corollary}
\begin{proof}
	Mercer's theorem provides 
	\[
	0
	\leq 
	\kappa_{r,N,N'}
	\leq 
	\sup_{x\in \Lambda_N,y\in \Lambda_N}
	\left|\sum_{i=N'+1}^{\infty} \lambda_{N,i} \, e_{N,i}(x) \, e_{N,i}(y)\right| 
	\to 0 \quad \text{ as } N'\to \infty.
	\]
	Furthermore, as $\Lambda_N$ is bounded, $\|\tilde r_{N,N'}\|_{L^1(\Lambda_N)}\leq |\Lambda_N|\kappa_{r,N,N'}$. 
	Now the assertion follows from Theorem~\ref{theo:approxSol}.
\end{proof}

We want to state a convergence rate for Matérn kernels. For this we first establish an auxiliary result.

\begin{lemma} \label{lem:KL-rates}
Let $k=k_{\alpha,m}$ be a Mat\'ern smoothing function  with $\alpha > d$.
Further, let $(\Lambda_N)_{N\in\mathbb{N}}$ denote a compact exhaustion  of $\mathbb{R}^d$ with $D\subset\Lambda_1$ and $\text{diam}(\Lambda_1)\geq 1$. 
Then for every $\varepsilon \in (0,\frac{\alpha}{d}-\frac{1}{2})$ there exists a constant $C>0$ such that for each $N \in \mathbb N$ the uniform bound $\kappa_{r,N,N'}$ on the remainder of the Mercer series of the restriction $k_N$ of $k$ to $\Lambda_N\times \Lambda_N$ satisfies
\begin{equation}
	\kappa_{r,N,N'}
    \leq
    C \text{diam}(\Lambda_N)^{4(\alpha-\frac{d}{2}-\varepsilon)}
        {N^\prime}^{-2\frac{\alpha}{d}+2+2\varepsilon}
    \qquad \forall \, N' \in \mathbb N,
\end{equation}
where  
$\kappa_{r,N,N'}=\sup_{x\in D,y\in\Lambda_N}\lvert\sum_{i=N'+1}^\infty\lambda_{N,i} 
    e_{N,i}(x)e_{N,i}(y)\rvert$.
\end{lemma}
\begin{proof}
Applying \cref{eq:bounds eigenvalues times eigenfunctions} for a given $\varepsilon \in (0,\frac{\alpha}{d}-\frac{1}{2})$ provides the existence of a constant $C>0$ such that for all $N,N'\in\N$, there holds
\begin{equation*}
\begin{split}
	\kappa_{r,N,N^\prime} 
	\leq
    &\sum_{i=N^\prime+1}^\infty 
    \sup_{x\in D,y\in\Lambda_N} \lvert \lambda_{N,i} e_{N,i}(x)e_{N,i}(y) \rvert
	\leq 
	\sum^\infty_{i=N^\prime +1} \|\sqrt{\lambda_{N,i}}e_{N,i}\|^2_{L^\infty(\Lambda_N)}\\
    \leq 
    \ &C^2 \text{diam}(\Lambda_N)^{4(\alpha-\frac{d}{2}-\varepsilon)}
        \sum_{i=N^\prime+1}^\infty i^{-2\frac{\alpha}{d} + 1 +2\varepsilon} \\
    \leq 
    \ &C^2 \text{diam}(\Lambda_N)^{4(\alpha-\frac{d}{2}-\varepsilon)}
        \int_{N^\prime}^\infty x^{-2\frac{\alpha}{d} + 1 +2 \varepsilon} \,\d x\\
    \leq 
    \ &C^2 ( 2\alpha/d - 2 - 2\varepsilon )^{-1} 
        \text{diam}(\Lambda_N)^{4(\alpha-\frac{d}{2}-\varepsilon)}
        {N^\prime}^{-2\frac{\alpha}{d}+2+2\varepsilon}.
\end{split}
\end{equation*}
Note that the series converges since $-2\frac{\alpha}{d} + 1 + 2 \varepsilon < -1$. 
Redefining $C$ yields the assertion.
\end{proof}

We now combine Corollaries~\ref{lem:cut_off} and \ref{lem:KL} to obtain our second main result:
\begin{theorem} \label{theo:ApproxKL}
  Let the assumption of Theorem~\ref{theo:approxSol} hold and let $k_{\alpha,m}$ be a Matérn kernel with $\alpha>d$. Let $\delta:=\text{diam}(D)$ and fix $x_0\in D$ with $D\subseteq x_0+\left[-\frac{\delta}{2},\frac{\delta}{2}\right]^d$. For fixed $0<\tilde{m}<m$ let
  $$\delta_N:=\frac{\delta+1}{2}+\frac{2}{\tilde{m}}\left(\frac{\alpha}{d}-1\right)\log N\text{ and }\Lambda_N:=x_0+\left[-\delta_N,\delta_N\right]^d, N\in\N.$$
  Moreover, we denote the solution resulting from the truncated Mercer expansion in Corollary~\ref{lem:KL} for the above $\Lambda_N$ by $u_{N,N}$ and the solution given in Lemma~\ref{lemma:weak solution} a) by $u$.
  
  For every $\upsilon\in (0,2\alpha/d-2)$ there is a constant $C>0$ such that
  \begin{equation*}
    \forall\,N\in\N:\,\| u - u_{N,N}\|_{L^n(\pspace; H^1(D))} 
	    \leq 
	C N^{-\upsilon}.
  \end{equation*}
\end{theorem}

\begin{proof}
Let $\upsilon\in (0,2\alpha/d-2)$ be given. 
We fix $\varepsilon \in (0,\frac{\alpha}{d}-1)$ such that $\upsilon<2\alpha/d-2-2\varepsilon$. 
With this $\varepsilon$ and $\tilde{m}\in (0,m)$ from the hypothesis we define
\[
   m':=\frac{\tilde{m}\left(\frac{\alpha}{d}-1-\varepsilon\right)}
           {\left(\frac{\alpha}{d}-1\right)}
   \in (0,m)
\]
so that
\[
    \frac{2}{m'}\left(\frac{\alpha}{d}-1-\varepsilon\right)
    =
    \frac{2}{\tilde{m}}\left(\frac{\alpha}{d}-1\right)
\]
resulting in
\[
    \delta_N
    =
    \frac{\delta+1}{2}+\frac{2}{m'}\left(\frac{\alpha}{d}-1-\varepsilon\right)\log N
\]
as well as
\[
    d_e(D,\Lambda_N^c)
    >
    \frac{2\left(\frac{\alpha}{d}-1-\varepsilon\right)}{m'} \log N, \quad N\in \N.
\]
Since $|\Lambda_N|=2^d\delta_N^d$, $\text{diam}(\Lambda_1)>1$, and $\text{diam}(\Lambda_N)=\sqrt{2d\,\delta_N}$, combining \cref{lem:cut_off}, \cref{lem:KL} and \cref{lem:KL-rates}, we obtain with suitable constants, denoting $\|\cdot\|_{L^n(\pspace; H^1(D))}$ by just $\|\cdot\|$,
\begin{align*}
	\| u - u_{N,N}\|
	\leq 
	\ &\| u - u_N \| + \| u_N - u_{N,N}\| \\
	\leq 
	\ & C^\prime \e^{-m' d_e(D,\Lambda_N^c)} 
	+ 
	\hat C | \Lambda_N| C\left(\text{diam}(\Lambda_N)\right)^{4\left(\alpha-\frac{d}{2}-\varepsilon\right)}
        N^{-2\left(\frac{\alpha}{d}-1 - \varepsilon\right)} \\
    \leq 
    \ &C^\prime N^{-2\left(\frac{\alpha}{d}-1-\varepsilon\right)} + 
	 \hat C C\, 2^d\sqrt{2d}\,\delta_N^{d+2\left(\alpha-\frac{d}{2}\right)-2\varepsilon} 
        N^{-2\left(\frac{\alpha}{d} - 1 - \varepsilon \right)}\\
    = 
    \ & C^\prime N^{-2\left(\frac{\alpha}{d}-1-\varepsilon\right)} 
    + 
	\hat C C\, 2^d\sqrt{2d}\,\delta_N^{2\left(\alpha-\varepsilon\right)} 
        N^{-2\left(\frac{\alpha}{d} - 1 - \varepsilon \right)}
\end{align*}
Next, we use for
\[
   \delta_N
   =
   \frac{\delta+1}{2}+\frac{2}{m'}\left(\frac{\alpha}{d}-1-\varepsilon\right) \log N
\]
that for arbitrary $\varepsilon'>0$ there is $C''>0$, depending only on $\delta, \alpha, d, m', \varepsilon$, and $\varepsilon'$, with $\delta_N^{2(\alpha-\varepsilon)}\leq C'' N^{\varepsilon'}, N\in\N$.
Applying this to $\varepsilon'$ with $\upsilon+\varepsilon'<2\left(\frac{\alpha}{d}-1-\varepsilon\right)$ we continue with our inequality from above and obtain
\begin{equation*}
\begin{split}
	\|u-u_{N,N}\|
	\leq 
	\ & C' N^{-2\left(\frac{\alpha}{d}-1-\varepsilon\right)} + 
	 \hat{C} C\, 2^d\sqrt{2d}\,\delta_N^{2\left(\alpha-\varepsilon\right)} 
        N^{-2\left(\frac{\alpha}{d} - 1 - \varepsilon \right)}\\
    \leq 
    \ & C' N^{-2\left(\frac{\alpha}{d}-1-\varepsilon\right)+\varepsilon'} 
       + \hat{C} C C'' 2^d\sqrt{2d} N^{-2\left(\frac{\alpha}{d}-1-\varepsilon\right)+\varepsilon'}\\
    \leq 
    \ & \left(C' +\hat{C} C C'' 2^d\sqrt{2d}\right) N^{-\upsilon}
\end{split}
\end{equation*}
which proves the assertion.
\end{proof}

\begin{remark} \label{rem:KL-finite}
  \leavevmode
  \begin{enumerate}[(i)]
	\item
	  We can combine Theorem~\ref{theo:approxSol} and Corollary~\ref{lem:KL} to obtain the convergence of the approximated solutions for any positive definite kernel function satisfying the assumptions of Theorem~\ref{theo:approxSol}. For the derivation of a convergence rate, however, additional knowledge of the remainder $r_N$ is needed.
	\item
	Note that $Z_{k_{N,N}}(x) = \sum_{i=1}^{N} \lambda_{N,i} e_{N,i}(x) Z(e_{N,i})$ depends only on the finite-dimensional L\'evy distribution of $(Z(e_{N,1}),\ldots,Z(e_{N,N}))$. 
Thus, Theorem~\ref{theo:ApproxKL} provides an approximation scheme for $u$ obtained from coefficients from an infinite-dimensional distribution by $u_{N,N}$ obtained from a finite-dimensional distributions.
\end{enumerate}
\end{remark}

\section{Outlook}

\label{sec:outlook}

We have established a comprehensive theory for the existence, integrability and finite-dimensional approximation of solutions to the random PDEs \eqref{eq:random-diff} with conductivity given by transformed L\'evy random fields including rates of convergence. 

At this point it seems natural to proceed with a numerical treatment of the PDEs based on the established stochastic approximation scheme. 
However, this is not straightforward as the uncorrelated L\'evy random variables $(Z(e_1),\ldots,Z(e_n))$ of the finite-dimensional approximation will, in general, fail to be independent.
Standard (sparse) tensor quadrature formulae for numerically computing the expected value of quantities of interest are therefore not applicable without modification. 
Research on quadrature rules to numerically integrate high-dimensional L\'evy distributions is needed.

The statistical investigation of the actual distribution of, e.g., hydraulic conductivity in groundwater problems, is necessary to further clarify the relevance of L\'evy models. 
The option to insert discrete regions of enhanced conductivity via non-isotropic kernel functions $k$ with random orientation applied to smooth Poisson noise seems adequate to model crack-like structures in the subsurface. 
Only minor changes to the theory presented here would be needed to cover this case as well.  

Considering statistical aspects further, the results in Section \ref{section:approx} allow an interesting additional application. Even in the Gaussian case, a proof of the robustness of uncertainty quantification methods under statistical estimation error of the covariance or semi-variogram functions is missing. 
We suggest that our arguments used in the proof of the approximation can be adapted to show that given a consistent estimation of the covariance function, the expected value of quantities of interest converges in the large sample limit.   

In this work, we have treated L\'evy processes with a Poisson part that permits infinite activity $\int_{\{|s|<1\}} \nu(\d s)=\infty$  while satisfying $\int_{\{|s|<1\}} |s|\nu(\d s)<\infty$. 
This allowed us to shift the compensator term $i t s\mathds{1}_{\{|s|\leq 1\}}(s) $ for small jumps in the L\'evy characteristic to the constant $b$, cf.\ Definition \ref{def:levy}. 
L\'evy measures for which only $\int_{\{|s|<1\}} |s|^2\nu(\d s)<\infty$ require a different set of tail estimates to be developed. 
At the same time it would be of interest to weaken the integrability conditions for the L\'evy measure and allow for thicker tails for $\nu(\d s)$ for large $s$. 
E.g. the methods presented here are far from being applicable to  $\alpha$-stable L\'evy fields with extremely fat tails. 

Another interesting direction of research is to work directly with L\'evy random fields with positive paths making the transformation $T(z)$ unnecessary. 
E.g. smoothed Gamma noise with a positive kernel function $k(x)$ is an interesting candidate as gamma and lognormal distributions are rather similar. 
If one avoids using the transformation $T(z)$, the maximum value problem for $Z_k(x)$, however, is turned in a minimum value problem, which requires a rather different set of techniques to prove analogous results to those given in this paper.

\vspace{.2cm}

\appendix
\section{Proof of Lemma \ref{lem:matern}}\label{appendix b}

In this appendix we give a proof of Lemma \ref{lem:matern}. 
Recall the Fourier transform 
\[
    \hat k_{\alpha,m} \colon \R^d\rightarrow\R, \quad \xi \mapsto \frac{1}{(|\xi|^2+m^2)^\alpha},
    \qquad
    \alpha\in\R, \; m>0, 
\]
of the Matérn kernel $k_{\alpha,m} \in \mathscr{S}'(\R^d)$ (cf.\ Definition~\ref{def:matern}).
In the notation introduced in Section~\ref{sub:Schwartz}, we seek to determine for fixed $q\in\N_0$ those $\alpha\in\R$ such that
\begin{equation}\label{preparation convolution}
	\tau_y\Big(k_{\alpha,m}^\vee\Big) \in \mathscr{S}_q(\R^d) \quad \forall\,y\in\R^d,
\end{equation}
where $\tau_y$ denotes translation by $y$ and $\vee$ reflection at the origin. 
If (\ref{preparation convolution}) holds and $\omega\in\mathscr{S}_q'(\R^d)$ then the convolution
\[
	\omega*k_{\alpha,m} \colon \R^d\rightarrow\C,
	\quad
	y\mapsto\left\langle \omega,\tau_y\left(k_{\alpha,m}^\vee\right)\right\rangle
\]
is well-defined. 
Moreover, if not only (\ref{preparation convolution}) is satisfied but also the mapping
\[
	\R^d\rightarrow(\mathscr{S}_q(\R^d),|\cdot|_q),
	\quad
	y\mapsto \tau_y\big(k_{\alpha,m}^\vee\big)
\]
is continuous, the convolution yields a continuous function as well. In order to investigate the validity of (\ref{preparation convolution}) together with the continuous dependence on $y$, some technical preparations are necessary. 
We first observe that
\[
	\tau_y\left(\left(\mathscr{F}^{-1}\left(\frac{1}{(|\xi|^2+m^2)^\alpha}\right)\right)^\vee\right)
	=
	\mathscr{F}^{-1}\left(\frac{\e^{i \xi \cdot y}}{(|\xi|^2+m^2)^\alpha}\right),
\]
where $\xi \cdot y$ denotes the Euclidean scalar product of $\xi, y\in\R^d$.

Because the Fourier transform commutes with the operator $(|\xi|^2-\Delta)$, (\ref{preparation convolution}) will follow if $\alpha\in\R, m>0$ are such that
\begin{equation} \label{preparation convolution 1}
    \forall\,y\in\R^d:\quad
    \big(|\xi|^2-\Delta\big)^q\left[\frac{\e^{i \xi \cdot y}}{(|\xi|^2+m^2)^\alpha}\right]\in L^2(\R^d).
\end{equation}
In order to determine those values of $\alpha$ which satisfy the above property, we will apply part iii) of the following lemma to $f = \hat k_{\alpha,m}$. 
Recall that a smooth function is said to be of moderate growth if each of its partial derivatives is polynomially bounded.

\begin{lemma}\label{preparation convolution 2}
	\hspace{2ex}
	\begin{enumerate}
	\item[i)] 
	For every smooth function $Q$ on $\R$ which is of moderate growth and every $y\in\R^d$ there holds that for all $f\in\mathscr{S}'(\R^d)$
		\[
		\left(|\xi|^2-\Delta\right)\left[Q(\xi \cdot y) f\right]
		=
		Q(\xi \cdot y)\left(|\xi|^2-\Delta\right)f-|y|^2Q''(\xi \cdot y) f-2Q'(\xi \cdot y)\langle y,\nabla\rangle f,
		\]
		where $\langle y,\nabla\rangle=\sum_{j=1}^d y_j\partial_j$ denotes the derivative in direction $y$.
	\item[ii)] 
		For every $r\in\N$, $y\in\R^d$, and each $f\in\mathscr{S}'(\R^d)$ we have
		$$\left(|\xi|^2-\Delta\right)\langle y,\nabla\rangle^r f=\begin{cases}
		\mbox{if }r=1:&\langle y,\nabla\rangle\left(|\xi|^2-\Delta\right)f-2\xi \cdot y\,f,\\
		\mbox{if }r\geq 2:&\langle y,\nabla\rangle^r\left(|\xi|^2-\Delta\right)f\\
		&-r(r-1)|y|^2\langle y,\nabla\rangle^{r-2} f\\
		&-2r\,\xi \cdot y\langle y,\nabla\rangle^{r-1}f.
		\end{cases}$$
		
	\item[iii)] 
	For every $p\in\N$ there are $k_l\in\N, r_{n,l}\in\N_0$ and univariate polynomials $P_{n,l}, Q_{n,l}$, where $l\in\{0,\ldots,\max\{0,p-2\}\}, n\in\{1,\ldots, k_l\}$, such that for all $f\in\mathscr{S}'(\R^d)$ and $y\in\R^d$
		\begin{align*}
		\left(|\xi|^2-\Delta\right)^p(\e^{i \xi \cdot y}f)&
		=
		\e^{i\xi \cdot y}\left(\left(|\xi|^2-\Delta\right)^p f-2ip\langle y,\nabla\rangle\left(|\xi|^2-\Delta\right)^{p-1}f\right.\\
		&\hspace{6ex}+p|y|^2\left(|\xi|^2-\Delta\right)^{p-1}f\\
		&\hspace{6ex}\left.+\sum_{l=0}^{p-2}\sum_{n=1}^{k_l}P_{n,l}(|y|^2) Q_{n,l}(\xi \cdot y)\langle y,\nabla\rangle^{r_{n,l}}\big(|\xi|^2-\Delta\big)^l f \right)
		\end{align*}
		and such that $\deg Q_{n,l}\leq p-1$, $r_{n,l}\leq p$ as well as
		$$\deg P_{n,l}+\deg Q_{n,l}+r_{n,l}+l\leq p+1$$
		for all $l\in\{0,\ldots,\max\{0,p-2\}\}, n\in\{1,\ldots,k_l\}$.
	\end{enumerate}
\end{lemma}

\begin{proof}
	A direct calculation shows that i) holds.
	
	ii) Since the constant coefficient differential operators $\langle y,\nabla\rangle$ and $\Delta$ commute,
	\[
	\langle y,\nabla\rangle\big(|\xi|^2 f-\Delta f\big)
	=
	2 \xi\cdot y\,f+|\xi|^2\langle y,\nabla\rangle f - \Delta\langle y,\nabla\rangle f,
	\]
	which yields the assertion for $r=1$. Using the just proved equality twice, it is straightforward to show that the asserted equation is true for $r=2$.
	Assuming that the equation holds for $r\geq 2$ a straightforward calculation gives
	\begin{align*}
	\left(|\xi|^2-\Delta\right)\langle y,\nabla\rangle^{r+1}f
	=&\langle y,\nabla\rangle^{r+1}\left(|\xi|^2-\Delta\right) f -2r(r+1)|y|^2\langle y,\nabla\rangle^{r-1}f\\
	&-2(r+1)\xi \cdot y\langle y,\nabla\rangle^r f
	\end{align*}
	proving ii).
	
	iii) We prove the claim by induction on $p$. For $p=1$, part i) yields
	\[
	(|\xi|^2-\Delta)(\e^{i \xi \cdot y}f)
	=
	\e^{i\xi \cdot y}\left(\left(|\xi|^2-\Delta\right)f-2i\langle y,\nabla\rangle f+|y|^2f\right).
	\]
	Assuming that the claim holds for $p\in\N$ we obtain by the induction hypothesis and the case $p=1$ that
	\begin{align}
	&\left(|\xi|^2-\Delta\right)^{p+1}(\e^{i \xi \cdot y}f)\nonumber\\
	=&
	\e^{i\xi \cdot y}\left(\left(|\xi|^2-\Delta\right)^{p+1} f-2ip\left(|\xi|^2-\Delta\right)\langle y,\nabla\rangle\left(|\xi|^2-\Delta\right)^{p-1}f+p|y|^2\left(|\xi|^2-\Delta\right)^p f\nonumber\right.\\
	&\hspace{6ex}+\sum_{l=0}^{p-2}\sum_{n=1}^{k_l}\left(|\xi|^2-\Delta\right)\left[P_{n,l}(|y|^2) Q_{n,l}(\xi \cdot y)\langle y,\nabla\rangle^{r_{n,l}}\left(|\xi|^2-\Delta\right)^l f\right]\nonumber\\
	&-2i\langle y,\nabla\rangle\left(|\xi|^2-\Delta\right)^p f-4p\langle y,\nabla\rangle^2\left(|\xi|^2-\Delta\right)^{p-1}f\nonumber\\
	&\hspace{6ex}-2i p|y|^2\langle y,\nabla\rangle\left(|\xi|^2-\Delta\right)^{p-1}f\label{ugly calculation 1}\\
	&\hspace{6ex}-\sum_{l=0}^{p-2}\sum_{n=1}^{k_l}2i \langle y,\nabla\rangle\left[P_{n,l}(|y|^2) Q_{n,l}(\xi \cdot y)\langle y,\nabla\rangle^{r_{n,l}}\left(|\xi|^2-\Delta\right)^l f\right]\nonumber\\
	&+|y|^2\left(|\xi|^2-\Delta\right)^p f-2ip|y|^2\left(|\xi|^2-\Delta\right)^{p-1}f+p(|y|^2)^2\left(|\xi|^2-\Delta\right)^{p-1}f\nonumber\\
	&\hspace{6ex}+\left.\left.\sum_{l=0}^{p-2}\sum_{n=1}^{k_l}|y|^2P_{n,l}(|y|^2) Q_{n,l}(\xi \cdot y)\langle y,\nabla\rangle^{r_{n,l}}\left(|\xi|^2-\Delta\right)^l f\right]\right).\nonumber
	\end{align}
	For the second double sum on the right hand side of the above equality (\ref{ugly calculation 1}) we calculate
	\begin{align} \label{ugly calculation 2}
	&\sum_{l=0}^{p-2}\sum_{n=1}^{k_l}2i \langle y,\nabla\rangle\left[P_{n,l}(|y|^2) Q_{n,l}(\xi \cdot y)\langle y,\nabla\rangle^{r_{n,l}}\left(|\xi|^2-\Delta\right)^l f\right]\nonumber\\
	&=\sum_{l=0}^{p-2}\sum_{n=1}^{k_l}2iP_{n,l}(|y|^2)Q_{n,l}'(\xi \cdot y)|y|^2\langle y,\nabla\rangle^{r_{n,l}}\left(|\xi|^2-\Delta\right)^l f\\
	&\hspace{6ex}+\sum_{l=0}^{p-2}\sum_{n=1}^{k_l}2iP_{n,l}(|y|^2)Q_{n,l}(\xi \cdot y)\langle y,\nabla\rangle^{r_{n,l}+1}\left(|\xi|^2-\Delta\right)^l f\nonumber,
	\end{align}
	while an application of parts i) and ii) to the summands of the first double sum on the right hand side of equality (\ref{ugly calculation 1}) combined with $\Delta_\xi Q_{n,l}(\xi \cdot y)=Q_{n,l}''(\xi \cdot y)|y|^2$ and $\nabla_\xi Q_{n,l}(\xi \cdot y) = Q_{n,l}'(\xi \cdot y)y$ gives
	\begin{align}\label{ugly calculation 3}
	&\left(|\xi|^2-\Delta\right)\left[P_{n,l}(|y|^2) Q_{n,l}(\xi \cdot y)\langle y,\nabla\rangle^{r_{n,l}}\left(|\xi|^2-\Delta\right)^l f\right]\\
	=&\begin{cases}
	\mbox{if }r_{n,l}=0:& P_{n,l}(|y|^2)Q_{n,l}(\xi \cdot y)\left(|\xi|^2-\Delta\right)^{l+1}f\\
	& - P_{n,l}(|y|^2)|y|^2Q_{n,l}''(\xi \cdot y)\left(|\xi|^2-\Delta\right)^l f\nonumber\\
	& -2 P_{n,l}(|y|^2)Q_{n,l}'(\xi \cdot y)\langle y,\nabla\rangle\left(|\xi|^2-\Delta\right)^l f\nonumber\\
	&\nonumber\\
	\mbox{if }r_{n,l}=1:& P_{n,l}(|y|^2)Q_{n,l}(\xi \cdot y)\langle y,\nabla\rangle\left(|\xi|^2-\Delta\right)^{l+1}f\nonumber\\
	& - P_{n,l}(|y|^2)|y|^2Q_{n,l}''(\xi \cdot y)\langle y,\nabla\rangle\left(|\xi|^2-\Delta\right)^l f\nonumber\\
	& -2 P_{n,l}(|y|^2)Q_{n,l}'(\xi \cdot y)\langle y,\nabla\rangle^2\left(|\xi|^2-\Delta\right)^l f\nonumber\\
	&-2P_{n,l}(|y|^2)Q_{n,l}(\xi \cdot y)\xi \cdot y\left(|\xi|^2-\Delta\right)^l f \\
	& \nonumber\\
	\mbox{if }r_{n,l}\geq 2:& P_{n,l}(|y|^2)Q_{n,l}(\xi \cdot y)\langle y,\nabla\rangle^{r_{n,l}}\left(|\xi|^2-\Delta\right)^{l+1} f\nonumber\\
	& - P_{n,l}(|y|^2)|y|^2Q_{n,l}''(\xi \cdot y)\langle y,\nabla\rangle^{r_{n,l}}\left(|\xi|^2-\Delta\right)^l f\nonumber\\
	& -2 P_{n,l}(|y|^2)Q_{n,l}'(\xi \cdot y)\langle y,\nabla\rangle^{r_{n,l}+1}\left(|\xi|^2-\Delta\right)^l f\nonumber\\
	& -2r_{n,l}P_{n,l}(|y|^2)Q_{n,l}(\xi \cdot y)\xi \cdot y \langle y,\nabla\rangle^{r_{n,l}-1}\left(|\xi|^2-\Delta\right)^l f\nonumber\\
	& -r_{n,l}(r_{n,l}-1)P_{n,l}(|y|^2)|y|^2 Q_{n,l}(\xi \cdot y)\langle y,\nabla\rangle^{r_{n,l}-2}\left(|\xi|^2-\Delta\right)^l f.\nonumber
	\end{cases}
	\end{align}
	By this equality, it holds for the first double sum on the right hand side of (\ref{ugly calculation 1})
	\begin{align}\label{ugly calculation 4}
	&\sum_{l=0}^{p-2}\sum_{n=1}^{k_l}\left(|\xi|^2-\Delta\right)\left[P_{n,l}(|y|^2) Q_{n,l}(\xi \cdot y)\langle y,\nabla\rangle^{r_{n,l}}\left(|\xi|^2-\Delta\right)^l f\right]\nonumber\\
	=&\sum_{l=0}^{p-2}\sum_{n=1}^{k_l}P_{n,l}(|y|^2)Q_{n,l}(\xi \cdot y)\langle y,\nabla\rangle^{r_{n,l}}\left(|\xi|^2-\Delta\right)^{l+1} f\nonumber\\
	&-\sum_{l=0}^{p-2}\sum_{n=1}^{k_l}P_{n,l}(|y|^2)|y|^2Q_{n,l}''(\xi \cdot y)\langle y,\nabla\rangle^{r_{n,l}}\left(|\xi|^2-\Delta\right)^l f\nonumber\\
	&-\sum_{l=0}^{p-2}\sum_{n=1}^{k_l}P_{n,l}(|y|^2)2Q_{n,l}'(\xi \cdot y)\langle y,\nabla\rangle^{r_{n,l}+1}\left(|\xi|^2-\Delta\right)^l f\\
	&-\sum_{l=0}^{p-2}\sum_{n=1;r_{n,l}\geq 1}^{k_l}2r_{n,l}P_{n,l}(|y|^2)Q_{n,l}(\xi \cdot y)\xi \cdot y\langle y,\nabla\rangle^{r_{n,l}-1}\left(|\xi|^2-\Delta\right)^l f\nonumber\\
	&-\sum_{l=0}^{p-2}\sum_{n=1;r_{n,l}\geq 2}^{k_l} r_{n,l}(r_{n,l}-1)P_{n,l}(|y|^2)|y|^2 Q_{n,l}(\xi \cdot y)\langle y,\nabla\rangle^{r_{n,l}-2}\left(|\xi|^2-\Delta\right)^l f.\nonumber
	\end{align}
	Inserting (\ref{ugly calculation 4}) and (\ref{ugly calculation 2}) into the first double sum and second double sum of the right hand side in (\ref{ugly calculation 1}), respectively, applying part ii) to the second summand, and rearranging the terms we derive
\begin{align*}
	&\big(|\xi|^2-\Delta\big)^{p+1}(\e^{i \xi \cdot y}f)\\
	=
	&\e^{i\xi \cdot y}\Big(\big(|\xi|^2-\Delta\big)^{p+1} f-2i(p+1)\langle y,\nabla\rangle\big(|\xi|^2-\Delta\big)^p f+(p+1)|y|^2\big(|\xi|^2-\Delta\big)^p f\\
	&\hspace{6ex}+\sum_{l=0}^{p-2}\sum_{n=1}^{k_l}P_{n,l}(|y|^2)Q_{n,l}(\xi \cdot y)\langle y,\nabla\rangle^{r_{n,l}}\big(|\xi|^2-\Delta\big)^{l+1} f\\
	&\hspace{6ex}+\sum_{l=0}^{p-2}\sum_{n=1}^{k_l}P_{n,l}(|y|^2)|y|^2\big(Q_{n,l}(\xi \cdot y)-2iQ_{n,l}'(\xi \cdot y)-Q_{n,l}''(\xi \cdot y)\big)\cdot\\
	&\hspace{36ex}\cdot \langle y,\nabla\rangle^{r_{n,l}}\big(|\xi|^2-\Delta\big)^l f\\
	&\hspace{6ex}+\sum_{l=0}^{p-2}\sum_{n=1}^{k_l}P_{n,l}(|y|^2)\big(2iQ_{n,l}(\xi \cdot y)-2Q_{n,l}'(\xi \cdot y)\big)\langle y,\nabla\rangle^{r_{n,l}+1}\big(|\xi|^2-\Delta\big)^l f\\
	&\hspace{6ex}-\sum_{l=0}^{p-2}\sum_{n=1;r_{n,l}\geq 1}^{k_l}2r_{n,l}P_{n,l}(|y|^2)Q_{n,l}(\xi \cdot y)\xi \cdot y\langle y,\nabla\rangle^{r_{n,l}-1}\big(|\xi|^2-\Delta\big)^l f\\
	&\hspace{6ex}-\sum_{l=0}^{p-2}\sum_{n=1;r_{n,l}\geq 2}^{k_l} r_{n,l}(r_{n,l}-1)P_{n,l}(|y|^2)|y|^2 Q_{n,l}(\xi \cdot y)\langle y,\nabla\rangle^{r_{n,l}-2}\big(|\xi|^2-\Delta\big)^l f\\
	&+4ip \xi \cdot y\big(|\xi|^2-\Delta\big)^{p-1} f-4p\langle y,\nabla\rangle^2\big(|\xi|^2-\Delta\big)^{p-1}f-2i p|y|^2\langle y,\nabla\rangle\big(|\xi|^2-\Delta\big)^{p-1}f\\
	&-2ip|y|^2\big(|\xi|^2-\Delta\big)^{p-1}f+p\,(|y|^2)^2\big(|\xi|^2-\Delta\big)^{p-1}f\Big).
	\end{align*}
	Thus, $\big(|\xi|^2-\Delta\big)^{p+1}(\e^{i\xi \cdot y}f)$ is of the asserted form. Moreover, since the powers of the directional derivative $\langle y,\nabla\rangle$ are at most $\max\{2,r_{n,l}+1\}$, by the induction hypothesis, they are bounded by $p+1$. Furthermore, 
	the polynomials of one variable $t$ which are applied to the scalar product $\xi \cdot y$ are either $Q_{n,l}$, $Q_{n,l}-2iQ_{n,l}'-Q_{n,l}''$, $2i Q_{n,l}-2Q_{n,l}'$, $t\mapsto Q_{n,l}(t)t$, or $t\mapsto t$, which by the induction hypothesis implies that their 
	respective degree is bounded above by $p$. Finally, using again the induction hypothesis, it follows that in each summand of the above expression the sum of the degrees of the polynomials in $|y|^2$, $\xi \cdot y$, the power of $\langle y,\nabla\rangle$, and the power 
	of $\big(|\xi|^2-\Delta\big)$ is bounded above by $p+1$ which proves the assertion for $p+1$ and gives iii).
\end{proof}

Our objective is to apply part iii) of the previous lemma to $f(\xi)=\frac{1}{(|\xi|^2+m^2)^\alpha}$ in order to derive for which $\alpha\in\R, m>0$
\[
	\forall\,y\in\R^d:\,
	\left(|\xi|^2-\Delta\right)^q\left[\frac{\e^{i\xi \cdot y}}{(|\xi|^2+m^2)^\alpha}\right]\in L^2(\R^d)
\]
holds. 
For this we still need one more technical result.

\begin{proposition}\label{preparation convolution 3}
	Let $\alpha\in\R$ and $m>0$.
	\begin{enumerate}
		\item[i)] Let $A$ be a polynomial of a single variable and $y\in\R^d$. Denoting for $r\in\N$ the integer part of $\frac{r}{2}$ by $\floor*{\frac{r}{2}}$ there are polynomials $A_0,\ldots A_{\floor*{\frac{r}{2}}}$ of a single variable such that $\deg A_j=\deg A+j$ and
		$$\langle y,\nabla\rangle^r\left[\frac{A(|\xi|^2)}{(|\xi|^2+m^2)^\alpha}\right]=\frac{\sum_{j=0}^{\floor*{\frac{r}{2}}} A_j(|\xi|^2)(\xi \cdot y)^{r-2j}|y|^{2j}}{(|\xi|^2+m^2)^{\alpha+r}}.$$
		\item[ii)] For every $p\in\N_0$ there is a univariate polynomial $Q$ with $\deg Q_p=3p$ such that
		$$\big(|\xi|^2-\Delta\big)^p\frac{1}{(|\xi|^2+m^2)^\alpha}=\frac{Q_p(|\xi|^2)}{(|\xi|^2+m^2)^{\alpha+2p}}.$$
	\end{enumerate}
\end{proposition}

\begin{proof}
	The assertion in i) is clearly true for $y=0$. For $y\neq 0$, multiplying both sides of the asserted equality by $|y|^{-r}$ we see that we can assume without loss of generality that $|y|=1$. For $|y|=1$ we prove i) by induction on $r$. For $r=1$ a straightforward calculation shows that
	$$\langle y,\nabla\rangle\left[\frac{A(|\xi|^2)}{(|\xi|^2+m^2)^\alpha}\right]=\frac{2 \big(A'(|\xi|^2)(|\xi|^2+m^2)-\alpha A(|\xi|^2)\big)\xi \cdot y}{(|\xi|^2+m^2)^{\alpha+1}}.$$
	Since no polynomial is a solution to the ordinary differential equation $u'(t)(t+m^2)-c u(t)=0, c\in\R,$ the degree of the polynomial
	$$\R\ni t\mapsto A'(t)(t+m^2)-\alpha A(t)$$
	equals $\deg A$ which proves the claim for $r=1$.
	
	Assume the claim to be true for $r\in\N$. Taking into account that $r$ is odd precisely when $r-2\floor*{\frac{r}{2}}=1$, or when $\floor*{\frac{r}{2}}+1=\floor*{\frac{r+1}{2}}$, and using $|y|=1$ as well as the induction hypothesis, it follows
	\begin{align*}
	&\langle y,\nabla\rangle^{r+1}\left[\frac{A(|\xi|^2)}{(|\xi|^2+m^2)^\alpha}\right]=\langle y,\nabla\rangle\left[\frac{\sum_{j=0}^{\floor*{\frac{r}{2}}} A_j(|\xi|^2)(\xi \cdot y)^{r-2j}}{(|\xi|^2+m^2)^{\alpha+r}}\right]\\
	=&\frac{(2A_0'(|\xi|^2)(|\xi|^2+m^2)-2(\alpha+r)A_0(|\xi|^2))(\xi \cdot y)^{r+1}}{(|\xi|^2+m^2)^{\alpha+r+1}}\\
	&\hspace{3ex}+\frac{\sum_{j=1}^{\floor*{\frac{r}{2}}}(2A_j'(|\xi|^2)+(r+2(j-1))A_{j-1}(|\xi|^2))(|\xi|^2+m^2)(\xi \cdot y)^{r+1-2j}}{(|\xi|^2+m^2)^{\alpha+r+1}}\\
	&\hspace{3ex}-\frac{\sum_{j=1}^{\floor*{\frac{r}{2}}}2(\alpha+r)A_j(|\xi|^2)(\xi \cdot y)^{r+1-2j}}{(|\xi|^2+m^2)^{\alpha+r+1}}\\
	&\hspace{3ex}+\begin{cases}
	\mbox{if }r\mbox{ is odd}:&\frac{A_{\floor*{\frac{r}{2}}}(|\xi|^2)(|\xi|^2+m^2){(|\xi|^2+m^2)^{\alpha+r+1}}(\xi \cdot y)^{r+1-2\floor*{\frac{r+1}{2}}}}{(|\xi|^2+m^2)^{\alpha+r+1}}\\
	\mbox{if }r\mbox{ is even}:&0.
	\end{cases}
	\end{align*}
	Using the induction hypothesis once more we see that the degree of the polynomial
	$$\tilde{A}_0:\R\ni t\mapsto 2A_0'(t)(t+m^2)-2(\alpha+r)A_0(t)$$
	is equal to $\deg A_0=\deg A$ while for $1\leq j\leq\floor*{\frac{r}{2}}$ the degree of the polynomials 
	$$\tilde{A}_j:\R\ni t\mapsto(2A_j'(t)+(r+2(j-1))A_{j-1}(t))(t+m^2)-2(\alpha+r)A_j(t)$$
	satisfies $\deg(\tilde{A}_j)=\deg A_j=\deg A+j$. Finally, for odd $r$ the degree of the polynomial
	$$\tilde{A}_{\floor*{\frac{r+1}{2}}}:\R\ni t\mapsto A_{\floor*{\frac{r}{2}}}(t)(t+m^2)$$
	is equal to $\deg A_{\floor*{\frac{r}{2}}}+1=\deg A+ \floor*{\frac{r}{2}}+1$ by the induction hypothesis. Because by the above we have
	$$\langle y,\nabla\rangle^{r+1}\Big[\frac{A(|\xi|^2)}{(|\xi|^2+m^2)^\alpha}\Big]=\frac{\sum_{j=0}^{\floor*{\frac{r+1}{2}}} \tilde{A}_j(|\xi|^2)(\xi \cdot y)^{r+1-2j}}{(|\xi|^2+m^2)^{\alpha+r+1}}$$
	it follows that the claim also holds for $r+1$ which proves i).
	
	In order to prove ii), we switch to polar coordinates (writing $r^2=|\xi|^2$ as usual), so we actually claim that  for every $p\in\N_0$ there is a univariate polynomial $Q_p$ with real coefficients and of degree $3p$ such that
	\begin{equation}
	\Big(r^2-\frac{1}{r^{d-1}}\frac{\partial}{\partial r}\big(r^{d-1}\frac{\partial}{\partial r}\big)\Big)^p\frac{1}{(r^2+m^2)^\alpha}=\frac{Q_p(r^2)}{(r^2+m^2)^{\alpha+2p}}.
	\end{equation}
	Indeed, for $p=0$ the claim is obviously true. Assuming the claim to be true for $p\in\N_0$ a tedious but straightforward calculation yields
	$$\Big(r^2-\frac{1}{r^{d-1}}\frac{\partial}{\partial r}\big(r^{d-1}\frac{\partial}{\partial r}\big)\Big)^{p+1}\frac{1}{(r^2+m^2)^\alpha}
	=\frac{Q_{p+1}(r^2)}{(r^2+m^2)^{\alpha+2(p+1)}},$$
	where
	\begin{align*}
	Q_{p+1}(t)&:=t(t+m^2)^2Q_p(t)-d(t+m^2)\big(Q'_p(t)2(t+m^2)-(2\alpha+4p)Q_p(t)\big)\\
	&-4t(t+m^2)\big(Q''_p(t)(t+m^2)+(1-\alpha-2p)Q'_p(t)\big)\\
	&+(2\alpha+4p+2)t\big(Q'_p(t)2(t+m^2)-(2\alpha+4p)Q_p(t)\big)
	\end{align*}
	is a polynomial with real coefficients of degree $\mbox{deg }Q_p+3=3(p+1)$, which proves ii).
\end{proof}

We subsume the results so far obtained in the next proposition which is an immediate consequence of Proposition \ref{preparation convolution 3} ii) applied to each summand obtained by applying Lemma \ref{preparation convolution 2} iii) to $f = k_{\alpha,m}$.

\begin{proposition}\label{convolution}
Let $q\in\N_0$ be fixed. 
Moreover, let $\alpha\in\R$ and $m>0$. 
Then, there are $k_l\in\N, 0\leq r_{n,l}\leq q, 0\leq l\leq \max\{0, q-2\}, 1\leq n\leq k_l$ and univariate polynomials  $P_{n,l}$, $Q_{n,l}$ with $\deg Q_{n,l}\leq q-1$, $\tilde{Q}_q$ with $\deg\tilde{Q}_q=3q$, 
	$\tilde{Q}_{q-1}$ with $\deg\tilde{Q}_{q-1}=3(q-1)$, $\tilde{Q}_{q-1,1}$ with $\deg\tilde{Q}_{q-1,1}=3(q-1)$, and $\tilde{Q}_{l,j,r_{n,l}}$ with $\deg\tilde{Q}_{l,j,r_{n,l}}=3l+j$, where $0\leq l\leq \max\{0, q-2\}$, $1\leq n\leq k_l$, $0\leq j\leq \floor*{\frac{r_{n,l}}{2}}$ such 
	that for all $y\in\R^d$
\begin{align}\label{explicit formula}
	&\left(|\xi|^2-\Delta\right)^q\left(\frac{\e^{i\xi \cdot y}}{(|\xi|^2+m^2)^\alpha}\right)
	=
	\e^{i\xi \cdot y}\left(\frac{\tilde{Q}_q(|\xi|^2)}{(|\xi|^2+m^2)^{\alpha+2q}}\right.\nonumber\\
	&\hspace{8ex}
	-2iq \frac{\tilde{Q}_{q-1,1}(|\xi|^2)\xi \cdot y}{(|\xi|^2+m^2)^{\alpha+2(q-1)+1}}+q|y|^2\frac{\tilde{Q}_{q-1}(|\xi|^2)}{(|\xi|^2+m^2)^{\alpha+2(q-1)}}\\
		&\hspace{8ex}\left.+\sum_{l=0}^{q-2}\sum_{n=1}^{k_l}\sum_{j=0}^{\floor*{\frac{r_{n,l}}{2}}}\frac{P_{n,l}(|y|^2)|y|^{2j}Q_{n,l}(\xi \cdot y)\tilde{Q}_{l,j,r_{n,l}}(|\xi|^2)(\xi \cdot y)^{r_{n,l}-2j}}{(|\xi|^2+m^2)^{\alpha+2l+r_{n,l}}}\right).\nonumber
\end{align}
\end{proposition}

Now we have everything at our disposal to prove Lemma \ref{lem:matern}.\\

\textit{Proof of Lemma \ref{lem:matern}.}
	Before we prove the implications asserted in Lemma \ref{lem:matern}, we consider when for fixed $y\in\R^d$ each of the summands in (\ref{explicit formula}) belongs to $L^2(\R^d)$. Because for fixed $y\in\R^d$ each summand is a continuous function of 
	$\xi\in\R^d$, we can assume without loss of generality that $|\xi|\geq 1$ in the following considerations.
	
	While for fixed $y\in\R^d$ the first summand in (\ref{explicit formula}) belongs to $L^2(\R^d)$ whenever $\alpha>\frac{d}{4}+q$ the second and third summand in (\ref{explicit formula}) are in $L^2(\R^d)$ for $\alpha>\frac{d}{4}+q-\frac{3}{2}$ and $\alpha>\frac{d}{4}+q-1$, respectively. Finally, an application of the Cauchy-Schwarz inequality shows that 
	each term appearing in the triple sum in (\ref{explicit formula}) belongs to $L^2(\R^d)$ for fixed $y\in\R^d$ if $\alpha>\frac{d}{4}+\frac{3}{2}(q-1)$.
	
Hence, a sufficient condition on $\alpha>0$ (and $m\in\R$) in order that
\[
	\forall\,y\in\R^d:\,
	\R^d\rightarrow\C,
	\quad
	\xi \mapsto\left(|\xi|^2-\Delta\right)^q\left(\frac{\e^{i\xi \cdot y}}{(|\xi|^2+m^2)^\alpha}\right)\in L^2(\R^d)
\]
holds is $\alpha>\max\{\frac{d}{4}+q,\frac{d}{4}+\frac{3}{2}(q-1)\}=\frac{d}{4}+q+\max\{0,\frac{q-3}{2}\}$, so that i) implies iii). Moreover, if i) holds, employing (\ref{explicit formula}) for $y_0\in\R^d$ and $y\in\R^d$ from a fixed compact neighborhood of $y_0$ it follows that
\begin{align*}
	&\left\|(|\xi|^2-\Delta)^q\left(\frac{\e^{i\xi \cdot y_0}}{(|\xi|^2+m^2)^\alpha}\right)-(|\xi|^2-\Delta)^q\left(\frac{\e^{i\xi \cdot y}}{(|\xi|^2+m^2)^\alpha}\right)\right\|_{L^2}^2\\
	=
	&\int_{\R^d}\left|(\e^{i\xi \cdot y_0}-\e^{i\xi \cdot y})\frac{\tilde{Q}_q(|\xi|^2)}{(|\xi|^2+m2)^{\alpha+2q}}\right.\\
	&\hspace{4ex}-2iq(\e^{i\xi \cdot y_0}\xi \cdot y_0-\e^{i\xi \cdot y}\xi \cdot y)\frac{\tilde{Q}_{q-1,1}(|\xi|^2)}{(|\xi|^2+m^2)^{\alpha+2(q-1)+1}}\\
	&\hspace{4ex}+q(\e^{i\xi \cdot y_0}|y_0|^2-\e^{i\xi \cdot y}|y|^2)\frac{\tilde{Q}_{q-1}(|\xi|^2)}{(|\xi|^2+m^2)^{\alpha+2(q-1)}}\\
	&\hspace{4ex}+\sum_{l=0}^{q-2}\sum_{n=1}^{k_l}\sum_{j=0}^{\floor*{\frac{r_{n,l}}{2}}}(\e^{i\xi \cdot y_0}P_{n,l}(|y_0|^2)|y_0|^{2j}Q_{n,l}(\xi \cdot y_0)(\xi \cdot y_0)^{r_{n,l}-2j}\\
	&\hspace{6ex}\left.-\e^{i\xi \cdot y}P_{n,l}(|y|^2)|y|^{2j}Q_{n,l}(\xi \cdot y)(\xi \cdot y)^{r_{n,l}-2j})\frac{\tilde{Q}_{l,j,r_{n,l}}(|\xi|^2)}{(|\xi|^2+m^2)^{\alpha+2l+r_{n,l}}}\right|^2\,d\xi.
\end{align*}
Since $y$ belongs to a fixed compact neighborhood of $y_0$ the above integrands have an integrable majorant independent of $y$ so that Lebesgue's Dominanted Convergence Theorem implies the continuity of
\[
   \R^d\rightarrow L^2(\R^d),
   \quad
   y\mapsto (|\xi|^2-\Delta)^q \left(\frac{\e^{i\xi \cdot y}}{(|\xi|^2+m^2)^{\alpha}}\right)
\]
in $y_0$. 
Since $y_0$ was chosen arbitrarily we conclude continuity of
\[
	\R^d\rightarrow (\mathscr{S}_q(\R^d),|\cdot|_q),
	\quad
	y\mapsto \frac{\e^{i\xi \cdot y}}{(|\xi|^2+m^2)^{\alpha}}.
\]
Since the Hermite functions $h_\alpha$ are eigenfunctions of the Fourier transform and the corresponding eigenvalues are unimodular, it follows that the Fourier transform is a contractive linear self mapping on $\mathscr{S}_q(\R^d)$ implying with the above the continuity of
\[
	\R^d\mapsto(\mathscr{S}_q(\R^d),|\cdot|_q),
	\quad
	y\mapsto\mathscr{F}\left(\frac{\e^{i\xi \cdot y}}{(|\xi|^2+m^2)^{\alpha}}\right)
	=
	\tau_y\left(\left(\mathscr{F}\left(\frac{1}{(|\xi|^2+m^2)^{\alpha}}\right)\right)^\vee\right)
\]
Hence, i) implies ii). Moreover, obviously, iii) follows from ii).
	
Next, we assume that iii) is valid. 
By the arguments elaborated at the beginning of Appendix \ref{appendix b} (see (\ref{preparation convolution 1})) we assume that
\[
	h:\R^d\rightarrow\C,
	\quad
	\xi\mapsto\left(|\xi|^2-\Delta\right)^q\left(\frac{\e^{i\xi \cdot y}}{(|\xi|^2+m^2)^\alpha}\right)
\]
	belongs to $L^2(\R^d)$ for every $y\in\R^d$. For the particular case $y=0$, by Proposition \ref{preparation convolution 3} ii), the above $h$ becomes
	$$h(\xi)=\frac{Q_q(|\xi|^2)}{(|\xi|^2+m^2)^{\alpha+2q}}$$
	with a suitable polynomial $Q_q$ on $\R$ of degree $3q$. Thus,  there are $C>0$ and $R>m^2$ with $|Q_q(|\xi|^2)|\geq C|\xi|^{6q}, |\xi|\geq R$,
	so that
	$$\infty>\int_{\R^d}\frac{|Q_q(|\xi|^2)|^2}{(|\xi|^2+m^2)^{2\alpha+4q}}\,d\xi\geq\frac{C\sigma(S^{d-1})}{4^{\alpha+2q}}\int_R^\infty r^{4q+d-1-4\alpha}\,dr$$
	where $\sigma(S^{d-1})$ denotes the surface of the unit sphere of $\R^d$. Thus, $4q+d-1-4\alpha<-1$, which shows that iii) implies iv). The proof of Lemma \ref{lem:matern} is complete.$\hfill\square$

\section{Proof of Proposition~\ref{prop:moments} and Some Consequences} \label{sec:moment_gauss}

We first prove Proposition~\ref{prop:moments}. 
In order to do so, for $f_j\in\mathscr{S}$, we consider the joint characteristic function of the random vector $Z^n := (Z(f_1),\ldots, Z(f_n))^T$, which by linearity of $Z$ is given by
\begin{align*}
	\varphi_{Z^n}(t_1,\ldots, t_n) 
	&= 
	\ev{\e^{i\sum_{j=1}^n t_jZ(f_j)}}
	=
	\ev{\e^{iZ\left(\sum_{j=1}^n t_j f_j\right)}}
	=
	\varphi_Z\left(\sum_{j=1}^n t_j f_j\right)\\
	&=
	\exp\left(\int_{\R^d}\left(\psi\circ \left(\sum_{j=1}^n t_j f_j\right)\right)(x)\,\d x\right),
\end{align*}
where
\[
	\psi(t) 
	= 
	i b t - \frac{\sigma^2 t^2}{2} + \int_{\R\setminus\{0\}}\e^{its} - 1 - its\mathds{1}_{\{|s|\leq 1\}}(s)\nu(\d s).
\]
Since the mixed moments $\ev{\prod^n_{j=1} Z(f_j)}$ are related to the joint characteristic function $\varphi_{Z^n}$ via
\begin{equation*}
	  \ev{\prod^n_{j=1} Z(f_j)} 
	  = 
	  \left.\frac{1}{i^n} \frac{\partial^n}{\partial t_1 \ldots \partial t_n} \varphi_{Z^n}(t)\right|_{t=0},
\end{equation*}
the former can be calculated by an application of Faà di Bruno's formula to $f(x) = \e^x$ and 
\begin{equation*}
	g(t) = g(t_1,\ldots,t_n) = \int_{\R^d}\left( \psi \circ \left( \sum^n_{j=1}t_j f_j \right) \right)(x)\,\d x:
\end{equation*}
\begin{theorem}[Faà di Bruno]
Let $g : \R^n \to \C$ have partial derivatives up to order $n$ and $f : \C \to \C$ be $n$ times differentiable in an open neighborhood of $g(\R^n)$. 
Then there holds
\begin{equation*}
	\frac{\partial^n}{\partial t_1\ldots \partial t_n} f\circ g 
	= 
	\sum^n_{j=1} f^{(j)}\circ g  \sum_{\substack{I\in \mathscr{P}^{(n)}_j\\ 
	I =\{I_1,\ldots, I_j\}}} \prod_{l=1}^j \partial_{I_l} g,
\end{equation*}
where $\mathscr{P}^{(n)}_k$ is the colletion of partitions of $\{1,\ldots,n\}$ into exactly $k$ disjoint subsets, $\partial_{A} = \left[ \prod_{j\in A}  \frac{\partial}{\partial t_j}\right]$ and $f^{(k)}(x)= \left( \frac{\d^k}{\d x^k}f \right)(x)$.
	\label{thm:FaaDiBruno}
\end{theorem}

Taking into account $g(0)=0$ it thus only remains to calculate the partial derivatives of
\begin{equation*}
	  g(t) 
	  = 
	  \int_{\R^d} ibk(x) - \frac{\sigma^2 k(x)^2}{2} + \int_{\R\setminus\{0\}}\e^{ik(x)s} - 1 - ik(x)s\mathds{1}_{\{|s|\leq 1\}}(s) \,\nu(\d s) \,\d x,
\end{equation*}
where we have used the abbreviation $k:=\sum_{j=1}^n t_j f_j$. It follows from the general assumption on $\nu$, i.e.\  $\int_{\R\setminus\{0\}}\min\{1,s^2\}\nu(\d s) < \infty$, that the majorant $2\min\{|s|,1\}|k(x)|$ of the inner intergrand is integrable. Thus, a standard consequence of Lebesgue's dominated convergence theorem together with an application of Fubini's theorem yields for $1\leq m\leq n$
	\begin{equation*}
	  \frac{1}{i} \frac{\partial}{\partial t_m}g(0) = \int_{\R^d}b f_m(x) + \int_{|s|>1} f_m(x)s \nu(\d s)\,\d x = \left(b+\int_{|s|>1}s\nu(\d s)\right) \int_{\R^d}f_m(x)\,\d x.
	\end{equation*}
	Likewise, the general assumption $\int_{\R\setminus\{0\}}\min\{1,s^2\}\nu(\d s) < \infty$ yields for $1\leq m_1,m_2\leq n$
	\begin{equation*}
	  \frac{1}{i^2}\frac{\partial^2}{\partial t_{m_1}\partial t_{m_2}}g(0) = \left( \sigma^2 +\int_{\R\setminus\{0\}}s^2\nu(\d s) \right) \int_{\R^d} f_{m_1}(x)f_{m_2}(x) \,\d x.
	\end{equation*}
	Inductively, since by hypothesis $\int_{\R\setminus\{0\}}s^l \nu(\d s)< \infty$ for $l\geq 3$, we derive
	\begin{equation*}
	  \frac{1}{i^l}\frac{\partial^l}{\partial t_{m_1}\ldots \partial t_{m_l}} g(0) = \int_{\R\setminus\{0\}}s^l\nu(\d s) \int_{\R^d} \prod^l_{j=1} f_{m_j}(x) \,\d x
	\end{equation*}
	Defining $b_1 = \int_{|s|>1}s\nu(\d s)$ and $b_n = \int_{\R\setminus\{0\}}s^n\nu(\d s)$ for $n\geq 2$, as well as
	\begin{equation*}
	  c_n =
	 \begin{cases}
	   b+b_1 &: n=1,\\
	   \sigma^2 + b_2 &: n=2,\\
	   b_n &: n\geq 3,
	 \end{cases}
	\end{equation*}
	we finally obtain
	\begin{align*}
	  \ev{\prod^n_{j=1}Z(f_j)} &= \frac{1}{i^n} \sum^n_{j=1}\sum_{\substack{I\in \mathscr{P}^{(n)}_j\\I = \{I_1,\ldots,I_j\}}} \prod_{l=1}^j \partial_{I_l} g(0)=  \sum_{\substack{I\in \mathscr{P}^{(n)}\\ I = \{I_1, \ldots, I_j\}}} \prod_{l=1}^j c_{|I_l|}\int_{\R^d} \prod_{m\in I_l} f_m \,\d x.
	\end{align*}
	Note that the total order of differentiation per summand equals $n$, thus the factor $\frac{1}{i^n}$ can be split up among all factors as above. This completes the proof of Proposition \ref{prop:moments}.\hfill$\square$\\
	
	As an application of Proposition \ref{prop:moments} we calculate the covariance of $Z(f_1)$ and $Z(f_2)$ for a L\'evy noise field $Z$ on $\pspace$. First of all, we note that the above proof shows $Z(f)\in L^2\pspace$, $f\in\mathscr{S}$. Since $Z$ is $\tripleBar{\cdot}$-continuous and $\mathscr{S}$ is $\tripleBar{\cdot}$-dense in $L^1\cap L^2(\R^d)$ we conclude $Z(f)\in L^2\pspace$ for $f\in L^1\cap L^2(\R^d)$. Thus, the covariance of $Z(f_1)$ and $Z(f_2)$ is well-defined and with $c_1$ and $c_2$ as above we have 
	\begin{align*}
  		\ev{Z(f_1)} &= c_1\int_{\R^d}f_1(y)\,\d y,\\
  		\ev{Z(f_1)Z(f_2)} &= c_2\int_{\R^d}f_1(y)f_2(y)\,\d y + c_1^2\int_{\R^d}f_1(y)\,\d y \; \int_{\R^d}f_2(y)\,\d y
	\end{align*}
	as well as
	\begin{align*}
  		&\Cov(Z(f_1),Z(f_2)) = \ev{(Z(f_1)Z(f_2)}-\ev{Z(f_1)}\ev{Z(f_2)})\\
  		&= c_2\int_{\R^d}f_1(y)f_2(y)\,\d y + c_1^2\int_{\R^d}f_1(y)\,\d y \int_{\R^d}f_2(y)\,\d y-c_1^2\int_{\R^d}f_1(y)\,\d y \int_{\R^d}f_2(y)\,\d y\\
  		&= \left(\sigma^2 + \int_{\R\setminus\{0\}}s^2\,\nu(\d s)\right)\int_{\R^d}f_1(y)f_2(y)\,\d y.
	\end{align*}
	The special case of $f_1=k_{x_1}=k(x_1-\cdot), f_2=k_{x_2}=k(x_2-\cdot), k\in L^1(\R^d)\cap L^2(\R^d), x_1,x_2\in\R^d$ gives the two-point covariance function of the smoothed L\'evy noise field $Z_k$
	\begin{align*}
  		C(x_1,x_2) &:= \Cov(Z_k(x_1), Z_k(x_2)) = \Cov(Z(k_{x_1}), Z(k_{x_2}))\\
  		&= \left(\sigma^2 + \int_{\R\setminus\{0\}}s^2\,\nu(\d s)\right) \int_{\R^d} k(x_1-y)k(x_2-y)\,\d y.
	\end{align*}
	In case of Matérn kernels $k_{\alpha,m}$, we have 
	\begin{equation*}
  		\int_{\R^d} k_{\alpha,m}(x_1-y)k_{\alpha,m}(x_2-y)\,\d y = (k^\vee * k)(x_1-x_2) = k_{2\alpha,m}(x_1-x_2).
	\end{equation*}
	Comparing the covariance function of a smoothed Lévy field to the covariance of a smoothed pure Gaussian random field, i.e.\ $\nu=0$, we have
	\begin{equation*}
  		C_{\text{Lévy}} = \frac{\sigma^2 + \int_{\R\setminus\{0\}}s^2\,\nu(\d s)}{\sigma^2}C_{\text{Gauss}}.
	\end{equation*}
	In particular, the eigenfunctions of the integral operators associated with $C_{\text{Lévy}}$ and $C_{\text{Gauss}}$ coincide while the corresponding eigenvalues are multiples of one another.

\section{Bounds for Eigenvalues and Eigenfunctions of Mat\'ern Integral Operators}\label{app:conv_rates}

The purpose of this appendix is to establish bounds on the eigenvalues and eigenfunctions of the compact operator $K: L^2(\Lambda) \to L^2(\Lambda)$, $\Lambda \subset \R^d$ compact with $\Lambda = \overline{\interior(\Lambda)}$, where
\begin{equation*}
  [Kf](x) = \int_\Lambda k_{\alpha,m}(x-y)f(y)\,\d y,
  \qquad
  f\in L^2(\Lambda), \; x\in \R^d,
\end{equation*}
and where $k_{\alpha,m}$ is the Mat\'ern kernel with parameters $\alpha> \frac{d}{2}$, $m>0$. 
Our approach is very much inspired by \cite{BachmayrEtAl2018}.
However, beside giving explicitly the dependence of the constants on the  domain $\Lambda$ we simplify some of the arguments for our particular setting.

Since $k_{\alpha,m}$ is real-valued and has a positive Fourier transform, $K$ is a positive, self-adjoint operator.
We denote by $(e_{\Lambda,j})_{j\in \N}=:(e_j)_{j\in\N}$ and $(\lambda_{\Lambda,j})_{j\in \N}=:(\lambda_j)_{j\in\N}$ the orthonormal basis of $L^2(\Lambda)$ consisting of eigenfunctions of $K$ and the corresponding sequence of eigenvalues. 
Since $K$ is positive, $\lambda_j \geq 0$, $j\in \N$. As usual we assume that the eigenvalue sequence is decreasing.

Extending every $f\in L^2(\Lambda)$ by zero to $\R^d$ and denoting this extension of $f$ to $\R^d$ again by $f$, we interpret $L^2(\Lambda)$ as a closed subspace of $L^2(\R^d)$.
Since $\Lambda$ is compact we likewise have $L^2(\Lambda)\subset L^1(\R^d)$, thus
\begin{equation*}
   k_{\alpha,m} * f\in L^1(\R^d)
   \quad\forall f\in L^2(\Lambda)
\end{equation*}
and $Kf = (k_{\alpha,m}*f)|_{\Lambda}$.
Clearly, $K$ is the compression to $L^2(\Lambda)$ of the convolution operator on $L^2(\R^d)$ with convolutor $k_{\alpha,m}$.

Recall that for $s\in \R$ we have the Sobolev space
\begin{equation*}
  H^s(\R^d) = \{u\in \mathscr{S}'(\R^d); (1+|\xi|^2)^{\frac{s}{2}}\hat{u}\in L^2(\R^d)\}
\end{equation*}
with norm
\begin{equation*}
  \|u\|_{H^s(\R^d)} = (2\pi)^{-d}\|(1+|\xi|^2)^{\frac{s}{2}}\hat{u}\|_{L^2(\R^d)}.
\end{equation*}
Moreover, for $G \subset \R^d$ open, we have
\begin{equation*}
  H^s(G)=\{u\in \mathscr{D}'(G); \exists U\in H^s(\R^d) : U|_G = u\}
\end{equation*}
with norm
\begin{equation*}
  \|u\|_{H^s(\R^d)} = \min_{U\in H^s(\R^d), U|_G = u} \|U\|_{H^s(\R^d)}.
\end{equation*}
It follows immediately that $k_{\alpha,m} * f \in H^\alpha(\R^d), f\in L^2(\R^d),$ and since $\alpha > \frac{d}{2}$ we have $\widehat{k_{\alpha,m}*f}\in L^1(\R^d), f\in L^2(\R^d)$. Because
\begin{equation*}
  \forall j\in \N: \lambda_j e_j = (k_{\alpha,m} * e_j)\mathds{1}_\Lambda
\end{equation*}
it follows $\lambda_j \neq 0$ for $j\in \N$ as well as $e_j\in H^\alpha(\interior(\Lambda))\cap C(\Lambda)$.

Let $\frac{d}{2} < s < \alpha$. By the Fourier inversion formula and the fact that due to $s>\frac{d}{2}$ the Fourier transform of every $H^s(\R^d)$ function belongs to $L^1(\R^d)$ (see e.g. \cite[Corollary~7.9.4]{hoermander}) for every $U\in H^s(\R^d)$ with $U|_{\interior(\Lambda)} = e_j$
\begin{equation*}
  \|e_j\|_{L^\infty(\Lambda)} \leq (2\pi)^{-d}\|\hat{U}\|_{L^1(\R^d)} \leq \|(1+|\xi|^2)^{-s}\|_{L^1(\R^d)}\|U\|_{H^s(\R^d)}
\end{equation*}
so that with $c_s:=\|(1+|\xi|^2)^{-s}\|_{L^1(\R^d)}$ we have
\begin{equation*}
  \forall j\in \N: \|e_j\|_{L^\infty(\Lambda)} \leq c_s\|e_j\|_{H^s(\interior(\Lambda))}.
\end{equation*}
Applying an interpolation inequality (see e.g. \cite[Theorem~B.8 and Lemma~B.1]{McLean2000}) gives
\begin{equation}
  \begin{split}
    \forall j\in \N: \|e_j\|_{L^\infty(\Lambda)}&\leq c_s\alpha\left( \frac{\sin(\frac{s\pi}{\alpha})}{\pi s(\alpha-s)}\right)^{\frac{1}{2}} \|e_j\|^{1-\frac{s}{\alpha}}_{L^2(\Lambda)}\|e_j\|^{\frac{s}{\alpha}}_{H^\alpha(\interior(\Lambda))}\\
    &= c_s\alpha\left( \frac{\sin(\frac{s\pi}{\alpha})}{\pi s(\alpha-s)}\right)^{\frac{1}{2}} \|e_j\|^{\frac{s}{\alpha}}_{H^\alpha(\interior(\Lambda))}.
  \end{split}
    \label{eq:ejLinfty}
\end{equation}
From the definition of $\|\cdot\|_{H^\alpha(\interior(\Lambda))}$ and Plancherel's Theorem we conclude
\begin{align*}
  \lambda^2_j\|e_j\|^2_{H^\alpha(\interior(\Lambda))}&\leq \|k_{\alpha,m}*e_j\|^2_{H^\alpha(\R^d)} = (2\pi)^{-2d}\int_{\R^d} (1+|\xi|^2)^\alpha (m^2+|\xi|^2)^{-2\alpha} |\hat{e_j}|^2\,\d\xi\\
  &\leq \max\{1,m^{-2\alpha}\} (2\pi)^{-2d}\int_{\R^d} \widehat{k_{\alpha,m}*e_j}\, \overline{\hat{e_j}}\,\d x \\
  &= \max\{1,m^{-2\alpha}\} (2\pi)^{-2d} (Ke_j,e_j)_{L^2(\Lambda)} = \max\{1,m^{-2\alpha}\}(2\pi)^{-2d}\lambda_j.
\end{align*}
Combining the previous inequality with \eqref{eq:ejLinfty} we obtain for every $s\in (\frac{d}{2},\alpha)$:
\begin{equation}
  \forall j\in \N: \sqrt{\lambda_j}\|e_j\|_{L^\infty(\Lambda)} \leq c_s\alpha \left( \frac{\sin(\frac{s\pi}{\alpha})}{\pi s(\alpha-s)} \right)^{\frac{1}{2}}\max\{1,m^{-s}\} (2\pi)^{-\frac{ds}{\alpha}} \lambda_j^{\frac{1}{2}- \frac{s}{2\alpha}}.
  \label{eq:AnsatzFunctionLInfty}
\end{equation}

Next we derive estimates for the eigenvalue sequence $(\lambda_j)_{j\in \N}$. Let $\delta \geq  \text{diam}(\Lambda)$. Without loss of generality we assume that $\Lambda \subseteq [-\frac{\delta}{2}, \frac{\delta}{2}]^d$.
Thus our integral operator $K$ is determined by $k_{\alpha, m}|_{[-\delta,\delta]^d}$.
For arbitrary $\gamma > \delta$ we can interpret $L^2(\Lambda)$ as a subspace of $L^2([-\gamma,\gamma]^d)$ by extending functions by zero.
Again we do not distinguish notationally functions from $L^2(\Lambda)$ and their extensions.

If $k$ is any continuous, real-valued and even extension of $k_{\alpha,m}|_{[-\delta,\delta]^d}$ to $[-\gamma,\gamma]^d$ --- note that $k_{\alpha,m}$ is a radial function, so in particular even --- it follows that
\begin{equation*}
  \tilde K: L^2([-\gamma,\gamma]^d) \to L^2([-\gamma,\gamma]^d), 
  \quad
  (\tilde K f)(x) := \int_{[-\gamma,\gamma]^d}k(x-y)f(y)\,\d y
\end{equation*}
is a self-adjoint, compact operator which satisfies $ Kf = \tilde K f|_{\Lambda}, f\in L^2(\Lambda)$. Since $\tilde K$ is self-adjoint and compact, there exists an orthonormal basis $(f_j)_{j\in \N}$ of $L^2([-\gamma,\gamma]^d)$ consisting of eigenfunctions of $\tilde K$ and a real sequence of corresponding eigenvalues $(\tilde{\lambda}_j)_{j\in\N}$ which without loss of generality have decreasing moduli.

Clearly, for every $j\in \N$ the operators
\begin{align*}
  B_j: L^2(\Lambda)\to L^2(\Lambda), & \,f\mapsto \sum^j_{l=1}\lambda_l (f,e_l)e_l,\\
  C_j: L^2(\Lambda)\to L^2(\Lambda), & \,f\mapsto \sum^j_{l=1}\tilde{\lambda}_l (f,f_l|_\Lambda) f_l|_\Lambda
\end{align*}
and
\begin{equation*}
  \tilde C_j : L^2([-\gamma,\gamma]^d)\to L^2([-\gamma, \gamma]^d), f\mapsto \sum^j_{l=1} \tilde{\lambda}_l (f,f_l)f_l
\end{equation*}
are continuous linear operators with at most $j$-dimensional range, where $(\cdot,\cdot)$ denotes the inner product in $L^2(\Lambda)$ and $L^2([-\gamma,\gamma]^d)$, respectively.

Then $\tilde C_jf|_\Lambda = C_jf$ for $f\in L^2(\Lambda)$, $j\in \N_0$, and denoting temporarily the norms of $L^2(\Lambda)$ and $L^2([-\gamma,\gamma]^d)$ and the corresponding operator norms by $\|\cdot\|_\Lambda$ and $\|\cdot\|_\gamma$, respectively, we have
\begin{equation*}
  \forall f\in L^2(\Lambda) : \|Kf- C_jf\|_\Lambda = \|(\tilde K f-\tilde C_j f)|_\Lambda\|_\Lambda \leq \|\tilde K f - \tilde C_j f\|_\gamma
\end{equation*}
so that
\begin{equation}
  \forall j\in \N: \|K-C_j\|_\Lambda \leq \|\tilde K - \tilde C_j\|_\gamma.
  \label{ineq:LambdaGammaNorms}
\end{equation}
Since $K$ is a positive self-adjoint compact operator and $\tilde K$ is a self-adjoint compact operator, by a well known result (see e.g. \cite[Lemma~16.5 and its proof]{MeVo}) we obtain with \eqref{ineq:LambdaGammaNorms}:
\begin{equation}
  \forall j\in \N: \lambda_j=\|K-B_{j-1}\|_\Lambda \leq \|K-C_{j-1}\|_\Lambda\leq \|\tilde K - \tilde C_{j-1}\|_\gamma = |\tilde{\lambda}_j|.
  \label{eq:EigenvalueNorms}
\end{equation}

Up to now we did not specify the extension $k$ of $k_{\alpha,m}|_{[-\delta,\delta]^d}$. For this we fix $\chi>\max\{1,\frac{1}{\delta}\}$ and a real valued, even $\phi_{1,\chi} \in \mathscr{D}(\R^d)$ with $\phi_{1,\chi}|_{[-\frac{1}{\chi},\frac{1}{\chi}]^d}=1$ and $\supp \phi_{1,\chi} \subseteq [-1,1]^d$. For $\gamma \geq \chi\delta$ we define $\phi_{\gamma,\chi}(x):=\phi_{1,\chi}(\frac{1}{\gamma}x)$ so that $\phi_{\gamma,\chi}|_{[-\delta,\delta]^d}=1$ and $\supp \phi_{\gamma,\chi} \subseteq  [-\gamma,\gamma]^d$. In order to simplify the notation we write $\phi_1$ and $\phi_\gamma$ instead of $\phi_{1,\chi}$ and $\phi_{\gamma,\chi}$, respectively.

Then $k_\gamma:=k_{\alpha,m}\phi_\gamma$ is an even extension of $k_{\alpha,m}|_{[-\delta,\delta]^d}$ with $\supp k_\gamma \subseteq [-\gamma,\gamma]^d$, $\gamma\geq\chi\delta$. We define the $\gamma$-periodic extension $k_p$ of $k_\gamma$ by
\begin{equation*}
  \forall x\in \R^d : k_p(x):= \sum_{n\in \Z^d}k_\gamma(x+2\gamma n).
\end{equation*}
Then for the integral operator $\tilde K$ corresponding to $k_\gamma$ we have
\begin{align*}
  \forall x\in \R^d, n\in \Z^d :&\, [\tilde K e^{-i\frac{\pi}{\gamma}n\cdot y}](x) = \int_{[-\gamma,\gamma]^d}k_\gamma(x-y)e^{-i\frac{\pi}{\gamma}n\cdot y}\,\d y\\
  &= \int_{[-\gamma,\gamma]^d}k_p(x-y)e^{-i \frac{\pi}{\gamma}n\cdot y}\,\d y =   \int_{[-\gamma,\gamma]^d}k_p(z)e^{-i \frac{\pi}{\gamma}n\cdot z}\,\d z\, e^{-i \frac{\pi}{\gamma}n\cdot x}.
\end{align*}
Since $\{e^{-i \frac{\pi}{\gamma}n\cdot x}, n\in \Z^d\}$ is an orthogonal basis of $L^2([-\gamma,\gamma]^d)$ it follows that the Fourier coefficients $c_n(k_p)$ of $k_p$
\begin{equation*}
  \forall n\in \Z^d:c_n(k_p):=\int_{[-\gamma,\gamma]^d}k_p(z)e^{-i \frac{\pi}{\gamma}n\cdot z}\,\d z = \int_{[-\gamma,\gamma]^d}k_\gamma(z)e^{-i \frac{\pi}{\gamma}n\cdot z}\,\d z
\end{equation*}
(which are real since $k_\gamma$ is even) are the eigenvalues of $\tilde K$, i.e., a suitable enumeration of $(c_n(k_p))_{n\in \Z^d}$ yields the eigenvalue sequence $(\tilde{\lambda}_j)_{j\in \N}$.

Because $\supp k_\gamma\subseteq [-\gamma,\gamma]^d$ we have
\begin{equation}
  \begin{split}
  \forall n\in \Z^d :&\, |c_n(k_p)| = \left|\int_{[-\gamma,\gamma]^d} k_\gamma(z)e^{-i \frac{\pi}{\gamma}n\cdot z}\,\d z\right| = \left|\int_{\R^d}k_\gamma(z)e^{-i \frac{\pi}{\gamma}n\cdot z}\,\d z\right|\\
  &= \left|\widehat{k_{\alpha,m}\cdot \phi_\gamma}(-\frac{\pi}{\gamma}n)\right| = \left|\hat{k}_{\alpha,m}*\hat{\phi}_\gamma(-\frac{\pi}{\gamma}n)\right|.
\end{split}
\label{eq:FourierCoeff}
\end{equation}
Moreover, for $\xi\in \R^d$ it holds
\begin{equation}
  \begin{split}
    \left|\hat{k}_{\alpha,m}*\hat{\phi}_\gamma(\xi)\right|&\leq \left|\int_{|\eta|\leq \frac{|\xi|}{2}}\hat{k}_{\alpha,m}(\eta)\hat{\phi}_\gamma(\xi-\eta) \,\d\eta\right| +\left|\int_{|\eta|\geq \frac{|\xi|}{2}}\hat{k}_{\alpha,m}(\eta)\hat{\phi}_\gamma(\xi-\eta)\,\d\eta\right| \\
    &\leq\max_{|\zeta|\geq \frac{|\xi|}{2}}\left|\hat{\phi}_\gamma(\zeta)\right| \|\hat{k}_{\alpha,m}\|_{L^1(\R^d)} + \max_{|\zeta|\geq \frac{|\xi|}{2}}\left|\hat{k}_{\alpha,m}(\zeta)\right| \|\hat{\phi}_\gamma\|_{L^1(\R^d)}.
  \end{split}
  \label{eq:ConvolutionByNorms}
\end{equation}
Because $\hat{\phi}_\gamma(\xi) = \gamma^d \hat{\phi}_1\left(\gamma \xi \right)$ it follows
\begin{equation}
  \|\hat{\phi}_\gamma\|_{L^1(\R^d)} = \|\hat{\phi}_1\|_{L^1(\R^d)}.
  \label{eq:normeqphi}
\end{equation}
Moreover, due to $\gamma\geq \chi\delta\geq 1$ and $\alpha>\frac{d}{2}$ we have
\begin{align*}
  |\hat{\phi}_\gamma(\xi)| &=\gamma^d \left( 1+\left|\gamma \xi\right|^2 \right)^{-\ceil{\alpha}} \left( 1+\left|\gamma \xi\right|^2 \right)^{\ceil{\alpha}}\left|\hat{\phi}_1\left( \gamma \xi \right)\right| \\
  &= \gamma^d \left( 1+\left|\gamma\xi\right|^2 \right)^{-\ceil{\alpha}} \left|\widehat{(1-\Delta)^{\ceil{\alpha}}\phi_1}\left(\gamma \xi \right)\right| \\
  &\leq\gamma^{2\alpha} \left( 1+\left|\xi\right|^2 \right)^{-\alpha} \left\|(1-\Delta)^{\ceil{\alpha}}\phi_1\right\|_{L^1(\R^d)}.
\end{align*}
Inserting this and \eqref{eq:normeqphi} into \eqref{eq:ConvolutionByNorms} gives for $\xi\in \R^d$:
\begin{align*}
  |\hat{k}_{\alpha,m}*\hat{\phi}_\gamma(\xi)| \leq & \gamma^{2\alpha} \left( 1+\frac{|\xi|^2}{4} \right)^{-\alpha} \left\|(1-\Delta)^{\ceil{\alpha}}\phi_1\right\|_{L^1(\R^d)} \|\hat{k}_{\alpha,m}\|_{L^1(\R^d)}\\
  &  + \left(m^2 + \frac{|\xi|^2}{4}\right)^{-\alpha}\|\hat{\phi}_1\|_{L^1(\R^d)} \\
  \leq& \gamma^{2\alpha} \max\{1,m^{-2\alpha}\}\left( 1+\frac{|\xi|^2}{4} \right)^{-\alpha} \left[\left\|(1-\Delta)^{\ceil{\alpha}}\phi_1\right\|_{L^1(\R^d)} \right.\\
  &\left.\cdot \|\widehat{k}_{\alpha,m}\|_{L^1(\R^d)} +\|\widehat{\phi_1}\|_{L^1(\R^d)} \right].
\end{align*}
Thus, with \eqref{eq:FourierCoeff} we get for every $n\in\Z^d$
\begin{align*}
    |c_n(k_p)| \leq& \max\{1,m^{-2\alpha}\}\left[\|(1-\Delta)^{\ceil{\alpha}}\phi_1\|_{L^1(\R^d)}\|\hat{k}_{\alpha,m}\|_{L^1(\R^d)}+ \|\hat{\phi}_1\|_{L^1(\R^d)}\right]\cdot\\
    &\cdot \gamma^{2\alpha}\max\left\{1,\frac{2\gamma}{\pi}\right\}^{2\alpha}\left(1+|n|^2\right)^{-\alpha}
\end{align*}
For $\gamma\geq \chi\delta$ we conclude for every $n\in \Z^d$ and $0<\eta\leq |c_n(k_p)|$ that
\begin{align*}
    |n| &< \max\{1,m^{-1}\}\left[\|(1-\Delta)^{\ceil{\alpha}}\phi_1\|_{L^1(\R^d)}\|\hat{k}_{\alpha,m}\|_{L^1(\R^d)}+ \|\hat{\phi}_1\|_{L^1(\R^d)}\right]^{\frac{1}{2\alpha}}\cdot\\
  &\cdot \max\left\{1,\frac{2\gamma}{\pi}\right\}\gamma \eta^{-\frac{1}{2\alpha}}.
\end{align*}
%
%
Hence, for $\gamma\geq\chi\delta$ and $\eta>0$ we obtain
\begin{align*}
    &\#\{n\in \Z^d; |c_n(k_p)| \geq \eta\} \leq 2^d\max\{1,m^{-d}\} \cdot\\
    &\cdot \left[\|(1-\Delta)^{\ceil{\alpha}}\phi_1\|_{L^1(\R^d)}\|\hat{k}_{\alpha,m}\|_{L^1(\R^d)}+ \|\hat{\phi}_1\|_{L^1(\R^d)}\right]^{\frac{d}{2\alpha}}
    \left(\max\left\{1,\frac{2\gamma}{\pi}\right\}\right)^d
    \gamma^d
    \eta^{-\frac{d}{2\alpha}}
\end{align*}
For the eigenvalue sequence $(\tilde{\lambda}_j)_{j\in \N}$ of the operator $\tilde K$ associated with $k_\gamma$, $\gamma \geq \chi\delta$, it thus follows for all $j\in\N$
\begin{align*}
  |\tilde{\lambda}_j|&\leq 4^\alpha \max\{1,m^{-2\alpha}\}
  \left[\|(1-\Delta)^{\ceil{\alpha}}\phi_1\|_{L^1(\R^d)}\|\hat{k}_{\alpha,m}\|_{L^1(\R^d)}+ \|\hat{\phi}_1\|_{L^1(\R^d)}\right]\cdot\\
    &\cdot 
    \left(\max\left\{1,\frac{2\gamma}{\pi}\right\}\right)^{2\alpha} \gamma^{2\alpha} j^{-\frac{2\alpha}{d}}.
\end{align*}

Hence, taking \eqref{eq:EigenvalueNorms} and \eqref{eq:AnsatzFunctionLInfty} into account we finally obtain for every $s\in (\frac{d}{2},\alpha)$, $\gamma \geq \chi\delta$ and each $j\in\N$
\begin{align*}
  &\sqrt{\lambda_j}\|e_j\|_{L^\infty(\Lambda)} \leq c_s\alpha\left( \frac{\sin(\frac{s\pi}{\alpha})}{\pi s(\alpha-s)} \right)^{\frac{1}{2}}\max\{1,m^{-s}\} (2\pi)^{-\frac{ds}{\alpha}}
  2^{\alpha-s} \max\{1,m^{-(\alpha-s)}\}\cdot \\
  &\cdot \left[\|(1-\Delta)^{\ceil{\alpha}}\phi_1\|_{L^1(\R^d)}\|\hat{k}_{\alpha,m}\|_{L^1(\R^d)}+ \|\hat{\phi}_1\|_{L^1(\R^d)}\right]^{\frac{\alpha-s}{2\alpha}} \gamma^{2(\alpha-s)} j^{-\frac{\alpha}{d}+ \frac{s}{d}}.
\end{align*}
It follows that for every $\delta>0$ and each $\epsilon \in (0,\frac{\alpha}{d}-\frac{1}{2})$ (with $s=\frac{d}{2} + \epsilon $ in the previous inequality) for every $\chi>\max\left\{1,\frac{1}{\delta}\right\}$ there is a constant $C>0$, depending only on $\alpha, m, \epsilon$, and $\chi$, such that for every compact subset $\Lambda\subseteq\R^d$ with $\overline{\mbox{int}(\Lambda)}=\Lambda$ and $\mbox{diam}(\Lambda)\leq \delta$ and every $\gamma\geq \chi\delta$, there holds
\begin{equation}\label{eq:bounds eigenvalues times eigenfunctions}
    \sqrt{\lambda_{\Lambda,j}}\|e_{\Lambda,j}\|_{L^\infty(\Lambda)} 
    \leq 
    C \gamma^{2(\alpha-\frac{d}{2} - \epsilon)} j^{-\frac{\alpha}{d}+\frac{1}{2}+\epsilon},
    \qquad
    j\in\N.
\end{equation}

\textbf{Acknowledgments:}
H.\ Gottschalk and M.\ Reese gratefully acknowledge partial financial support by the German Federal Ministry of Research and Education (BMBF, Grant-No.:  05M18PXA).

\bibliographystyle{abbrvnat}
\bibliography{references}

\begin{thebibliography}{63}
\providecommand{\natexlab}[1]{#1}
\providecommand{\url}[1]{\texttt{#1}}
\expandafter\ifx\csname urlstyle\endcsname\relax
  \providecommand{\doi}[1]{doi: #1}\else
  \providecommand{\doi}{doi: \begingroup \urlstyle{rm}\Url}\fi

\bibitem[Albeverio et~al.(1989)Albeverio, Holden, H{\o}egh-Krohn, and
  Kolsrud]{albeverio1989representation}
S.~Albeverio, H.~Holden, R.~H{\o}egh-Krohn, and T.~Kolsrud.
\newblock Representation and construction of multiplicative noise.
\newblock \emph{Journal of Functional Analysis}, 87\penalty0 (2):\penalty0
  250--272, 1989.
\newblock \doi{10.1016/0022-1236(89)90010-4}.

\bibitem[Albeverio et~al.(1996)Albeverio, Gottschalk, and Wu]{albeverio96}
S.~Albeverio, H.~Gottschalk, and J.-L. Wu.
\newblock Convoluted generalized white noise, {S}chwinger functions and their
  analytic continuation to {W}ightman functions.
\newblock \emph{Rev. Math. Phys.}, 8\penalty0 (6):\penalty0 763--817, 1996.
\newblock \doi{10.1142/S0129055X96000287}.

\bibitem[Albeverio et~al.(1998)Albeverio, Wu, and
  Zhang]{albeverio1998parabolic}
S.~Albeverio, J.~L. Wu, and T.-S. Zhang.
\newblock Parabolic {SPDE}s driven by {P}oisson white noise.
\newblock \emph{Stochastic Processes and their Applications}, 74\penalty0
  (1):\penalty0 21--36, 1998.
\newblock \doi{10.1016/S0304-4149(97)00112-9}.

\bibitem[Albeverio et~al.(2005)Albeverio, Gottschalk, and Minorou]{albeverio05}
S.~Albeverio, H.~Gottschalk, and Y.~Minorou.
\newblock Particle systems in the grand canonical ensemble, scaling limits and
  quantum field theory.
\newblock \emph{Rev. Math. Phys.}, 17\penalty0 (2):\penalty0 175--226, 2005.
\newblock \doi{10.1142/S0129055X05002327}.

\bibitem[Applebaum(2009)]{applebaum2009levy}
D.~Applebaum.
\newblock \emph{L{\'e}vy Processes and Stochastic Calculus}.
\newblock Cambridge University Press, 2009.
\newblock \doi{10.1017/CBO9780511809781}.

\bibitem[Babu\v{s}ka et~al.(2003)Babu\v{s}ka, Liu, and
  Tempone]{BabuskaEtAl2003}
I.~Babu\v{s}ka, K.-M. Liu, and R.~Tempone.
\newblock Solving stochastic partial differential equations based on the
  experimental data.
\newblock \emph{Mathematical Models and Methods in Applied Sciences},
  13\penalty0 (3):\penalty0 415--444, 2003.
\newblock \doi{10.1142/S021820250300257X}.

\bibitem[Babu\v{s}ka et~al.(2004)Babu\v{s}ka, Tempone, and
  Zouraris]{BabuskaEtAl2004}
I.~Babu\v{s}ka, R.~Tempone, and G.~Zouraris.
\newblock Galerkin finite element approximations of stochastic elliptic partial
  differential equations.
\newblock \emph{SIAM Journal on Numerical Analysis}, 42\penalty0 (2):\penalty0
  800--825, 2004.
\newblock \doi{10.1137/S0036142902418680}.

\bibitem[Babu\v{s}ka et~al.(2005)Babu\v{s}ka, Tempone, and
  Zouraris]{BabuskaEtAl2005}
I.~Babu\v{s}ka, R.~Tempone, and G.~Zouraris.
\newblock Solving elliptic boundary value problems with uncertain coefficients
  by the finite element method: the stochastic formulation.
\newblock \emph{Computer Methods in Applied Mechanics and Engineering},
  194:\penalty0 1251--1294, 2005.
\newblock \doi{10.1016/j.cma.2004.02.026}.

\bibitem[Bachmayr et~al.(2017{\natexlab{a}})Bachmayr, Cohen, DeVore, and
  Migliorati]{BachmayrEtAl2017b}
M.~Bachmayr, A.~Cohen, R.~DeVore, and G.~Migliorati.
\newblock Sparse polynomial approximation of parametric elliptic {PDE}s. {P}art
  {II}: Lognormal coefficients.
\newblock \emph{ESAIM: Mathematical Modelling and Numerical Analysis},
  51\penalty0 (1):\penalty0 341--363, 2017{\natexlab{a}}.
\newblock \doi{10.1051/m2an/2016051}.

\bibitem[Bachmayr et~al.(2017{\natexlab{b}})Bachmayr, Cohen, and
  Migliorati]{BachmayrEtAl2017a}
M.~Bachmayr, A.~Cohen, and G.~Migliorati.
\newblock Sparse polynomial approximation of parametric elliptic {PDE}s. part
  i: Affine coefficients.
\newblock \emph{ESAIM: Mathematical Modelling and Numerical Analysis},
  51\penalty0 (1):\penalty0 321--339, 2017{\natexlab{b}}.
\newblock \doi{10.1051/m2an/2016045}.

\bibitem[Bachmayr et~al.(2018)Bachmayr, Cohen, and
  Migliorati]{BachmayrEtAl2018}
M.~Bachmayr, A.~Cohen, and G.~Migliorati.
\newblock Representations of {G}aussian random fields and approximation of
  elliptic {PDE}s with lognormal coefficients.
\newblock \emph{Journal of Fourier Analysis and Applications}, 24:\penalty0
  621--649, 2018.
\newblock \doi{10.1007/s00041-017-9539-5}.

\bibitem[Barth and Stein(2018)]{BarthStein2018}
A.~Barth and A.~Stein.
\newblock A study of elliptic partial differential equations with jump
  diffusion coefficients.
\newblock \emph{SIAM/ASA J. Uncertainty Quantification}, 6\penalty0
  (4):\penalty0 1707--1743, 2018.
\newblock \doi{10.1137/17M1148888}.

\bibitem[Benth and Theting(2002)]{BenthTheting2002}
F.~Benth and T.~Theting.
\newblock Some regularity results for the stochastic pressure equation of
  {W}ick-type.
\newblock \emph{Stochastic Analysis and Applications}, 20\penalty0
  (6):\penalty0 1191--1223, 2002.
\newblock \doi{10.1081/SAP-120015830}.

\bibitem[Berg and Forst(1975)]{BergForst}
C.~Berg and G.~Forst.
\newblock \emph{Potential Theory on Locally Compact Abelian Groups}.
\newblock Springer-Verlag, New York-Heidelberg, 1975.
\newblock Ergebnisse der Mathematik und ihrer Grenzgebiete, Band 87.

\bibitem[Bittner(2019)]{bittner}
L.~Bittner.
\newblock \emph{On Shape Calculus with Elliptic {PDE} Constraints in Classical
  Function Spaces}.
\newblock PhD thesis, University of Wuppertal, 2019.

\bibitem[Charrier(2012)]{Charr12}
J.~Charrier.
\newblock Strong and weak error estimates for elliptic partial differential
  equations with random coefficients.
\newblock \emph{SIAM Journal on Numerical Analysis}, 50\penalty0 (1):\penalty0
  216--246, 2012.
\newblock \doi{10.1137/100800531}.

\bibitem[Chiu et~al.(2013)Chiu, Stoyan, Kendall, and Mecke]{chiu2013stochastic}
S.~N. Chiu, D.~Stoyan, W.~S. Kendall, and J.~Mecke.
\newblock \emph{Stochastic Geometry and its Applications}.
\newblock John Wiley \& Sons, 2013.
\newblock \doi{10.1002/9781118658222}.

\bibitem[Chkifa et~al.(2014)Chkifa, Cohen, and Schwab]{ChkifaEtAl2014}
A.~Chkifa, A.~Cohen, and C.~Schwab.
\newblock High-dimensional adaptive sparse polynomial interpolation and
  applications to parametric {PDE}s.
\newblock \emph{Foundations of Computational Mathematics}, 14:\penalty0
  601--633, 2014.
\newblock \doi{10.1007/s10208-013-9154-z}.

\bibitem[Chkifa et~al.(2015)Chkifa, Cohen, and Schwab]{ChkifaEtAl2015}
A.~Chkifa, A.~Cohen, and C.~Schwab.
\newblock Breaking the curse of dimensionality in sparse polynomial
  approximation of parametric {PDE}s.
\newblock \emph{Journal de Math{\'e}matiques Pures et Appliqu{\'e}es},
  103:\penalty0 400--428, 2015.
\newblock \doi{10.1016/j.matpur.2014.04.009}.

\bibitem[Cohen et~al.(2010)Cohen, DeVore, and Schwab]{CohenEtAl2010}
A.~Cohen, R.~DeVore, and C.~Schwab.
\newblock Convergence rates of best {$N$}-term {G}alerkin approximations for a
  class of elliptic s{PDE}s.
\newblock \emph{Foundations of Computational Mathematics}, 10:\penalty0
  615--646, 2010.
\newblock \doi{10.1007/s10208-010-9072-2}.

\bibitem[Cohen et~al.(2011)Cohen, DeVore, and Schwab]{CohenEtAl2011}
A.~Cohen, R.~DeVore, and C.~Schwab.
\newblock Analytic regularity and polynomial approximation of parametric and
  stochastic elliptic {PDE}'s.
\newblock \emph{Analysis and Applications}, 9\penalty0 (1):\penalty0 11--47,
  2011.
\newblock \doi{10.1142/S0219530511001728}.

\bibitem[Dalecky and Fomin(1991)]{dalecky91}
Y.~Dalecky and S.~Fomin.
\newblock \emph{Measures and Differential Equations in Infinite-Dimensional
  Space}.
\newblock Springer, 1991.

\bibitem[Deb et~al.(2001)Deb, Babu\v{s}ka, and Oden]{DebEtAl2001}
M.~Deb, I.~Babu\v{s}ka, and J.~Oden.
\newblock Solution of stochastic partial differential equations using
  {G}alerkin finite element techniques.
\newblock \emph{Computer Methods in Applied Mechanics and Engineering},
  190:\penalty0 6359--6372, 2001.
\newblock \doi{10.1016/S0045-7825(01)00237-7}.

\bibitem[Ernst et~al.(2018)Ernst, Sprungk, and Tamellini]{ErnstEtAl2018}
O.~Ernst, B.~Sprungk, and L.~Tamellini.
\newblock Convergence of sparse collocation for functions of countably many
  {G}aussian random variables (with application to elliptic {PDE}s).
\newblock \emph{SIAM Journal on Numerical Analysis}, 56\penalty0 (2):\penalty0
  877--905, 2018.
\newblock \doi{10.1137/17M1123079}.

\bibitem[Frauenfelder et~al.(2005)Frauenfelder, Schwab, and
  Todor]{FrauenfelderEtAl2005}
P.~Frauenfelder, C.~Schwab, and R.~Todor.
\newblock Finite elements for elliptic problems with stochastic coefficients.
\newblock \emph{Computer Methods in Applied Mechanics and Engineering},
  194:\penalty0 205--228, 2005.
\newblock \doi{10.1016/j.cma.2004.04.008}.

\bibitem[Galvis and Sarkis(2009)]{GalvisSarkis2009}
J.~Galvis and M.~Sarkis.
\newblock Approximating infinity-dimensional stochastic {D}arcy's equations
  without uniform ellipticity.
\newblock \emph{SIAM Journal on Numerical Analysis}, 47\penalty0 (5):\penalty0
  3624--3651, 2009.
\newblock \doi{10.1137/080717924}.

\bibitem[Galvis and Sarkis(2012)]{GalvisSarkis2012}
J.~Galvis and M.~Sarkis.
\newblock Regularity results for the ordinary product stochastic pressure
  equation.
\newblock \emph{SIAM J. Math. Anal.}, 44\penalty0 (4):\penalty0 2637--2665,
  2012.
\newblock \doi{10.1137/110826904}.

\bibitem[Gelfand and Vilenkin(1964)]{gelfand64}
I.~Gelfand and N.~Vilenkin.
\newblock \emph{Generalized Functions, IV. Some Applications of Harmonic
  Analysis}.
\newblock Academic Press, 1964.

\bibitem[Gin{\'e} and Nickl(2016)]{gine2016mathematical}
E.~Gin{\'e} and R.~Nickl.
\newblock \emph{Mathematical Foundations of Infinite-Dimensional Statistical
  Models}, volume~40.
\newblock Cambridge University Press, 2016.
\newblock \doi{10.1017/CBO9781107337862}.

\bibitem[Gittelson(2010)]{Gittelson2010}
C.~Gittelson.
\newblock Stochastic {G}alerkin discretization of the log-normal isotropic
  diffusion problem.
\newblock \emph{Mathematical Models and Methods in Applied Sciences},
  20\penalty0 (2):\penalty0 237--263, 2010.
\newblock \doi{10.1142/S0218202510004210}.

\bibitem[Gottschalk and Smii(2007)]{gottschalk2007determine}
H.~Gottschalk and B.~Smii.
\newblock How to determine the law of the solution to a {SPDE} driven by a
  {L}{\'e}vy space-time noise.
\newblock \emph{J. Math. Phys}, 43:\penalty0 1--22, 2007.
\newblock \doi{10.1063/1.2712916}.

\bibitem[Hackbusch(2017)]{Hackbusch2017}
W.~Hackbusch.
\newblock \emph{Elliptic Differential Equations: Theory and Numerical
  Treatment}, volume~18 of \emph{Springer Series in Computational Mathematics}.
\newblock Springer, 2nd edition, 2017.
\newblock \doi{10.1007/978-3-662-54961-2}.

\bibitem[Hoang and Schwab(2014)]{HoangSchwab2014}
V.~Hoang and C.~Schwab.
\newblock $n$-term {W}iener chaos approximation rates for elliptic {PDE}s with
  lognormal {G}aussian random inputs.
\newblock \emph{Mathematical Models and Methods in Applied Sciences},
  24\penalty0 (4):\penalty0 797--826, 2014.
\newblock \doi{10.1142/S0218202513500681}.

\bibitem[Holden et~al.(1996)Holden, Oksendal, Ub{\o}e, and
  Zhang]{HoldenEtAl1996}
H.~Holden, B.~Oksendal, J.~Ub{\o}e, and T.~Zhang.
\newblock \emph{Stochastic Partial Differential Equations: A Modeling, White
  Noise Functional Approach}.
\newblock Probability and its Applications. Birkh{\"a}user, Boston, 1996.
\newblock \doi{10.1007/978-1-4684-9215-6_4}.

\bibitem[H{\"o}rmander(2003)]{hoermander}
L.~H{\"o}rmander.
\newblock \emph{The Analysis of Linear Partial Differential Operators I}.
\newblock Springer Verlag, 2003.

\bibitem[It\^o(1984)]{ito84}
K.~It\^o.
\newblock \emph{Foundations of Stochastic Differential Equations in Infinite
  Dimensional Spaces}.
\newblock SIAM, 1984.

\bibitem[Kallenberg(2002)]{Kallenberg}
O.~Kallenberg.
\newblock \emph{Foundations of Modern Probability}.
\newblock Probability and its Applications (New York). Springer-Verlag, New
  York, second edition, 2002.
\newblock ISBN 0-387-95313-2.
\newblock \doi{10.1007/978-1-4757-4015-8}.
\newblock URL \url{https://doi.org/10.1007/978-1-4757-4015-8}.

\bibitem[Kunita(1990)]{kunita90}
H.~Kunita.
\newblock \emph{Stochastic Flows and Stochastic Differential Equations}.
\newblock Cambridge University Press, 1990.

\bibitem[L{\'e}vy(1954)]{levy1954theorie}
P.~L{\'e}vy.
\newblock \emph{Th{\'e}orie de l'addition des variables al{\'e}atoires},
  volume~1.
\newblock Gauthier-Villars, 1954.

\bibitem[Lindgren et~al.(2011)Lindgren, Rue, and Lindstr\"om]{Lindgren2011}
F.~Lindgren, H.~Rue, and J.~Lindstr\"om.
\newblock An explicit link between {G}aussian fields and {G}aussian {M}arkov
  random fields: the stochastic partial differential equation approach.
\newblock \emph{Journal of the Royal Statistical Society}, 73\penalty0
  (4):\penalty0 423--498, 2011.
\newblock \doi{10.1111/j.1467-9868.2011.00777.x}.

\bibitem[Lord et~al.(2014)Lord, Powell, and Shardlow]{LordEtAl2014}
G.~Lord, C.~Powell, and T.~Shardlow.
\newblock \emph{Introduction to Computational Stochastic PDEs}.
\newblock Cambridge University Press, 2014.
\newblock \doi{10.1017/CBO9781139017329}.

\bibitem[Matthies and Bucher(1999)]{MatthiesBucher1999}
H.~Matthies and C.~Bucher.
\newblock Finite elements for stochastic medie problems.
\newblock \emph{Computer Methods in Applied Mechanics and Engineering},
  168:\penalty0 3--17, 1999.
\newblock \doi{10.1016/S0045-7825(98)00100-5}.

\bibitem[Matthies and Keese(2005)]{MatthiesKeese2005}
H.~Matthies and A.~Keese.
\newblock Galerkin methods for linear and nonlinear elliptic stochastic partial
  differential equations.
\newblock \emph{Computer Methods in Applied Mechanics and Engineering},
  194:\penalty0 1295--1331, 2005.
\newblock \doi{10.1016/j.cma.2004.05.027}.

\bibitem[McLean(2000)]{McLean2000}
W.~McLean.
\newblock \emph{Strongly elliptic systems and boundary integral equations}.
\newblock Cambridge University Press, Cambridge, 2000.
\newblock ISBN 0-521-66332-6; 0-521-66375-X.

\bibitem[Meise and Vogt(1997)]{MeVo}
R.~Meise and D.~Vogt.
\newblock \emph{Introduction to functional analysis}, volume~2 of \emph{Oxford
  Graduate Texts in Mathematics}.
\newblock The Clarendon Press, Oxford University Press, New York, 1997.
\newblock ISBN 0-19-851485-9.
\newblock Translated from the German by M. S. Ramanujan and revised by the
  authors.

\bibitem[Minlos(1959)]{minlos59}
R.~Minlos.
\newblock Generalized random processes and their extension in measure.
\newblock \emph{Tr. Mosk. Mat. Obs.}, 8:\penalty0 497--518, 1959.

\bibitem[Mugler and Starkloff(2013)]{MuglerStarkloff2013}
A.~Mugler and H.-J. Starkloff.
\newblock On the convergence of the stochastic {G}alerkin method for random
  elliptic partial differential equations.
\newblock \emph{ESAIM: Mathematical Modelling and Numerical Analysis},
  47:\penalty0 1237--1263, 2013.
\newblock \doi{10.1051/m2an/2013066}.

\bibitem[Nobile et~al.(2016)Nobile, Tamellini, and Tempone]{NobileEtAl2016}
F.~Nobile, L.~Tamellini, and R.~Tempone.
\newblock Convergence of quasi-optimal sparse-grid approximation of
  {H}ilbert-space-valued functions: application to random elliptic {PDE}s.
\newblock \emph{Numerische Mathematik}, 134:\penalty0 343--388, 2016.
\newblock \doi{10.1007/s00211-015-0773-y}.

\bibitem[Pr{\'e}v\^{o}t and R{\"o}ckner(2007)]{PrevotRockner2007}
C.~Pr{\'e}v\^{o}t and M.~R{\"o}ckner.
\newblock \emph{A Concise Course on Stochastic Partial Equations}, volume 1905
  of \emph{Lecture Notes in Mathematics}.
\newblock Springer-Verlag, Berlin-Heidelberg, 2007.
\newblock \doi{10.1007/978-3-540-70781-3}.

\bibitem[Reed and Simon(1978)]{ReSi}
M.~Reed and B.~Simon.
\newblock \emph{Methods of Modern Mathematical Physics {IV}: Analysis of
  Operators}.
\newblock Academic Press [Harcourt Brace Jovanovich, Publishers], New
  York-London, 1978.
\newblock ISBN 0-12-585004-2.

\bibitem[Roman and Sarkis(2006)]{RomanSarkis2006}
L.~Roman and M.~Sarkis.
\newblock Stochastic {G}alerkin method for elliptic {SPDE}s: A white noise
  approach.
\newblock \emph{Discrete and Continuous Dynamical Systems--Series B},
  6\penalty0 (4):\penalty0 941--955, 2006.
\newblock \doi{10.3934/dcdsb.2006.6.941}.

\bibitem[Rudin(1990)]{Rudin1962}
W.~Rudin.
\newblock \emph{Fourier analysis on groups}.
\newblock Wiley Classics Library. John Wiley \& Sons, Inc., New York, 1990.
\newblock ISBN 0-471-52364-X.
\newblock \doi{10.1002/9781118165621}.
\newblock URL \url{https://doi.org/10.1002/9781118165621}.
\newblock Reprint of the 1962 original, A Wiley-Interscience Publication.

\bibitem[Sato(2013)]{Sato13}
K.~Sato.
\newblock \emph{L\'{e}vy Processes and Infinitely Divisible Distributions},
  volume~68 of \emph{Cambridge Studies in Advanced Mathematics}.
\newblock Cambridge University Press, Cambridge, 2013.
\newblock ISBN 978-1-107-65649-9.
\newblock Translated from the 1990 Japanese original, Revised edition of the
  1999 English translation.

\bibitem[Schwartz(1954)]{Schwartz1954}
L.~Schwartz.
\newblock Espaces de fonctions diff{\'e}rentiables {\`a} valeurs vectorielles.
\newblock \emph{Journal d'Analyse Math{\'e}matique}, 4:\penalty0 88--148, 1954.
\newblock \doi{10.1007/BF02787718}.

\bibitem[Talagrand(1994)]{Tal94}
M.~Talagrand.
\newblock Sharper bounds for {G}aussian and empirical processes.
\newblock \emph{The Annals of Probability}, 22\penalty0 (1):\penalty0 175--226,
  1994.
\newblock \doi{10.1214/aop/1176988847}.

\bibitem[Talagrand(2014)]{talagrand2014upper}
M.~Talagrand.
\newblock \emph{Upper and Lower Bounds for Stochastic Processes: Modern Methods
  and Classical Problems}, volume~60.
\newblock Springer Science \& Business Media, 2014.
\newblock \doi{10.1007/978-3-642-54075-2}.

\bibitem[Theting(2000)]{Theting2000}
T.~Theting.
\newblock Solving {W}ick-stochastic boundary value problems using a finite
  element method.
\newblock \emph{Stochastics and Stochastics Reports}, 70:\penalty0 241--270,
  2000.
\newblock \doi{10.1080/17442500008834254}.

\bibitem[Torquato(2013)]{torquato2013random}
S.~Torquato.
\newblock \emph{Random Heterogeneous Materials}.
\newblock Springer, 2013.
\newblock \doi{10.1007/978-1-4757-6355-3}.

\bibitem[Triebel(1992)]{Triebel1992}
H.~Triebel.
\newblock \emph{Higher analysis}.
\newblock Hochschulb\"{u}cher f\"{u}r Mathematik. [University Books for
  Mathematics]. Johann Ambrosius Barth Verlag GmbH, Leipzig, 1992.
\newblock ISBN 3-335-00321-7.
\newblock Translated from the German by Bernhardt Simon [Bernhard Simon] and
  revised by the author.

\bibitem[Vakhania et~al.(1987)Vakhania, Tarieladze, and Chobanyan]{VaTaCh}
N.~Vakhania, V.~Tarieladze, and S.~Chobanyan.
\newblock \emph{Probability distributions on {B}anach spaces}, volume~14 of
  \emph{Mathematics and its Applications (Soviet Series)}.
\newblock D. Reidel Publishing Co., Dordrecht, 1987.
\newblock ISBN 90-277-2496-2.
\newblock \doi{10.1007/978-94-009-3873-1}.
\newblock URL \url{https://doi.org/10.1007/978-94-009-3873-1}.
\newblock Translated from the Russian and with a preface by Wojbor A.
  Woyczynski.

\bibitem[van Handel(2016)]{vanhandel2014probability}
R.~van Handel.
\newblock {Probability in High Dimension}.
\newblock {APC 550 Lecture Notes}, Princeton University, 2016.

\bibitem[Walsh(1986)]{Walsh1986}
J.~Walsh.
\newblock An introduction to stochastic differential equations.
\newblock In P.~L. Hennequin, editor, \emph{{\'E}cole d'{\'E}t{\'e} de
  Probabilit{\'e}s de Saint Flour XIV - 1984}, volume 1180 of \emph{Lecture
  Notes in Mathematics}, pages 265--439. Springer, 1986.
\newblock \doi{10.1007/BFb0074920}.

\bibitem[Wolpert and Ickstadt(1998)]{wolpert1998poisson}
R.~L. Wolpert and K.~Ickstadt.
\newblock Poisson/gamma random field models for spatial statistics.
\newblock \emph{Biometrika}, 85\penalty0 (2):\penalty0 251--267, 1998.
\newblock URL \url{https://www.jstor.org/stable/2337356}.

\end{thebibliography}

\end{document}